\documentclass[reqno]{amsart}
\usepackage{epsfig,latexsym,amsfonts,amssymb,amsmath,amscd}
\usepackage{amsthm}
\usepackage[mathscr]{eucal}
\usepackage{mathrsfs}
\usepackage{mathbbol}
\usepackage{oldgerm,units,color}

\usepackage{pst-node}
\usepackage{graphicx}

\usepackage{eepic, epic}

\usepackage{ifthen}

\def\sgn{{\operatorname{sgn}}}

\def\ad{{\operatorname{ad}}}
\def\adL{{\operatorname{ad}}\,L}

\def\Null{{\operatorname{Null}}}

\newcommand{\Det}[1]{ |{#1}|}
\usepackage{tikz}
\usepackage{epsfig,latexsym,amsfonts,amssymb,amsmath,amscd,graphics,epic}
\usepackage{amsfonts,amssymb,amsmath,amscd,amsthm}
\usepackage[mathscr]{eucal}
\usepackage{mathrsfs}

\usepackage{mathbbol}

\usepackage{oldgerm,units}
\usepackage{wrapfig,epsfig}

\usepackage{ifthen}
\usepackage{mathbbol}

\usepackage{amsthm}
\usepackage[mathscr]{eucal}
\usepackage{mathrsfs}
\usepackage{mathbbol}
\usepackage{oldgerm,units}
\usepackage{wrapfig}

\newtheorem{theorem}{Theorem}[section]

\newtheorem{definition}[theorem]{Definition}

\newtheorem{example}[theorem]{Example}

\newcommand{\Real}{\mathbb R}

\newcommand{\Net}{\mathbb N}

\newcommand{\one}{\mathbb{1}}
\newcommand{\zero}{\mathbb{0}}

\newcommand{\trop}[1]{\mathcal{#1}}

\newcommand{\tG}{\trop{G}}

\newcommand{\tM}{\trop{M}}

\newcommand{\tT}{\trop{T}}

\newcommand{\Hom}{Hom}

\newcommand{\sig}{\sigma}

    \ifx\proof\undefined
    \newenvironment{proof}{
    \smallskip
    \noindent\emph{Proof.}}{\hfill\(\Box\)
    \bigskip
    } \fi

\newcommand{\ifdef}[3]{\ifthenelse{\equal{#1}{true}}{#2}{#3}}

\definecolor{lgray}{gray}{0.90}

\def\vep{\varepsilon}

\def\({\left(}
\def\){\right)}

\def\Z{{\mathbb Z}}
\def\Q{{\mathbb Q}}

\def\End{{\operatorname{End}}}

\def\pipe{{\underset{{\ \, }}{\mid}}}

\def\vsemifield0{$\nu$-semifield$^\dagger$}
\def\vsemiring0{$\nu$-semiring$^\dagger$}

\def\lmodg{\mathrel  \pipeGS \joinrel \joinrel \joinrel =}

\def\CFunFF1{\operatorname{CFun} (F,F)}

\def\semiring0{semiring$^{\dagger}$}
\def\system0{system$^{\dagger}$}
\def\systems0{systems$^{\dagger}$}
\def\Systems0{Systems$^{\dagger}$}
\def\Semiring0{Semiring$^{\dagger}$}
\def\Semirings0{Semirings$^{\dagger}$}
\def\semidomain0{semidomain$^{\dagger}$}
\def\semifield0{semifield$^{\dagger}$}
\def\semifields0{semifields$^{\dagger}$}
\def\vsemifields0{$\nu$-semifields$^{\dagger}$}
\def\domain0{domain$^{\dagger}$}
\def\predomain0{pre-domain$^{\dagger}$}
\def\predomains0{pre-domains$^{\dagger}$}
\def\domains0{domains$^{\dagger}$}
\def\vdomains0{$\nu$-domains$^{\dagger}$}

\def\tTz{\tT_{\zero}}
\def\Fun{\operatorname{Fun}}

\def\domains0{domains$^\dagger$}

\newcommand{\etype}[1]{\renewcommand{\labelenumi}{(#1{enumi})}}
\def\eroman{\etype{\roman}}

\def\pipe{{\underset{{\tG}}{\mid}}}

\def\pipe{{\underset{{\tG}}{\mid}}}

\def\lmodg{\mathrel  \pipe \joinrel \joinrel =}

\def\sig{\sigma}

\def\invr{{^{-1}}}

\def\tra{\operatorname{tr}}

\def\a{\alpha}
\newtheorem{thm}[theorem]{Theorem}

\newtheorem*{thm*}{Theorem}

\newtheorem{cor}[theorem]{Corollary}

\def\Hom{\operatorname{Hom}}
\newtheorem{lem}[theorem]{Lemma}
\newtheorem{rem}[theorem]{Remark}
\newtheorem{prop*}{Proposition}

\newtheorem{prop}[theorem]{Proposition}
\newtheorem{defn}[theorem]{Definition}
\newtheorem*{examp*}{Example}
\newtheorem*{examples*}{Examples}
\newtheorem*{remark*}{Remark}
\newtheorem{Note}[theorem]{Note}
\newtheorem{MNote}[theorem]{MAJOR NOTE}
\newtheorem*{defn*}{Definition}

\def\R{\Real}
\def\la{\lambda}

\def\tT{\mathcal T}
\def\tTz{\tT_\zero}

\numberwithin{equation}{section}

\def\M0{M_{\zero}}

\def\SR{R}

\def\tGz{\mathcal G_\zero}

\def\PS{P}

\def\Cong{\Phi}

\def\rone{{\one_\SR}}

\def\semirings0{semirings$^\dagger$}

\newcommand{\nPS}[1]{\PS_{(!#1)}}
\newcommand{\nPSo}[1]{\nPS{\one}}

\newcommand{\absl}[1]{|{#1}|}
\newcommand{\adj}[1]{\operatorname{adj}({#1})}

\newtheorem*{nothma}{\textbf{Theorem B}}
\newtheorem*{nothmaa}{\textbf{Theorem D}}
\newtheorem*{nothmb}{\textbf{Theorem E}}
\newtheorem*{nothmc}{\textbf{Theorem F}}

\newtheorem*{nothme}{\textbf{Theorem G}}
\newtheorem*{nothmf}{\textbf{Theorem C}}
\newtheorem*{nothmg}{\textbf{Proposition A}}
\newtheorem*{nothmh}{\textbf{Theorem I}}
\newtheorem*{nothmi}{\textbf{Theorem H}}
\newtheorem*{nothmj}{\textbf{Theorem J}}
\newtheorem*{nothmk}{\textbf{Theorem K}}

\newtheorem*{nothmm}{\textbf{Proposition L}}
\newtheorem*{nothmn}{\textbf{Theorem M}}

\begin{document}


\title[Algebras with a negation map]
{Algebras with a negation map}\footnote{This paper originally was
 entitled ``Symmetries in tropical algebra'' in
arXiv:1602.00353 [math.RA]}


\author[Louis Rowen]{Louis Halle Rowen}
\address{Department of Mathematics, Bar-Ilan University, Ramat-Gan 52900,
Israel} \email{rowen@math.biu.ac.il}


\date{\today}



\thanks{The author would like to thank the following researchers for helpful conversations:
Oliver~Lorscheid and Zur Izhakian for discussions on
tropicalization, Marianne Akian and Stephane Gaubert, together with
Adi Niv, for discussions on symmetrization and for interesting
examples, Max Knebusch for insights on ub semirings, Matt Baker for
explaining hyperfields, and especially the role of the ``sign
hyperfield,'' and Sergio Lopez for helpful comments on the first
arXiv version. Special thanks are due to Guy Shachar for a careful
reading of the first arXiv version which led to several corrections,
 to J.~Jun for finding several inaccuracies in various versions,
pointing out the  connection to fuzzy rings, Henry's paper, and
other interesting papers with which he has been involved, to
Marianne Akian for pointing out how various theorems should be
clarified, and especially to the referee, who provided a great many
improvements.}

\thanks{The author's research was supported by Israel Science Foundation grants No.
1207/12 and 1994/20. The author would also like to thank the
University of Virginia for its support during the initial
preparation of this work in 2015.}


\begin{abstract}
{In tropical mathematics, as well as other mathematical theories
involving semirings,  one often is challenged by the lack of
negation when trying to formulate the tropical versions of classical
algebraic concepts for which the negative is a crucial ingredient.
Following   ideas originating in work of Dress, Gaubert, and the
Max-Plus group and pursued further by Akian, Gaubert, and Guterman,
we study algebraic structures with negation maps, called
\textbf{systems}, showing how these unify the more viable
(super)tropical versions, as well as symmetrization, hypergroup
theory and fuzzy rings, thereby helping to explain similarities in
these theories. Special attention is paid to \textbf{metatangible}
 systems, whose algebraic theory includes all the main
examples, and is rich enough to facilitate computations and provide
a host of structural results. The systems studied in this paper are
``ground'' systems, insofar as they are the underlying structure
which can be studied via other ``systemic modules''.

By formalizing the structure, we can introduce morphisms. Morphisms
enable
 us to describe the tropicalization functor, as well as
  providing a link between classical algebraic
results and their tropical and hyperfield analogs.}

%
\end{abstract}

\maketitle

\tableofcontents




\section{Introduction}\label{conc0}

 This paper, a trimmed and rearranged version of  \cite{Row16v},
 which was
 in turn based on  \cite{Row16},
    was born from the desire to understand a mysterious parallel
 between structural
 results in what we will call the ``classical algebraic theory''
 and theorems formulated directly in varied aspects of tropical algebra,
 despite the former dealing with fields and the latter with the max-plus algebra
 and related semifields.  It is designed to lay the foundation for a unified algebraic
 theory,   which also
 encompasses diverse recent research in  hyperfields,  fuzzy rings, and especially
 tropical mathematics.

 Our objective in this project is three-fold: laying out the underlying foundations, providing applications
 to the literature, and then developing the structure with an emphasis on representation theory and geometry.
 This paper focuses on the first two parts; the third
  has involved several additional papers
 in the last two years, including \cite{Row17,AGR0,AGR,GaR,JMR,JMR1}. We need considerable
 background
to set up these foundations, so this paper has two introductions.
The first introduction (\S\ref{conc}) gives an idea of our goals,
but without full definitions, and the second elaborates on them more
fully in \S\ref{det}, together with an outline of the organization
in \S\ref{over1} and a list of main results in \S\ref{over27}.

The main ideas behind the theory are explained in a sequence of 17
notes interspersed throughout the paper, each labeled as a ``MAJOR
NOTE.''
%
%

\subsection{Short overview}\label{conc}$ $ \numberwithin{equation}{section}

  Tropicalization  originally was viewed as a process of
  taking logarithms and passing in the limiting case
  to the max-plus algebra, which is a semiring. Thus,
tropical algebra customarily has relied on the theory of semirings
 which goes back to Costa~\cite{Cos}
and Eilhauer \cite{Ei}, and for which we use Golan~\cite{golan92} as
our standard reference.

As the field of  Puiseux series  came into play, the underlying
semiring was viewed as the target of the ``Puiseux valuation,''
which differs somewhat from the max-plus algebra. Towards this end,
in \cite{zur05TropicalAlgebra, IKR, INR, IR1} a ``supertropical''
theory was initiated over a semiring  by means of a ``ghost map''
(where the negation map actually is the identity map), with various
applications to affine varieties, matrices, linear algebra, and
quadratic forms. Classical algebraic results were transferred to the
tropical theory by means of a somewhat mysterious ``surpassing''
relation on semirings, which satisfies many properties of equality,
and replaces equality in many generalizations of classical theorems,
especially for polynomials and matrix theory. Thus we are motivated
to ask how exactly this surpassing relation fits into the algebraic
theory.

Viro \cite{Vi} cast the supertropical theory and other
 mathematical ideas in terms of
 hyperfields, where the sum of two elements is a set, rather than a
 single element.

In all of these instances, there is a set $\tT$ of main interest
which we call the set of \textbf{tangible
 elements},
(e.g., the max-plus semialgebra but without $a+a = a$, the
symmetrization, or a
 hyperfield), together with a special operation resembling
negation (the switch map or hypernegation in the latter two cases).
Its intrinsic algebraic structure is not sufficient for a full
investigation, necessitating many  algebraic results to be
formulated and proved on an ad hoc basis. The situation is
 clarified significantly by tying $\tT$ to
 a $\tT$-module
$(\mathcal A,+)$  with a richer structure,  often via an embedding,
 studied in accordance with well-known techniques from
universal algebra. Due to the fact that $\tT$ is almost never closed
under addition,  much of its theory depends on understanding
$(\mathcal A,+)$, even though our ultimate interest lies in $\tT$.
In many instances, $\tT$ is a multiplicative monoid, although we
keep open the option that $\tT$ has different structure, possibly
nonassociative, for example with Lie multiplication.


To give an indication of where we are heading, let us start by
listing some classical examples.

\begin{example}$  $
 \begin{enumerate}
 \item $\mathcal A$ is an integral domain and $\tT = \mathcal A \setminus
 \{0\}$.
  \item  $\mathcal A$ is a  graded algebra,
 and $\tT$ is  the monoid   of homogeneous elements.
  \item $\mathcal A$ is a vector space with base  $\tT.$
   More specifically, $\mathcal A$ could be an algebra with a multiplicative
   base
  $\tT$.   For example, $\mathcal A$ could be the group algebra of
  a group $\tT$.
  \item $\mathcal A$ is a Hopf algebra and $\tT$ is a special subset
  (such as the group-like elements or primitive elements). For example, $\mathcal A$ could be  the enveloping
  algebra  of
  a Lie algebra $\tT$.
 \item $\mathcal A$ is the set of class functions from a finite
 group to a field $F$; $\tT$ is the irreducible characters.
  \item  $\tT$ is  the set of reduced words from an algebra $\mathcal
  A$ with a reduction procedure.
  \end{enumerate}
\end{example}
On the other hand, our main motivation is from tropical mathematics,
whose parent structure in tropical algebra is the well-known
\textbf{max-plus
 algebra}, described thoroughly in \cite{ABG}, and which can be
 viewed as a semiring, whose zero element (written $\zero$ here) is $-\infty$.
 But its algebraic structure is  inadequate.

The most basic trouble with the max-plus algebra is its lack of
negation, especially for studying the determinant of matrices. Other
shortcomings of the max-plus algebra  include failure of
 unique factorization in polynomials, conflicting definitions in
 linear algebra, and lack of a multiplicative determinant for matrices.
 Various researchers, going back to Kuntzman~\cite{Ku} in 1972,
tackled this last issue, with decisive results
 \cite[p.352, end of proof of (a)]{RS}. As noted earlier,
Gaubert \cite{Ga} offered a remedy in his dissertation, which was
continued together with the M.~Plus group and Akian and Guterman,
 using a ``symmetry''
 (\cite{Pl}, \cite{Ga}, \cite[\S 3.4]{BCOQ}, \cite{GaP},
\cite{AGG1}, \cite{Hen},  \cite{AGG2}, and \cite[Appendix~A]{Ju}),
 leading to a general ``transfer principle'' to generate semiring
identities. Here we offer a general ``negation map.''

Other predecessors of this study include the following, where $\tT$
is as in this paper:
 Dress \cite{Dr} introduced fuzzy rings to study
valuated matroids (where $\tT$ is a multiplicative subgroup).
Gaubert \cite{Ga} doubled $\tT$, using  ``symmetry,'' to provide a
viable negation.
 Lorscheid \cite{Lor} introduced the rather general framework
of~``blueprints'' in terms of semigroup rings.

Our goal
 here is to
 simplify and
 unify these approaches with  an axiomatic theory which
  is robust enough to provide features lacking in the max-plus algebra.

 \begin{MNote} The negation map $(-)$ is the key to our approach,
 pervading all of the concepts. When  $\tT \subset \mathcal A,$ for every element $a \in \tT$ there
 is a unique element~$(-)a$ for which $a + (-)a$ is a special element of~$\mathcal A$,
 called a ``quasi-zero'' since it replaces the classical element $\zero$. When $\one \in \tT$, then $(-)b   =
 ((-)\one)b$ for all $a \in \mathcal A$.

 The other main notion is a ``surpassing relation,''
 to replace equality in most theorems.
Altogether, our structure of choice, a \textbf{system}
(Definition~\ref{system0}),  is a quadruple $(\mathcal A, \tT, (-),
\preceq),$ where $ \mathcal A$ is a $ \tT$-module, $(-)$ is a
negation map, and $\preceq$ is a
 surpassing relation.\end{MNote}

Although systems were motivated by tropical algebra, our purpose
here is to show that they  are justified in having a theory of their
own.
 Systems include the classical case, as well as two key
applications: the ``standard'' supertropical
 semiring  given in Definition~\ref{super1},    and the symmetrized semiring
(Remark~\ref{sym07} and
 Example~\ref{surpsym}).   Other applications include
 hypergroups (\S\ref{hyperr}, Theorem~\ref{hypersys}), fuzzy rings (\S\ref{fuzz1}), matrices (\S\ref{linalg1}, Theorem~\ref{prodform3}),
tensor semialgebras (\S\ref{tenpro}), polynomials (\S\ref{poly}),
Grassmann semialgebras (Example~\ref{varex1}),
 the ``layered'' semiring of \cite{IKR0} (\S\ref{ELTalg0}), bilinear and quadratic forms  (Example~\ref{varex000}),
 and
 Poisson semialgebras
(\S\ref{Exsmore}). Nonassociative examples include Lie semialgebras
 given in \S\ref{Lie1}. Even when multiplication is associative,
  one might need to relax distributivity, to accommodate application to
hyperfields; distributivity of $\mathcal A$ over elements of $\tT$
is enough to run the theory.

Once we have an organized structure, we talk of the appropriate
category and its morphisms in~\ref{surp3}. Some of the traditional
maps, such as tropicalization, fit into this context, and the Lie
context in Lemma~\ref{Lies} also fits in. But this set-up still
requires  more hypotheses to obtain strong results.
  \begin{MNote}  In
all of our applications, we have ``unique quasi-negation'' as well
as the property, called being \textbf{metatangible}, which says that
for $a,a' \in \tT$, $a'+a \in \tT$ iff $a' \ne (-)a.$  Most of  this
paper explores the ramifications of metatangibility, which
encompasses many areas related to tropical mathematics.
 \end{MNote}

\subsection{Basic definitions and ongoing notation}\label{height}$ $

As customary, $\Net$ denotes the positive natural numbers,  $\Net_0$
denotes $\Net\cup \{0\}$, $\Q$ denotes the rational numbers, and
$\Real$ denotes the real numbers, all ordered with respect to
addition.

Recall that a \textbf{monoid} is a semigroup with a two-sided
identity element, denoted as $\zero$ for addition, and as~$\one$ for
multiplication. We customarily write $ab$ for $a \cdot b$. The
\textbf{power set} of a set $S$ is denoted $\mathcal P(S)$.



\begin{defn}\label{bipot1}
A \textbf{\semiring0} is a semiring $(\mathcal A, +, \cdot, \rone )$
without $\zero ,$ i.e., an additive Abelian semigroup $(\mathcal A,
+)$ and multiplicative monoid $(\mathcal A, \cdot, \rone )$
satisfying the usual distributive laws. \footnote{The reason that we
do not always require~$\mathcal A$
 to have the element $\zero $ is that $\zero $ just
distracts from our true goal,  negation maps and quasi-zeros, not
negatives.}

 A \semiring0 $(\mathcal A,+ \, ,\cdot \;,\one)$ is a
 \textbf{\semifield0} if $(\mathcal A,\cdot)$ is an Abelian group.
\end{defn}

\begin{defn}\label{chark}  
A semigroup $(\mathcal A,+)$  has \textbf{characteristic~$k>0$} if
$( k+1)a  =a  $ for all $a \in \mathcal A,$ with $k \ge 1$ minimal.
$\mathcal A$ has \textbf{characteristic
 $0$} if $\mathcal A$ does not have  characteristic
 ~$k$ for any
 $k\ge 1.$
\end{defn}

When $\one \in \tT\subseteq \mathcal A$ it is enough to check that
$( k+1)\one  =\one.  $  For example,   the max-plus algebra has
``characteristic 1.''
 This leads to the notion of
``$\mathbf F_1$ geometry.''

\begin{defn}
A \textbf{pre-order} is a transitive and reflexive relation. A
\textbf{partial order} (PO) is an antisymmetric pre-order. An
\textbf{order} is a total PO.

 A \textbf{(pre-)ordered monoid}
is a monoid~$\tM $ with a (pre-)order satisfying
\begin{equation}\label{ogr1} a \le b \quad \text{implies}\quad ca
\le cb,\ ac \le bc, \quad \forall  a,b,c \in \tM.\end{equation}
\end{defn}

\subsubsection{Bimagmas}$ $

 We want to maintain the ability to handle the Lie case. Even
when lacking associativity and/or distributivity, our algebraic
structures will have addition and multiplication of various sorts,
so we should weaken the notion of semiring. A \textbf{magma},
sometimes called groupoid, is a set with an operation having no
prescribed identities. Towards this end (and not knowing of
pre-existing terminology) we call  a semigroup $(\mathcal A,+)$ a
\textbf{bimagma} if it has multiplication as well as addition. A
\textbf{homomorphism} of bimagmas is a map preserving
 addition and multiplication. A bimagma is
\textbf{unital} if $(\mathcal A, \cdot)$ possesses a multiplicative
unit~$\one_\mathcal A $. Although the algebraic theory can be
carried out without commutativity, many applications are
commutative, so we often work with a commutative bimagma, denoted
\textbf{cbimagma}.

\begin{Note}  The
reader who feels uncomfortable with cbimagmas can read instead
``commutative semiring,'' which provides the main tropical setting.
\end{Note}

 An \textbf{ideal} of a cbimagma
$\mathcal A$ is an additive sub-semigroup $I$ such
 that $aI \subseteq I$ for all $a \in \mathcal A.$
 $\{ \zero  \}$ is an ideal when  $ \zero  \in \mathcal A$, but  it will be
replaced  by other ideals arising naturally in the structure.

\begin{defn}\label{barn}
 For any unital bimagma $\mathcal A$, we want the image of $\Net$ inside $\mathcal A$. Towards this end, we define  $\mathbf 1   = \one_{\mathcal A},$ and inductively $\mathbf {n+
1} =  \mathbf n + \one_{\mathcal A}$, and $\mathbf N(\mathcal A)$ to
be $ \{ \mathbf n : n \in \Net\}\subseteq \mathcal A.$  When
$\mathcal A$ is understood we write $\mathbf N$ for $\mathbf
N(\mathcal A)$.
\end{defn}

When $\mathcal A$ is a bimagma we still require the following
special case of distributivity: $$\mathbf m a = {m}a, \quad \forall
m \in \Net,\ a \in \mathcal A.$$  One tricky point  is that we may
not be able to identify $n$ with
 $\mathbf n$, even for semirings of characteristic~0, for example with
truncated numbers (Example \ref{nontang}(vii) below). Nevertheless, we have:

\begin{rem}\label{hom1}  There is a homomorphism $\Net \to \mathcal A$ given by $n \mapsto
\mathbf n$. 
 \end{rem}

\subsubsection{Connection to universal algebra and model theory}\label{univalg}$ $

 In many of our examples, $\tT$ is a multiplicative monoid, even a
group, and this structure is to be viewed intrinsically, as is the
negation map. Other properties (such as nonassociative Lie
multiplication, or nondistributive operations for hyperfields) also
can come into play. Furthermore, one might want to consider the more
esoteric bimagmas. Because of the varied algebraic structures
involved, the appropriate setting for the investigation would be
 universal algebra, for which
\cite{Jac1980} serves as our reference (also cf.~\cite{Coh,Berg}), but to
avoid digressions we leave that aspect of the
 theory to the reader's discretion, and make do with the terminology ``algebraic structure''
 instead of the more technical ``$\Omega$-algebra'' of
 \cite[Definition~2.1]{Jac1980}. Roughly speaking, we deal with
 given algebraic structures, each equipped with given $n$-ary operations for $n=0,1,2,$
 which satisfy given \textbf{identities}, which
are simple universal sentences (i.e., of the form $f(x_1, \dots,
x_m) = g(x_1, \dots, x_n)).$ A \textbf{homomorphism}
  by definition  preserves all of the operations
in the structure.

%

We recall from \cite [p.~61]{Jac1980} that a \textbf{congruence} of
an algebraic structure $\mathcal A$ is an equivalence relation
$\Cong$ on~$\mathcal A$
 preserving the operations of $\mathcal A$, or in other
words which is a subalgebra of $\mathcal A \times \mathcal A$.
Ideals do not work well in the absence of negatives, but fortunately
congruences serve as a replacement (although technically more
complicated). We also denote   $\Cong$ as~$\equiv$, or,
equivalently, as $\{ (\mathbf a,\mathbf a'): \mathbf a \equiv
\mathbf a' \} \subseteq \mathcal A\times \mathcal A.$  \cite
[p.~62]{Jac1980} ~provides the Noether correspondence
 between homomorphic images and congruences:
any
  homomorphism $\psi: \mathcal A \to \mathcal B$ gives
rise to a congruence $\Cong_\psi$ on $\mathcal A$ given by $(\mathbf
 a,\mathbf {a'}) \in \Cong_\psi $ iff $ \psi(\mathbf a) = \psi(\mathbf {a'})$; conversely, any
congruence~$\Cong$ gives rise to the corresponding structure
$\mathcal A/\Cong$ on the equivalence classes, together with a
natural
 homomorphism $\psi: \mathcal A \to \mathcal A/\Cong$ defined by $\mathbf
a \mapsto [\mathbf a].$

%

\section{A more detailed overview}\label{det}

We provide more details on the main concepts:
$\tT$-modules, negation maps, and surpassing relations.

\subsection{$\tT$-modules}

\begin{defn}\label{modu13}
A (left) $\tT$-\textbf{module}   over a set $\tT$ is an additive
monoid $(\mathcal A,+,\zero _\mathcal A)$ together with scalar
multiplication $\tT\times \mathcal A \to \mathcal A$ satisfying
distributivity over     $\tT$  in the sense that
$$a(b_1+b_2) = ab_1 +ab_2,\quad  \forall a \in \tT,\ b _i \in \mathcal
A,$$ also stipulating that $a\zero_{\mathcal A} = \zero_{\mathcal
A}$ for all $a$ in $\tT$ \footnote{Modules $\mathcal A$ over a
semiring $\tT$
  are called \textbf{semimodules} in the literature, where one
 stipulates that $\zero _\tT b = \zero_{\mathcal A}$, $\forall b \in \mathcal A.$}.

In this case, we also say that the set
 $\tT$   \textbf{acts} on  $(\mathcal A, +)$.
A $\tT$-module $\mathcal A$ is
\textbf{cancelative}   if   $a b  = a b'$ for~$a$ in $\tT$ implies $  b =    b'$.
\end{defn}

Our relevant   structures all are $\tT$-modules. Often $\tT$ itself
has extra structure,   passed on to~$\mathcal A.$

\begin{defn} When $ (\tT, \cdot, \one)$ is a monoid,
we say there is a \textbf{monoid action} 
 when $\one_\tT b =b$ and $(a_1a_2)b = a_1 (a_2 b)$, for all
$a_i \in \tT$ and $b\in \mathcal A$, and call $\mathcal A$ a
 \textbf{monoid module}. 
\end{defn}

Cancelation is clear for a monoid action over a group. 

%
%
%

%
%
%
%

\begin{example}\label{triv} We say that $\tT$ is  \textbf{trivial}
if $|\tT| = 1.$  Here are some instances:
\begin{itemize}
 \eroman
 \item $\mathcal A = \tT = \{ \zero \}.$
  \item (The Boolean semifield) $\mathcal A = \{ \zero ,\one
  \}$ and $ \tT = \{  \one
  \}$
  where $\one+\one = \one^2 = \one .$
   \item (The Boolean ``supertropical semifield'') $\mathcal A = \{ \zero ,\, \one,\ \one^\nu := \one+\one  \},$
  where $\tT = \{ \one\}$, $\one^2 = \one$,   and $ \one^\nu +\one = \one^\nu.$
\end{itemize}
\end{example}

%
%
%

\begin{defn}\label{presem}
 A \textbf{$\tT$-bimagma} (resp.~\textbf{$\tT$-cbimagma}) is
a bimagma (resp.~cbimagma)  $\mathcal A$ which is a unital
$\tT$-module satisfying
$$a(b_1 b_2) = (ab_1)b_2,\quad \forall a \in \tT,\ b _i \in \mathcal
A. \footnote{We do not require distributivity over all of $\mathcal
A$ a priori, which could
fail for hypergroups, one of our motivating examples.}$$ 

 A $\tT$-\textbf{\semiring0} is a \semiring0 which
is a $\tT$-bimagma.

A $\tT$-\textbf{\semifield0} is a commutative $\tT$-semiring for which
$\tT$ is a group. (This is called a \textbf{demifield}
in~\cite[Definition~4.1]{Bak} when $\tT$ is a hyperfield generating
$(\mathcal A,+)$.)

 A \textbf{semialgebra} over a commutative
(associative) \semiring0~$C$ is a $C$-bimagma (i.e., $\tT = C)$
which also distributes over $C$.
\end{defn}


\begin{Note}[Ongoing hypotheses and notation]     $(\mathcal A, +)$ is a $\tT$-module, whose zero element
$\zero $ is \textbf{absorbing} in the sense that $a \zero = \zero$
for all $a \in \tT$. To simplify the exposition we assume that $\tT
\subseteq \mathcal A$. (See \cite{AGR} for an elaboration of this
point.) We normally designate elements of $\tT$ as ``$a$'', and
elements of $\mathcal A$ as ``$b$.'' For simplicity, we assume $\tT
\subseteq \mathcal A$, and write $ \tTz$ for $ \tT \cup \{ \zero
\}$.
   Note that $ \{\zero \}\cup  \tT \cup ( \tT +  \tT)
=  \tTz +  \tTz .$  \end{Note}

\begin{defn}\label{modu-order} A \textbf{(pre-)ordered $\tT$-module} is a
$\tT$-module  $\mathcal A$ with a (pre-)order relation $\le$,
satisfying the following, for $a, a'\in \tT$ and $b,b_i,b_i' \in
\mathcal A$:
\begin{enumerate}
 \eroman
\item If $b_i \le b_i'$ for $i= 1,2$, then  $b_1 + b_2 \le b_1'
   + b_2'.$
    \item   If  $b_1 \le b_1'$ then $a b_1 \le ab_1'.$
       \item  (When $\tT$ is pre-ordered) If  $a \le a'$,  then $a b\le a' b.$

\end{enumerate}
\end{defn}
%
%
%


\subsubsection{Height}\label{height10}$ $

 \begin{defn}\label{height1} When 
 $\tT $  additively generates
$\mathcal A$, we define the \textbf{height} of an element $c \in
\mathcal A$ as the minimal number ~ $m_c$ such that $c = \sum
_{i=1}^{m_c} { a_i}$ with each $a_i \in \tT.$ (We say that $\zero$
has height~ 0.) The \textbf{height} of~$\mathcal A$ is the maximal
height of its elements (which is said to be~$\infty$ if these
heights are not bounded). \end{defn}

Thus $\mathcal A$ has height~1 iff $\mathcal A = \tT $ or $\tT
_{\zero}.$ $\mathcal A$ has height~2 iff $\mathcal A = \tT \cup (\tT
+ \tT )$ or $ \tT _{\zero} + \tT _{\zero} $, which also will play an
important role. The tropical theory falls largely into height 2.
Height 3 involves extra subtlety, such as various hyperfields and
``quasi-periodicity'' as indicated for example in
 Theorem~\ref{matrsys} and Example~\ref{Basicexamples}.
%

\subsection{Negation maps}\label{neg1}$ $

Since semigroups lack negation, we introduce a \textbf{negation map}:

\begin{defn}\label{negmap}
A \textbf{negation map} on a $\tT$-module $(\mathcal A,+)$ is an
 additive homomorphism
$(-) :\mathcal A \to \mathcal A$ of order $\le 2,$  written
$b\mapsto (-)b$, restricting to a  map $(-) :\tT\to \tT,$ compatible
 in the sense that \begin{equation}\label{neg}((-)a)b = a((-)b) =  (-)(ab), \ \forall a \in \tT,  \, b   \in \mathcal A.\end{equation}

A \textbf{negation map} on a $\tT$-bimagma $(\mathcal A,\cdot,+)$ is
 a negation map also
 satisfying~\eqref{neg}   for all $a,b \in \mathcal A$.
\end{defn}

\begin{rem}\label{makeneg} As Gaubert and Knebusch both observed, when $\tT$ contains a multiplicative
unit $\one_\tT$, the negation map on $\mathcal A$ is given simply by
$b \mapsto ((-)\one_\tT) b.$ But one has to be careful in choosing
$(-)\one $  when it is not specified a priori, cf.~\cite[Lemma~2.18]{Row16v}.
\end{rem}

Thus a negation map is a formal map $a \mapsto (-)a$, viewed as part
of our structure, that satisfies all of the properties of negation
{\it except} $a+((-)a) = \zero$, which comes automatically for
classical algebra.

Initially, negation is notably absent in tropical situations, but is
  circumvented in two main ways: the identity itself is a negation
map, paving the way to the ``supertropical theory,'' or else one can
introduce a negation map through the process of ``symmetrization''
(cf.~\S\ref{symm10} and \cite{AGG2}), which passes to $\mathcal A
\times \mathcal A.$

  To simplify notation, we write
$b(-)b' $ for $b+ ((-)b')$. Then we put $b^\circ : = b (-)b,$ called a
\textbf{quasi-zero}.

 To avoid ambiguity, we then  write the product
of $b$ and $(-)b'$ as $b((-)b')$, which occurs much more
rarely. Also we write $(\pm) b$ for ``$ b $~or $(-)b,$'' and
$b(\pm) b'$ for ``$ b+b' $ or  $b(-)b'.$''

We write $\mathcal A^\circ $ for $\{ b^\circ: b \in \mathcal A\},$ a
$\tT$-submodule of $(\mathcal A,+)$, and $\tT^\circ$ for $\{ a^\circ
: a \in \tT\}.$
 In the supertropical theory (Definition~\ref{super1} below), the
 quasi-zeros are the ``ghost'' elements. Under symmetrization (Definition~\ref{sym00}) the
 quasi-zeros have the form $(b,b).$ (In \cite[Definition~2.6]{AGG2} the quasi-zeros
are called ``balanced elements.'')

 \begin{MNote} The quasi-zero  takes the role
customarily assigned to the zero element in classical algebra, where
the only
 quasi-zero is~$\zero$ itself. $\mathcal A^\circ $ plays a fundamental role.
  \end{MNote}

Expressed in these terms, one of the challenges in  tropical
structure theories
 has been to describe~${\mathcal A}^\circ$ accurately.

\begin{lem}\label{minuszero}
$(-)\zero = \zero$ and $ \zero^\circ = \zero.$
\end{lem}\begin{proof} $(-)\zero = (-)\zero + \zero = (-)\zero + ((-)(-)\zero) = (-)(\zero
(-)\zero) = (-)( (-)\zero) = \zero.$ Hence $ \zero^\circ = \zero+
 \zero = \zero.$
\end{proof}

Note that $b^\circ = ((-)b)^\circ ,$ and  $ (mb)^\circ = m b^\circ$
for all $m \in \Net.$

\begin{defn}\label{impelt1}
We designate several important elements  of a unital $\tT$-module
$\mathcal A$,  for future reference:

 \begin{equation}\label{eeq} e = \one (-) \one, \quad e' = e + \one, \quad e^\circ = e (-) e = e+ e = 2e.\end{equation}
\end{defn}

 The most important quasi-zero  is $e,$ which acts similarly to
 $\zero$.  But $e$  need not
absorb in multiplication;
  rather
$ a e=  a(\one (-) \one) =  a  \one (-)  a  \one = ( a \one ) ^\circ $ for
any $a\in \tT.$


\begin{defn}\label{precedeq00}  A \textbf{quasi-negative} of $a \in \tT$ is an element $b\in \tT$
such that ${ a} +{ b}\in \mathcal A^\circ $. $\tT$ has
\textbf{unique quasi-negatives} if ${ a} + { b}\in \mathcal A^\circ$
for  $a,b \in \tT$ implies $b= (-)a $. In this situation we say
$\mathcal A$ and $(-)$ are \textbf{uniquely quasi-negated}.
\end{defn}

 By definition, $(-)a$ is a
quasi-negative of $a$.  

 \begin{MNote}
  The introduction of quasi-negatives to
replace  negatives  enables
 us
 to
develop
 the  analogs of some of the most basic
 structures of algebra.   The natural
condition  of unique quasi-negatives is surprising powerful,
  showing that $a' (-)a \in \mathcal A^\circ$ implies $a' = a$.
 \end{MNote}

\begin{lem}\label{negmatch} Suppose that $(-)$ is uniquely
quasi-negated.
\begin{enumerate}\eroman
\item
 For  $a_1 \ne a_2$ in
$\tT$ and $a_1 (-) a_2\in \tT,$ $a_1 + a_2 =  a_2$ implies  $a_1 (-)
a_2 = (-) a_2$ (or equivalently,  $a_2 (-) a_1 =  a_2$).

 \item If  $a_3 = a_1  + a_2 $, and $a_3(-)a_2$ is tangible then $a_3(-)a_2=a_1 $.\end{enumerate}
\end{lem}
\begin{proof} (i)  $ (a_1 (-) a_2) + a_2 = (a_1 + a_2)(-) a_2 = a_2(-)
a_2  = a_2 ^\circ,$ implying $ a_1 (-) a_2= (-)a_2$.

(ii)  $(a_3(-)a_2)(-) a_1 =  a_3(-)a_3 = a_3 ^\circ,$ so $a_1 =
a_3(-)a_2$.
\end{proof}

(As M.~Akian has pointed out, the conclusion of (i) fails   when $
\mathbf 2 = \one \ne e$ for $a_1 = a_2 = \one,$ since $a_1 + a_2  =
a_2  = \one$ whereas $a_1 (-) a_2  = e.$)

Other crucial consequences are to be given in
 Proposition~\ref{uniq67} and Proposition~\ref{uniq671}.
%
%
%


%
%

\subsubsection{Negation maps of the first and second kinds}$ $

\begin{defn}\label{kind0}
The negation map $(-)$ is of the \textbf{first kind}   if $(-)b =b$
for all $b \in \mathcal A.$ The negation map is of the
\textbf{second kind} if $(-)a \ne a$ for all $a \in \tT.$
\end{defn}

This distinguishes between ``supertropical'' and ``symmetrized''
algebra,  helping to explain why theorems holding for supertropical
semirings might fail for symmetrized semirings.

 \begin{lem}\label{hom00}
When $\mathcal A$ is a $\tT$-cancelative unital $\tT$-module with
$\one \in \tT$, the kind of the negation map is determined by
whether or not $(-)\one = \one.$
\end{lem}
\begin{proof}  If $a  \one = a = (-)a = a ((-)\one) $ for some  $a \in \tT$, then $\one = (-)\one.$

If $\one = (-)\one,$ then $b =  \one b=  ((-)\one)b   =  \one((-)b)= (-)b. $

\end{proof}

 As an
 indication of where we are headed, in the literature, a semigroup $(\mathcal A, +)$ is called
 \textbf{idempotent}
 (resp.~\textbf{bipotent}) if $a+a = a$ (resp.~$a+b \in \{ a, b\}$)
 for all $a,b \in \mathcal A$.

%
%

\subsection{Pseudo-triples and triples}

 \begin{defn}\label{sursys0}
 A \textbf{pseudo-triple}
is a collection $(\mathcal A, \tT, (-)),$  where   $\mathcal A$ is a
$\tT$-module and $(-)$ is a negation map on  $\mathcal A$.
%

 A \textbf{triple} is a pseudo-triple satisfying:
\begin{enumerate}\eroman
 \item $\tT \cap
\mathcal A^ \circ = \emptyset$,

 \item $\tT$
generates $( \mathcal A,+).$
\end{enumerate}

The triple is \textbf{unital} if $\mathcal A $ is unital with $\one
\in \tT.$
\end{defn}

\begin{MNote}  The triple is the fundamental
structure in this paper, taking the role akin to that of the base
ring for a module. $\tT \cap \mathcal  A^\circ = \emptyset$ enables
us to distinguish tangible elements from quasi-zeroes. The fact that
$\tT$ generates $( \mathcal A,+)$ enables us to reduce assertions
about $\mathcal  A$ to $\tT.$
\end{MNote}

\begin{example}\label{trivtrip} The only triple with $\tT$ trivial must be of the form $\Net a$ where $\tT = \{ a \}.$
 \end{example}

 \begin{prop}\label{nozero7} For any uniquely quasi-negated $\tT$-cancelative  pseudo-triple $(\mathcal A, \tT,
(-))$ with $\tT$ closed under multiplication (other than the  triple
with $\tT$ trivial of Example~\ref{trivtrip}), $\tT  \cap \mathcal A
^\circ = \emptyset$.
\end{prop}
\begin{proof} Suppose $a \in \tT  \cap \mathcal A ^\circ $. Then $a+a\in \mathcal A ^\circ $, so  $a =
(-)a$. Moreover, for any $a' \in \tT $,\ $aa'\in\tT$ and $aa' (-) a
\in \mathcal A ^\circ,$ so $aa' = a$. Likewise  $a^2 = a$, so $a' =
a$  by cancellation. Thus $\tT = \{a\}$.
\end{proof}

 Any unital triple has the
 subtriple
 $(\langle \one \rangle,  (\pm) \one, (-))$
generated by $\one$, which plays a useful role parallel to $\mathbf
F_1$-geometry.
 The pseudo-triple is needed in linear algebra, to describe linear combinations of vectors.

 \begin{lem}[{\cite[Remark~4.5]{AGG1}}]\label{circ0}  $\mathcal A^\circ $ is a $\tT$-submodule of $ \mathcal A$
 fixed under $(-)$, for any
 pseudo-triple $(\mathcal A,\tT,(-))$.
\end{lem} \begin{proof} $\zero  = \zero ^\circ \in \mathcal A^\circ,$
 and  $a(b^\circ) = a (b (-) b) = (ab) (-) (ab) = (ab)^\circ .$
 \end{proof}
 We express extra structure of $\tT$ and $\mathcal A$ with
the following terminology.

 \begin{defn}\label{minusmod}
%
%
%

 A \textbf{bimagma (resp.~cbimagma) triple}  is a triple $(\mathcal A, \tT, (-)),$ where $\mathcal A$ is also a bimagma (resp.~cbimagma).
A \textbf{semiring triple}  is a bimagma triple  which is also a
semiring (i.e., associative and distributive).
\end{defn}


Although the max-plus algebra is bipotent, idempotence   fails in
many other examples. It is too restrictive for our needs, so
instead, we settle for a slightly weaker version.

 \begin{defn}\label{impelt} A   triple  $(\mathcal A, \tT, (-))$ is  $(-)$-\textbf{bipotent}
  if $a_1 + a_2 \in \{a_1 ,a_2\}$
whenever $a_1, a_2 \in \tT$ with $a_2 \neq (-) a_1.$ In other words, $a_1 + a_2
\in \{a_1 ,a_2, a_1^\circ \}$ for all $a_1,a_2 \in \tT$. 
\end{defn}

 $(-)$-bipotence turns out to hold
in  the main
 variants of the max-plus algebra that   arise  in the tropical literature, but
 the following slightly weaker condition suffices to develop the theory.

\begin{defn}\label{metatan} A uniquely quasi-negated triple $(\mathcal A, \tT, (-))$   is \textbf{metatangible} if  $a + b \in \tT$ whenever $a, b \in \tT$ with $b
\neq (-) a.$
 \end{defn}

All  $(-)$-bipotent triples  clearly are metatangible.
\begin{rem}Often for metatangible triples,  negation maps of second
kind are preferable to those of first kind, since then $a \ne (-)a$
for $a \in \tT$ implies $a+a \in \tT$. If $(\mathcal A, \tT, (-))$
is $(-)$-bipotent with $(-)$ of second kind, then
$$ e+\one = \one (-) \one +\one = \one (-) \one = e = e(-)\one. $$ \end{rem}

 In Theorem~\ref{twoexamples}  we shall see
that ``most'' metatangible triples are  $(-)$-bipotent.

\begin{MNote} Metatangibility is the main engine for this paper, encompassing a wide range of examples, as seen
 in Theorem~\ref{mon11}.
\end{MNote}
%

 Tensor products of triples  are triples, as described below in
\S\ref{tenpro}. Although they lose  $(-)$-bipotence, they provide a
powerful tool in the theory.

%

%
%

\subsection{Surpassing relations}\label{surp}$ $

 Many subtleties cannot
be explained properly until we bring in another notion.

\begin{defn}\label{precedeq07}
A \textbf{surpassing relation} on a triple $(\mathcal A, \tT, (-))$,
denoted
  $\preceq $, is a partial pre-order satisfying the following, for elements of $\mathcal A$:

  \begin{enumerate}
 \eroman
  \item   $b \preceq b'$
  whenever $b + c^\circ = b'$ for some $c\in \mathcal A$.
    \item If $b \preceq b'$ then $(-)b \preceq (-)b'$.
  \item If $b \preceq c$ and $b' \preceq c'$   then  $b + b' \preceq c
   + c'.$
    \item   If  $a \in \tT$ and $b \preceq b'$ then $a b \preceq ab'.$
    \item   If  $a\preceq a' $ for $a,a' \in \tT,$ then $a =  a'.$
   \end{enumerate}
\medskip
We also write $b' \succeq b$ to indicate that $b \preceq b'$.

   A \textbf{strong surpassing relation} on a triple $\mathcal A$
   is a surpassing relation also satisfying  $b^\circ \not \preceq a $ for any $a \in \tT.$

%

A \textbf{surpassing $\circ$-PO} on $\mathcal A$
 is a surpassing relation   $\preceq $ which restricts to a PO on $\mathcal A^\circ$.



\end{defn}

One other property that one often wants is that $a \preceq a^\circ,$
which holds whenever $e'=e$, since then $a^\circ = a + a^\circ.$ It
holds in all of the tropical examples except the layered ones (when
$e'\ne e$), but fails miserably in the classical case. Here is our
major tropical example.

%

  \begin{defn}\label{precedeq23} The \textbf{$\circ$-relation} $\preceq_{\circ}
  $ is the relation $b  \preceq_\circ b'$ iff $b' =
b + c^\circ$ for some $c\in \mathcal A.$
\end{defn}

 In case $(-)$ is of the first kind, this
says $b' = b + 2c$.

\begin{lem}\label{precedeq008}
The relation  $\preceq_{\circ}$ is a partial pre-order satisfying
properties (i)--(iv) of Definition~\ref{precedeq07}.
\end{lem}
\begin{proof} (i) Condition (i)  of Definition~\ref{precedeq07} is by definition. To see Condition (ii), note that $a + c^\circ = b$ implies $(-)a + c^\circ = (-)(a + c^\circ) =
(-)b.$ Condition (iii) is immediate. For Condition (iv), if  $b' = b
+ c^\circ$ then $ab'  = ab + ac^\circ =  ab + (ac)^\circ.$
\end{proof}

\begin{prop}\label{uniq67} The relation  $ \preceq_\circ$ on a  triple with unique quasi-negatives is a strong surpassing
relation.
\end{prop}
\begin{proof} By Lemma~\ref{precedeq008}.  If $a = a' + c^\circ$ for
$a,a' \in \tT$, then $a  (-) a' = (a'+c)^\circ,$ implying $a' = a,$
yielding~Condition (v).  If $a' = b + c^\circ$ for $a'  \in \tT , b
\in A^\circ$, then $a' \in \tT \cap A^\circ$, a contradiction.
\end{proof}

The surpassing relation provides an important sub-$\tT$-module
 of $\mathcal A$.

\begin{defn}\label{Anull} $\mathcal A_{\operatorname{Null}} = \{b \in \mathcal A: b \succeq \zero\}.$
 \end{defn}

 \begin{rem}\label{ci1}   $ \mathcal A^\circ \subseteq \mathcal A_{\operatorname{Null}}$, by
 Definition~\ref{precedeq07}(i).
\end{rem}

%
%
%
%
%
%

\begin{MNote} One main
idea promoted here is that  the surpassing relation lies at the crux
of the theory, replacing equality when generalizing classical
theorems, and $\{ \zero \}$ should be replaced by the ideals
$\mathcal A^\circ$ and $\mathcal A_{\operatorname{Null}}$. Let us
see why the conditions of Definition~\ref{precedeq07} are desired.

\medskip

(i) shows that $\preceq$ refines $\preceq_\circ,$  and shows how the
quasi-zeros behave like $\zero$ under~$\preceq$.

(ii), (iii), and (iv) are needed to preserve the operations.

 By (v), the surpassing relation restricted $\tT$ is equality,
and provides a way of recovering   theorems from classical algebra.

\medskip

Ironically, instead of being symmetric (and thus an equivalence),
$\preceq$ is antisymmetric.  In introducing the surpassing relation,
which is not an identity  in universal algebra,
  we are crossing  over from universal algebra (with identities) to  a more general version
  of one-sided relations
  in model theory (with subtleties which we shall not discuss here).\end{MNote}

 Here is an enlightening
example of how $ \preceq_\circ$ generalizes classical algebra. In
any semiring with a negation map $ (-) $, we write $[a,b]$ for the
\textbf{Lie commutator} $ab (-) ba.$

\begin{lem}[\textbf{Leibniz $\preceq$-identities}]\label{Pois0}  $
[a,b] c +b[a,c]=[a,bc] +(bac)^\circ .$ In particular, $$[a,bc]
\preceq_\circ [a,b]
c +b[a,c]; \qquad [ab,c] \preceq_\circ a[b,c] + [a,c]b
.$$\end{lem}\begin{proof} $ [a,b] c +b[a,c] = (ab (-)ba)c + b(ac
(-)ca)= (abc(-)bca)+(bac(-)bac)=[a,bc] +(bac)^\circ$. The second
assertion is analogous.
\end{proof}


%
%

%
%

Note that when $\tT$ is a group, it satisfies the condition of
``thin elements'' of \cite[Definition~2.7]{AGG2}, and thus
\cite[Property~4.6]{AGG2} is relevant. We return to this issue in  \S\ref{matr1},
\S\ref{matr2} and in \cite{AGR}.

%
%


%
%

\subsection{Systems}\label{surpsy}$ $

 We put everything together.

\begin{defn}\label{system0}

A \textbf{system}  $(\mathcal A, \tT, (-), \preceq)$ is a triple
$(\mathcal A, \tT, (-))$ together with a surpassing relation
$\preceq$, satisfying the following properties:
\begin{enumerate}\eroman
 \item If $a + b
\succeq \zero$ for $a,b \in \tT$ then $b= (-)a $. (This is a bit
stronger than ``unique quasi-negatives'' for triples, but is the
same when $\preceq = \preceq_{\circ}$.)

\item
 $\tT \cap \mathcal A_{\operatorname{Null}}= \emptyset$. (This is already given when $\preceq = \preceq_{\circ}$.)
\end{enumerate}

The system $(\mathcal A, \tT, (-), \preceq)$ is of the
\textbf{resp.~first, second} kind if  $(-)$ is a negation map of the
resp.~first, second kind.


A \textbf{semiring system} (resp.~\textbf{$\tT$-semifield system})
 is a system where
$(\mathcal A, \tT, (-))$ is a semiring triple (resp.~$\tT$-semifield
triple).

 A \textbf{bimagma (resp.~cbimagma) system}  is a system where $(\mathcal A, \tT, (-))$ is a bimagma triple $(\mathcal A, \tT, (-)),$ where $\mathcal A$ is also a bimagma (resp.~cbimagma).

 The system is \textbf{strong} if the surpassing relation $\preceq$
 is strong.
\end{defn}

The role of $\preceq$ is sublime. It often comes naturally with the
triple, being equality for classical algebra, $\preceq_\circ$ in
tropically-oriented situations (see Theorem~\ref{circsurp}), and
$\subseteq$ for hypergroups (Definition~\ref{hyper}).


\begin{defn}\label{hyperty} A system $(\mathcal A, \tT, (-),\preceq)$ is of \textbf{hypergroup type} if:
\begin{itemize}
 \item
 $\mathcal A \subseteq \mathcal P(\tT)$, \item $\tT$ is identified with the set of singletons of ~$\mathcal A $,
 \item $\preceq = \subseteq$, \item $\zero \in a(-)a$ for all $a \in \tT$, \item $b+ b' = \cup \{ a_i+ a_j': a_i \in b, \, a_j' \in
  b', \, \forall b,b' \in \mathcal A\}.$\end{itemize}
\end{defn}

\begin{lem} Any system   $(\mathcal A, \tT, (-),\subseteq)$ of
hypergroup type  is   strong.
\end{lem}\begin{proof} Otherwise suppose $b ^\circ \subseteq \{a\}$. Then $b^\circ
= \{a\} \in \tT\cap \mathcal A^\circ = \emptyset.$
\end{proof}

\begin{MNote} Our two main sorts of systems either have $\preceq = \preceq_{\circ}$
or are of hypergroup type, and thus are strong.
 \end{MNote}

Systems include the classical case, the ``standard'' supertropical
 semiring, and  symmetrized algebras of ~\cite{AGG2}, as well as examples given in \S\ref{exsys}:  hyperfields
 \cite{Vi}, fuzzy rings \cite{Dr},  tracts \cite{BB2}, the ``layered'' semiring of~\cite{IKR0},
 the ``exploded'' algebra \cite{Par}, and the ELT-algebra of
 Sheiner~\cite{BlS,Sh}.
Once the  system  is established, it provides a mechanism for
obtaining effective definitions of new tropical
 algebraic structures, and also provides a guide for applying classical algebraic techniques
 in situations such as in \cite{AGR,GaR,JMR}.

The next observation is the key to the relationship between
$\mathcal A$ and~$\tT$ in a system.

\begin{prop}\label{uniq671} Suppose $(\mathcal A, \tT, (-), \preceq)$ is a system. \begin{enumerate}\eroman
 \item
  If $a + c = b$ for $a,b \in \tT$ and $c \in A_{\operatorname{Null}},$ then $b =
  a.$
   \item   $b \not
\preceq a $ for any $a \in \tT $ and $b\in \mathcal
A_{\operatorname{Null}}.$
 \item Conversely,  suppose that $(\mathcal A, \tT, (-))$
has a sub-$\tT$-module $\mathcal I \supseteq \mathcal A^\circ$ with
$\tT \cap \mathcal I = \emptyset$, having the property that for each
$a \in \tT$ there is a unique $a' \in \tT$ such that $a+a' \in
\mathcal I$. Define $\preceq_{\mathcal I}$  by $b \preceq _{\mathcal
I} b'$ if $b' = b + c$ for $c \in \mathcal I$. Then
$\preceq_{\mathcal I}$ is a strong surpassing relation, and
$(\mathcal A, \tT, (-),\preceq_{\mathcal I})$ is a system with
$\mathcal I = \mathcal A_{\operatorname{Null}}$.
\end{enumerate} \end{prop}
\begin{proof}
(i)
  $  b(-)a= a^\circ +c \in A_{\operatorname{Null}},$ so $b = a.$

(ii) $a \succeq b \succeq \zero$ contradicts $a \in\tT.$

(iii) $\zero \preceq _{\mathcal I} b'$ iff $b'   \in \mathcal I$,
proving $\mathcal I = \mathcal A_{\operatorname{Null}}$. The rest
follows easily from (i) and (ii).\end{proof}

 There are two ways of approaching systems ---
one is  in terms of the basic (expanded) algebraic structure, for
instance a ring or semiring, and the other, as in representation
theory, is in terms of a secondary structure (such as a module). Our
emphasis in this paper is on the former,   covering the basic
tropical algebraic structures, hypergroups, and fuzzy rings. One can
perform standard algebraic constructions,
 such as matrices, formal traces, bilinear forms, quadratic forms (all in
 Example~\ref{varex000}).
 The theory of semiring systems can   be viewed in the context of Lorscheid's ``blueprints''~\cite{Lor},  but also
  their specific extra information  permits us to hone in on
 other applications, which are not necessarily associative. The
 second approach is taken in \cite{AGR,JMR,JMR1}.

\subsection{Organization of this paper}\label{over1}$ $

 See \S\ref{over27} for a list of  the main
results. This paper is structured as follows:

\begin{enumerate}\eroman
  \item  A brief survey of the main concepts
  (tangible, negation map, triple, and system)  has been given in \S\ref{conc0}. The main objective is to
  describe the set $\tT$ in terms of a better structured $\tT$-module $(\mathcal A,+)$
  which also inherits whatever structure comes with $\tT$.
     \item    Metatangible triples, characterized by the property
     that $a_1 + a_2 \in \tT$ for all $a_1 \ne (-)a_2 \in \tT$, are  the focus of this paper and include all of the tropical
 applications, as well as many hyperfields and fuzzy rings. Surpassing relations,
     an extension of equality on~$\tT$,    replace equality in much of the theory.
   These ingredients are combined in  \S\ref{surpsy}  to
       yield the system.
        \item The major applications (max-plus
 algebra, supertropical
semirings, symmetrized semialgebras,  layered \semirings0,
``classical'' semialgebras,
 hypergroups, and fuzzy rings) are
described in \S\ref{tropexa}. In particular, the important technique
of symmetrization is presented in \S\ref{symm10}, to provide  a
  negation map when one is lacking.  Also we we bring ordered monoids into the picture.
          \item  In \S\ref{surpasrel} we describe other major examples, and bring in other key properties of triples and
          systems.
  \item  Additional structure is considered in
  \S\ref{unal}, where we also describe matrices, including the determinant, involutions, polynomials and their roots (to pave the way for affine
 geometry in Remark~\ref{geom1}),
  localization, and tensor products (\S\ref{tenpro}).
  \item  Metatangible systems are studied in depth and largely classified in
~\S\ref{surp22}.
\item
 Seeing that
systems  have a robust algebraic theory, we proceed to view them
categorically in \S\ref{catnot}, utilizing the surpassing relation
$\preceq$ as an essential
ingredient in the definition of morphism. 

%
%

\item Linear algebra is treated in \S\ref{linalg1}, with emphasis on   varying notions of matrix rank.
\item Tropicalization, which provides the connection with classical
mathematics via valuations, is studied  in \S\ref{Exsmore} in terms
of morphisms of systems. This provides the framework of defining and
investigating tropical analogs of classical algebraic structures.
  \item The Lie point of view is given in \S\ref{Liealg}.
  \item Directions suggested for further
research are given in \S\ref{ques}.

 \item The two main non-tropical applications are hypergroups (marked by
 $*$)
   and  fuzzy rings (marked by $**$).  Since hypergroups provide a rich source of examples and
 motivation, they
 are treated throughout the main text, although
 their main examples are put into Appendix A in order
 not to interrupt the flow of this study.
 Likewise,  fuzzy
 rings are viewed as systems in Appendix~B.
 \end{enumerate}
%

%
%
%
%
%
%
%
%

 \subsection{Main
results of this paper}\label{over27}$ $

\begin{nothmg}[Proposition~\ref{uniq671}] A system is described explicitly in terms of a
triple $(\mathcal A, \tT, (-))$ and a ``null''
$\tT$-module~$\mathcal A_{\operatorname{Null}}$.
 \end{nothmg}

Although distributivity can fail (e.g., over  hyperfields), there is
a way of recovering distributivity:

\begin{nothma}[Theorem~\ref{distres}]  Any left and right $\tT$-module $\mathcal A$
generated additively by a commutative monoid  $(\tT,\cdot)$ can be
made (uniquely) into a semiring via the multiplication rule
$$ \left(\sum_i a_i\right)\left(\sum_j b_j\right) = \sum _{i,j} a_i b_j$$
for $a_i, b_j \in \tT$.
 \end{nothma}


 \begin{nothmf}[Theorem~\ref{hypersys}, Proposition~\ref{hypersys7}]  For any hypergroup $\tT$ with negation $(-)$,
 let  $\overline{\tTz}$ be the sub-semigroup of
 the power set additively spanned by $\tT$; then $(\overline{\tTz}, \tT, (-))$ is a triple.
   $\subseteq$   is a surpassing~PO, and $(\overline{\tTz}, \tT, (-),\subseteq) $ is
  a  $\tT$-strictly negated system.

  The   hypergroup $\tT$ is~metatangible (resp.~closed),
iff its hypersystem $(\overline{\tTz}, \tT, (-), \subseteq)$
is~metatangible (resp.~$(-)$-bipotent).
\end{nothmf}

  The assortment of hypergroup
examples, given in~Appendix~A, sheds considerable light on the
theory.

Our main results on metatangible systems
 are of special interest,
  encompassing all of the tropical algebraic
 theories, as well as many other examples.

 \begin{nothmaa}[Theorem~\ref{twoexamples}] Any  cancelative metatangible unital triple  $(\mathcal A,\tT,(-))$  satisfies one of the following cases:
 \begin{enumerate}\eroman  \item   $\mathcal A$
 is $(-)$-bipotent.
 \item $\one (-)\one + \one =  \one,$ with one of the following two possibilities.
 \begin{itemize}  \item  $(-)$ is of the first kind, of characteristic 2. (In other words
 $\mathbf 3 = \mathbf 1 = \one.)$  In this case, $A$ has height
 $\le 2.$
  \item   $(-)$ is of the second kind, either of finite
  characteristic or with  $\{ \mathbf m : m \in \Z \}$ all distinct.
  \end{itemize}\end{enumerate}
 \end{nothmaa}

 \begin{nothmb}[Theorem~\ref{matrsys}]
 Any element $c$   of height $m_c \in \Net$   in   a    metatangible triple has a \textbf{uniform
 presentation},
  $c = m_c c_\tT$ for some element $c_\tT\in \tT$ and $m\ne 2$, or $c =
  c_\tT^\circ$.
 \end{nothmb}

 \begin{nothmc}[Theorem~\ref{uniq71}]
The uniform presentation is unique for any element of height $> 2$
in a cancelative $(-)$-bipotent unital system.
 \end{nothmc}

%

  \begin{nothme}[Theorem~\ref{circsurp}]
For any cancelative metatangible unital triple $(\mathcal
A,\tT,(-))$, we have the metatangible system $(\mathcal A,\tT,(-),
\preceq_\circ)$. Conversely, if there are elements $b \preceq b'$
but $b \not \preceq_\circ b'$, the triple $(\mathcal A,\tT,(-))$ is
either of first kind, of height $>2$, or of height 2 satisfying $b +
b' = b .$
\end{nothme}

   Despite the large assortment of examples given in \S\ref{Exs} and \cite{Row16v},  metatangible systems are described in
   Theorem~\ref{mon11}, as belonging to one of the classes reducing to the familiar examples from tropical theory,
    or satisfying specific properties called ``exceptional.''

%


 \begin{nothmi}[Theorem~\ref{hypsys1}] There is a faithful functor $\Psi$ from the
 category of canonical hypergroups into the category
 of  $\tT$-reversible systems, whose morphisms are the $\preceq$-morphisms, sending a
 hypergroup~$\tT$ to its  hypersystem $(\overline{\tTz}, \tT, (-), \subseteq)$.
\end{nothmi}

   \cite[Example~6.46]{Row16v} discusses weakening distributivity for hypersystems.

Symmetrization (\S\ref{symm10},\, \S\ref{symm11}) is a powerful tool
that enables us to move from arbitrary $\tT$-semirings to systems,
 leading to a
  major application,  the \textbf{transfer principle} which passes identities of rings to $\tT$-semiring
systems, and obtained for matrices in \cite[Theorem~3.4]{AGG1}
following an idea originating in  \cite[p.~352, end of proof of
(a)]{RS} carried on in  \cite{Ga},  as described in \S\ref{transf}.
To state the transfer principle precisely, one needs free objects.
Jacobson \cite[\S2.7]{Jac1980} provides a unified free construction,
but we only work with the following special cases.

The parent structure in tropical algebra is the well-known
\textbf{max-plus
 algebra}, described thoroughly in~\cite{ABG}.
  We append the subscript $_{\operatorname{max}}$ to indicate the
corresponding max-plus algebra, e.g., $\Net_{\operatorname{max}}$
or~$\Q_{\operatorname{max}}$, but to emphasize the algebraic
structure theory we still use the usual algebraic notation of
$\cdot$ and $+$ throughout (rather than $\odot$ and $\oplus$).
\begin{defn}
Throughout this definition, $\mathcal M$ denotes the \textbf{free
monoid}, which is the monoid in formal indeterminates, with
multiplication given by concatenation.

 \begin{enumerate}\eroman
   \item The  \textbf{free  monoid with negation map} ${\mathcal M}_{\pm}$
is the monoid in formal indeterminates and their formal negations,
with multiplication given by concatenation, together with the
relation $(-e_i)h = e_i((-)h)$ defined inductively. (For example,
$((-)e_1)( (-)e_2) = e_1e_2.$)

 \item  The  \textbf{free  \semiring0}
also is $\Net [\mathcal M]$ with the  \semiring0 structure.

 \item The construction of the   \textbf{free $\tT$-monoid semialgebra} is
similar, where we take $\mathcal A$ to be the free module over a
monoid $\tT$. 

 \item  The \textbf{free associative $ \Net_{\operatorname{max}}$-semialgebra}, is
 $\Net_{\operatorname{max}}[\mathcal M].$  Here, $ n x$ evaluates as $x$ in
$\Net_{\operatorname{max}}[\mathcal M].$

 \item The \textbf{free associative $\mathbf N$-semialgebra}   is $\mathbf
N [\mathcal M]$.

 \item The  \textbf{free $\mathbf N$-semialgebra with negation map}
$\mathbf N [ x, (-)x]$ is $\mathbf N [{\mathcal M}_{\pm}]$, which we
denote as $\mathbf N [ x, (-)x]$, writing~$x_i$ instead of $e_i.$ A
typical element of $\mathbf N [ x, (-)x]$ has the form $\sum
_{\mathbf i} (a_{\mathbf i} (-) b_{\mathbf i}) x_{\mathbf i}$, where
$\mathbf i = (i_1, \dots, i_t).$

  \item We can do this all with magmas.
The  \textbf{free  magma}   is the magma in formal indeterminates,
with multiplication given by concatenation, written with parentheses
since we lack associativity.  The  \textbf{free  magma with negation
map} ${\mathcal M}_{\pm}$ is the magma in formal indeterminates and
their formal negations.

 The  \textbf{$\Net$-magma
semialgebra}  is $\Net [\mathcal M]$  where $\mathcal M$ is the free
magma.

The  \textbf{free  bimagma} is given the natural bimagma structure.

%
%

 The \textbf{free (nonassociative) $ \Net_{\operatorname{max}}$-semialgebra}, which we
denote
 $ \Net \{\{ X \}\}$, is
the magma semialgebra $\Net_{\operatorname{max}}[\mathcal M].$  The
\textbf{free  $\mathbf N$-semialgebra}   is $\mathbf N [\mathcal M]$
where $\mathcal M$ is the free magma. The \textbf{free $\mathbf
N$-semialgebra with negation map} $\mathbf N [ x, (-)x]$ is $\mathbf
N [{\mathcal M}_{\pm}]$, which we denote as $\mathbf N [ x, (-)x]$,
writing~$x_i$ instead of $e_i.$
 \end{enumerate}
\end{defn}

 \begin{nothmh}[Transfer principle, Theorem~\ref{trans2}]  Suppose $P= \sum
_{\mathbf i} a_{\mathbf i}  x_{\mathbf i},\,  Q=   \sum _{\mathbf i}
b_{\mathbf i} x_{\mathbf i}\in \mathbf N \{ x, (-)x   \}$, where $
a_{\mathbf i}\ge b_{\mathbf i}$ for each $\mathbf i$. If the free
semiring satisfies the identity $\bar P = \bar Q $, then $ P
\succeq_\circ Q$ in $\mathbf N \{ x, (-)x \}$.
\end{nothmh}

 \subsubsection{The free module}  $ $

%
%

\begin{defn}\label{freem}
The \textbf{free  module}  $\mathcal A^{(I)}$ is the usual direct
sum of copies of $\mathcal A$ (i.e., with almost all entries
$\zero$), where we identify the $i$-th base element with the vector
having $\one$ in the $i$ component. If $\mathcal A$ is a
$\tT$-module then so is~$\mathcal A^{(I)}$, under the diagonal
action $a(b_i) = (ab_i).$  We can define  $ \tT^{(I)}\subseteq
\mathcal A^{(I)}$. We write $ \tT_{(j)}$ for the $j$-th component of
$\tT^{(I)}$, and put $\tT_{\mathcal A^{(I)}}:= \cup _j \tT_{(j)}$.
\end{defn}

Note that for $|I|$ infinite, $\tT^{(I)}$ is not a monoid even when
$\tT$ is a monoid.

\begin{example}\label{free1}
We can formally obtain a ``generic'' negation map in the free
$\mathcal A$-module. Given a set~$I$, we formally define $(-)I$ to
be another copy of $I$, indexed formally by $(-)i: i \in I,$ and put
$\overline I $ to be the disjoint union $ I \cup (-) I.$ The
\textbf{free $\mathcal A$-module with negation map} over a semiring
$\mathcal A$ is the free $\mathcal A$-module ${\mathcal
A}^{(\overline I )}$ whose base is formally denoted as $\{ e_i,\,
(-)e_i  = e_{(-)i}: i \in I\}$, with negation
map given by $e_i \mapsto (-)e_i$ and $(-)e_i \mapsto e_i.$ 
\end{example}

\begin{rem}\label{pillar7}  Notation as in Definition~\ref{freem}. \begin{enumerate}\eroman
 \item
If $(\mathcal A,\tT,(-))$ is a pseudo-triple, then $(\mathcal
A^{(I)},\tT,(-))$ is a pseudo-triple, where the negation map is
defined diagonally.

  \item If $(\mathcal A,\tT,(-))$ is a   triple, then  $(\mathcal A^{(I)},\tT_{\mathcal A^{(I)}},(-))$
    is a  triple over $\tT_{\mathcal A^{(I)}}$, defining negation and multiplication componentwise.
 Likewise for a system,
 the surpassing relation is defined componentwise.

\item If  a  triple $(\mathcal A, \tT, (-))$ has unique
quasi-negatives, then so does the triple $(\mathcal A^{(I)},
\tT_{\mathcal A^{(I)}}, (-))$,
 which will play an important
role in Sections \ref{tenpro}, \ref{linalg1}, and~\ref{Exsmore}.
\end{enumerate}
\end{rem}

Special cases are matrix semirings over semiring triples and
polynomial semirings over semiring triples, to be treated separately in
\S\ref{matr1} and \S\ref{poly}.

 \subsubsection{Linear algebra}$ $

 Linear algebra over systems is
particularly intriguing, since some of the supertropical results go
over, but others have counterexamples, as discussed in
\S\ref{linalg1}.

The negation map $(-)$ is used to define the $(-)$-determinant and adjoint in
 Equation~\eqref{eq:tropicalDetsign}.

 \begin{nothmj}[Theorem~\ref{prodform3}, reformulation of \cite{St}]  $ \absl {A } \absl {B} \preceq _\circ \absl
 {AB},
  $ for any matrices $A,B\in M_n(\mathcal A)$.
\end{nothmj}

 The main result unifying different notions of matrix
rank in \cite[Theorem~3.4]{IzhakianRowen2009TropicalRank} are
formulated rather transparently in this more general context
(for tangible vectors) in \S\ref{linalg1}, but only one direction
holds:

 \begin{nothmk}[Theorem~\ref{A1part}]  If the rows of a tangible
$n \times n$ matrix $A$ over a cancelative $(-)$-bipotent triple
 are dependent, then $|A| \in \mathcal A ^\circ.$
\end{nothmk}

  A wide-ranging counterexample for the converse was presented in \cite{AGG1}. This flavor of the
theory seems to depend on whether negation is of the first or second
kind, since we have more positive results for systems of the first
kind in \cite{AGR}, which delves more deeply into linear algebra
over systems.


%
%

Having the basic theory in place, we turn in \S\ref{Exsmore} to the
mainstay of tropical mathematics, which is tropicalization.
Tropicalization is explained  in \S\ref{tropi} as a
$\preceq$-morphism of systems.

\subsubsection{Applications}$ $

Here is a sample illustration of how the systemic theory can be
given a
 classical flavor.

 \begin{nothmm}[Proposition~\ref{Liesy}] If $L$ is a Lie semialgebra
(over a commutative semiring $C$) with a negation map, then
  $\adL$ is a Lie sub-semialgebra of $\End _C L$, and there is a Lie
 $\preceq$-morphism
  $L \to \adL$, given by $a \mapsto \ad_a.$
  \end{nothmm}

%
%

%

The main examples of hyperfields (supertropical hyperfield, Krasner
hyperfield, hyperfield of signs, phase hyperfield, and triangle
hyperfield) are described explicitly  in Appendix A as systems, some
of which are bipotent.

The application to fuzzy rings is given in Appendix B:

 \begin{nothmn}[Theorem~\ref{drive6}] Any $\tT$-coherent fuzzy triple $\mathcal A$ gives rise to a
system $(\mathcal A', \tT, (-), \preceq)$ with unique
quasi-negatives, where $(-)a = \vep a.$
  \end{nothmn}

The converse is given in  Proposition~\ref{drive3}.

\subsection{Redefining multiplication}$ $

 This discussion  is intended for
those readers who would like to see how hypergroups fit into the
theory. The motivation grew out of a conversation with Baker. Since
the ``tropical hyperfield'' of \cite{Bak} and \cite[\S5.2]{Vi} is
isomorphic to the ``extended'' tropical arithmetic in Izhakian's
Ph.D. dissertation (Tel-Aviv University) of 2005, also
cf.~\cite{zur05TropicalAlgebra},  given more formally
 in \cite{IzhakianRowen2007SuperTropical}, one would like  to see
how other major hyperrings also can be studied in terms of the more
amenable semiring theory. The tricky aspect is to obtain
distributivity for all of $\mathcal A$, which can be written down as
follows, when we assume that $\tT$ generates $(\mathcal A,+)$:

\begin{equation}\label{dist0}
    \left(\sum _i a_i\right) \left(\sum _j b_j \right) = \sum _{i ,j}  a_ib_j.\end{equation}
 for
$a_i, b_i \in \tT$.
 For instance, in  the study of
hyperfields it might seem at first glance that we must forego
distributivity in $\mathcal P (\tT)$, since the multiplication in
the power set of certain hyperfields need not distribute over
addition, as to be seen in Examples~\ref{Basicexamples}. However,
this difficulty is bypassed by the following surprising result,
which we call a theorem because of its significance, despite its
being almost trivial.

\begin{thm}\label{distres} Any left and right $\tT$-module $\mathcal A$
generated additively by a commutative monoid  $(\tT,\cdot)$ can be
made (uniquely) into a $\tT$-semiring with absorbing $\zero $, via
the multiplication
\begin{equation}\label{semir} \left(\sum_i a_i\right)\left(\sum_j b_j\right) = \sum
_{i,j} a_i b_j,\end{equation}
$$ \left(\sum_i a_i\right)\zero= \zero \left(\sum_i a_i\right)
=\zero,$$ for $a_i, b_j \in \tT$.
 \end{thm}
\begin{proof}
It suffices to show that this is well-defined, i.e., if $\sum_i a_i
= \sum_i a'_i$ then $ \sum _{i,j} a_i  b_j = \sum _{i,j} a'_i b_j$
(and likewise for $b_j,b_j'$) for $a_i', b_j' \in \tT$. But
$$ \sum _{i,j} a_i  b_j = \sum _{j}\left(\sum _i a_i  b_j\right) = \sum _{j}\left(\sum _i a_i\right)  b_j = \sum _{j}\left(\sum _i a'_i\right)  b_j
 = \sum _{i,j} a'_i b_j.$$
 Going the other direction, distributivity in the semiring
 forces \eqref{semir} to hold.
\end{proof}

%
%

%

\section{The main   tropically oriented
triples and systems}\label{tropexa}$ $

We review some of the concepts that have played a major role in
tropical algebra.

\subsection{The max-plus
 algebra and bipotent semirings}$ $

 The
max-plus algebra
 really concerns ordered groups,
such as $(\Q,+)$ or $(\R,+)$, which are viewed at once as max-plus
  \semifields0, generalizing Remark~\ref{makesemi}.

%
%

\begin{rem}\label{maxpsys} The
identity   is the only negation map on the max-plus
 algebra, by \cite[Proposition~2.11]{AGG2}). Then $a = (-)a =
 a^\circ$, so the max-plus
 algebra cannot be a triple (and unique quasi-negatives fail).
\end{rem}

\subsubsection{Green's partial order}$ $

We recall the following elegant observation of Green:

\begin{rem}\label{makesemi} (i) Any  ordered monoid $(\tM, \cdot \; )$
  gives rise to
a bipotent \semiring0, where we define $a+b$ to be $\max\{ a, b\}.$
Indeed, associativity is clear, and distributivity  follows from
the inequalities~\eqref{ogr1}. 
%

(ii) Conversely, any set $\mathcal M$ with a partial addition and
$\zero$ has a natural partial pre-order given by $a_1 \ge a_2$ in
$\tM$ if $a_1 = a_2 + b$ for some $b \in \mathcal M.$ It is a
pre-order when $\mathcal M$ is bipotent.

 Any such relation becomes trivial on a $\tT$-module when $\tT$
contains $-\one $, the classical negative of $\one ,$ since then
$a_2 = a_1 +(a_2+(-\one)a_1)$.
\end{rem}

Remark~\ref{makesemi} is tied in with the following property:

\begin{defn}\label{prope}
 An additive semigroup $S \subseteq \mathcal A$ is  \textbf{ub} (for \textbf{upper
 bound})
  if $a+b+c=a$ always implies $a+b=a.$%
%
%
\end{defn}

For example, the max-plus algebra is a ub \semifield0. This
criterion abounds in tropical algebra, as noted in \cite{IKR4}.
  By \cite[Proposition~0.5]{IKR5}, the Green partial pre-order     is a PO iff the
semigroup~$\mathcal A$ is ub.
%
%

\subsection{Supertropical  semirings  and supertropical
domains}\label{trop}$ $

 The difficulty arising in Remark~\ref{maxpsys} is  remedied by turning to supertropical
 algebra \cite{zur05TropicalAlgebra,IKR0,IzhakianRowen2007SuperTropical}.

\begin{defn}\label{super1} A  \textbf{supertropical semiring} is a quadruple
$\mathcal A  := (\mathcal A , \tTz, \tGz, \nu)$ where  $\mathcal A =
\tTz \cup \tGz$ is a commutative semiring, $\tTz$ is a submonoid,
and $\tGz \subset \mathcal A $ is a bipotent ideal given Green's
 order of~Remark~\ref{makesemi}(ii), together with an onto
multiplicative monoid homomorphism $\nu: \mathcal A \to \tGz$
satisfying $\nu^2 = \nu$, with $\nu|_{ \tT}$ being 1:1. Addition is
given by
 $$
 b+b' = \begin{cases}\nu(b)   \quad   \text{whenever}   \quad   \nu(b)= \nu(b'),\\
  b  \quad  \text{whenever}  \quad  \nu(b)> \nu(b'),\\
   b'  \quad  \text{whenever}  \quad  \nu(b) < \nu(b').
\end{cases} $$

 A  \textbf{supertropical domain} is a  supertropical semiring,
 where $\tT = \tTz \setminus \zero$ (in case $\zero \in \tTz$) is a multiplicative monoid.
\end{defn}

%

%
%
%
%

The elements of $\tGz$ are called \textbf{ghost elements} and $\nu:
\mathcal A \to \tGz$ is called the \textbf{ghost map}.      The
monoid~$\tT$ encapsulates the tropical aspect. To get a triple, we
take $(-)a = a,$ a negation map of the first kind.

 The \textbf{standard supertropical semifield} is
$\mathcal A : = \tT \cup \tGz$ where $\tG := (\tGz\setminus \{ 0 \},
\cdot)=\nu(\tT)$ is an ordered group. (Customarily $\tGz = \Q
_{\operatorname{max}}  $ or $ \R _{\operatorname{max}} $, rewriting
$+$ as $\cdot$).

\begin{defn}\label{sutropsys}

 The  standard
supertropical semifield  $\mathcal A : = \tT \cup \tGz$ yields the
\textbf{standard supertropical system} $(\mathcal A , \tT,
1_{\mathcal A }, \preceq_\circ)$, where   $b^\circ$ is $b+b=\nu(b),$
and $b\preceq_\circ b'$ when $b' = b + c^ \circ$ for some $c \in
\mathcal A .$ (The relation $\succeq_\circ$ is ``ghost
  surpasses,'' written as:

\medskip

 $a_1 \lmodg a_2$ in $\mathcal A$
 if $a_1 = a_2 + b+b$ for some $b \in \tT $,

\medskip

\noindent and $\nu$ of Definition~\ref{super1} is a special case of
$\circ.$) Now $e = \one_{\mathcal A } +\one_{\mathcal A } =
\one_{\mathcal A }^\circ$ is the multiplicative unit of $\tG$, and
$\tGz = e \mathcal A  = {\mathcal A }^\circ$.
\end{defn}

 Classical
algebraic results  were transferred to the tropical theory by means
of $\lmodg$ in \cite{AGG1,IzhakianRowen2007SuperTropical,IR1}.

\subsection{Symmetrization}\label{symm10}$ $

Although the max-plus algebra and its modules initially lack
negation, one obtains negation maps for them through the next main
idea, the symmetrization process, extracted from \cite{AGG1,Ga,GoM},
where a $\tT$-module is embedded into a super-module.


 \begin{definition}\label{sym00}
Given any $\tT$-module $\mathcal A$, not necessarily with a negation
map, define $\widehat {\mathcal A} = \mathcal A \times \mathcal A$
with componentwise addition, and $\widehat {\tT} =(\tT \times \{
\zero \}) \cup ( \{ \zero \} \times \tT)$ with multiplication
$\widehat {\tT} \times \widehat {\mathcal A}\to \widehat {\mathcal
A}$ given by the \textbf{twist action}
$$(a_0,a_1)(b_0,b_1) = (a_0 b_0 + a_1 b_1, a_0 b_1 + a_1 b_0), \quad a_i \in \tT, b_i \in \mathcal A.$$
When $\mathcal A$ is a bimagma we extend   multiplication to
$\widehat {\mathcal A}$ by  taking also $a_i \in \mathcal A$.

We also define the negation map on $\widehat {\mathcal A}$ as the
\textbf{switch map} $(-)_{\operatorname{sw}}$ given by
$(-)_{\operatorname{sw}}(b_0,b_1) = (b_1,b_0).$ Then
$(b_0,b_1)(-)_{\operatorname{sw}} (b_0,b_1) = (b_0+b_1,b_1+b_0),$ so
the quasi-zeroes all have the form $(c,c)$, and the surpassing
relation $\preceq_{\circ} $
 is given by:

\begin{equation} (b_0,b_1) \preceq_{\circ}  (b_0',b_1')  \quad
\text{iff} \quad  b_i'   = b_i + c  \  \text{for some } c \in
{\mathcal A}, \ i = 0,1.
\end{equation}
(The same $c$ is used for both components.)

 $(\widehat {\mathcal A},\widehat
 {\tT},(-)_{\operatorname{sw}},\preceq_{\circ})$ is the  \textbf{symmetrized
 system}.
 \end{definition}

\begin{rem}\label{sym07}
This is the structure  given at the beginning of \cite[\S 3.8]{GaP}
in the case $\mathcal A = \Real_{\max}$, and is the venue for
\cite[\S 3.4]{BCOQ}, \cite{AGG1},  \cite[Example 2.21]{AGG2},
 and
\cite{BeE,JoM}),  rather than what is called the ``symmetrized
algebra'' in~\cite{GaP}. But we prefer the terminology
``symmetrized'' for this version, which is appropriate to the
general structure theory.
\end{rem}

\begin{lem}$\widehat {\tT}\cup \{ \zero, \zero\}$ is a monoid whenever $\tT$  is.
$\widehat {\tT} $ is a union of two subgroups whenever $\tT$ is a
group.
\end{lem}
\begin{proof}
It is closed under multiplication, and inverses exist when $\tT$  is a group
($(\zero,a)^{-1} = (\zero,a^{-1})$).
\end{proof}

In particular, $\widehat{\mathbf N \cup \{\zero\}} $ is itself a
semiring which we call $\mathbf Z ,$ with negation given by
$(-)(\mathbf m, \mathbf n) = (\mathbf n, \mathbf m)$. The
construction of $\mathbf Z $ from $\mathbf N $ takes the place of
the familiar construction of $\Z$ from $\Net,$ with the difference
that here we distinguish $(\mathbf m, \mathbf n)$ from $(\mathbf m +
\mathbf k, \mathbf n + \mathbf k).$

\begin{example}\label{twist2}  For any semigroup $({\mathcal A},+)$,
$\hat {\mathcal A} $ is naturally a module over $\widehat{\Net} $.
\end{example}

\begin{lem}   $\widehat{\mathbf N} [\mathcal M]$ is isomorphic to $\widehat{\mathbf N
[{\mathcal M}]}$.
\end{lem}
\begin{proof} We send $(\mathbf m, \mathbf
n)a$ to $(\mathbf m a\, , \mathbf n a)$ for $a \in \mathcal M$, and
check that addition and multiplication are  preserved.
\end{proof}

%

%
The free module with negation map  can be viewed as the
symmetrization of the free module (without negation) ${\mathcal
A}^{(I)}$, where we identify $ e_i$ with $(e_i,\zero)$ and
$(-)e_i$ with $(\zero,e_i)$. 


%

\subsubsection{The $(-)$-bipotent symmetrized system: the version according to \cite[Proposition-Definition
2.12]{AGG2}}\label{symm11}$ $

Definition~\ref{sym00} fails to be metatangible. The following
modification  was introduced in \cite[Proposition~5.1]{AGG1} and
studied further under the name of ``symmetrized max-plus semiring''
in \cite[Proposition-Definition 2.12]{AGG2}  in the context of
tropical constructions.

\begin{example}\label{surpsym} One starts with an ordered semigroup $\tG$, putting $\tG_0 =  \tG
\cup \{ \zero\}, $ and defines $\tT_{\operatorname{sym}} = (\tG
\times \{ \zero\}) \cup ( \{ \zero\} \times \tG),$
$$\tG_{\operatorname{sym}}:= \tT_{\operatorname{sym}} \cup \{ (a,a): a\in \tG_0 \} \subseteq \widehat{\mathcal
A}.$$ Thus, viewing $\tG_0$ as a bipotent semiring,  addition on
$\tG_{\operatorname{sym}}$ is defined componentwise on $\tG_0 \times
\{ \zero\}$, $\{ \zero\} \times \tG_0$, and $ \{ (a,a): a\in \tG_0
\}$, whereas ``mixed'' addition satisfies:
 $$(a_0,\zero)   + (\zero, a_1) = \begin{cases} (a_0,\zero)  \text{ if } a_0 > a_1;
 \\ (\zero, a_1) \text{ if } a_0 < a_1;  \\  (a_1, a_1)
 \text{ if }   a_0 =  a_1 ; \end{cases}$$
 $$(a_0,\zero)   + (a_1, a_1) = \begin{cases} (a_0,\zero)  \text{ if } a_0 > a_1;
 \\ (a_1, a_1) \text{ if } a_0 \le a_1;  \end{cases}$$
$$(\zero, a_0)   + (a_1, a_1) = \begin{cases} (\zero, a_0)    \text{ if } a_0 > a_1;
 \\ (a_1, a_1) \text{ if } a_0 \le a_1.  \end{cases}$$

 Multiplication in $\tG_{\operatorname{sym}}$ is the twist action as
in Definition~\ref{sym00}. The   \textbf{$(-)$-bipotent symmetrized
system} is
$$(\tG_{\operatorname{sym}},\tT_{\operatorname{sym}},(-)_{\operatorname{sw}},
\preceq_\circ).$$
\end{example}

\begin{MNote}
 The symmetrized system
 of Example~\ref{surpsym}  is a system which is $(-)$-bipotent of the second kind,
 but its construction requires that $\tG$ be ordered.

 To obtain an ordered semigroup $\tG$ we could apply the following
 modification of
Remark~\ref{makesemi}(ii) to a bipotent semiring: $$a_1 \le a_2
\quad \text{ iff }\quad a_1+a_2 = a_2.$$ However, more in the spirit
of this study is to apply this condition to the tangible elements of
a $(-)$-bipotent pseudo-triple,   also putting $a\le a$. This works
whenever $(-)$ is of the first kind, but when $(-)$ is of the second
kind we might have $a$ and $(-)a$ incomparable, leading to
Definition~\ref{positive} below.

To bypass these considerations we
 often use the symmetrization of Definition~\ref{sym00} anyway. \end{MNote}

\subsection{*Hypergroups}\label{hyperr}$ $

Our other major example of a system is over a hypergroup $\tTz$,
defined below, which has inspired much of our material in systems.
 We follow the  treatments of Baker and Bowler~\cite{Bak} and Jun~\cite{Ju},
   discussing   examples from \cite{Bak} in Example~\ref{Basicexamples}; also see  ~\cite{CC,Vi}.

One would like to formulate the structure of $\tTz$ in terms of addition (as
well as  other possible operations such as multiplication) on
$\tTz$.
 But this is not feasible since $\tTz$ itself need not be closed under addition.

Intuitively, a hyper-semigroup should be a structure
$(\tTz,\boxplus,\zero)$ with $\boxplus : \tTz \times \tTz \to
\mathcal P(\tTz)$, for which the analog of associativity holds:
$$ (a_1 \boxplus a_2) \boxplus a_3 = a_1 \boxplus (a_2 \boxplus
a_3), \quad \forall a \in \tTz. $$

 $\tTz$ as defined in \cite{Bak,Vi} is injected
naturally into $\mathcal P(\tTz)$, identifying $\tTz$ with the singletons in $\mathcal P(\tTz)$.
There is some difficulty in the details:
$a_1 \boxplus a_2$ need not be a singleton, so technically $ (a_1
\boxplus a_2) \boxplus a_3 $ is not defined. This difficulty is
exacerbated with generalized associativity; for example, what does $
(a_1 \boxplus a_2) \boxplus (a_3\boxplus a_4) $ mean?

\begin{defn}\label{Hyp00} A \textbf{hyper-semigroup} is a structure
$(\tTz,\boxplus,\zero)$, where
\begin{enumerate}\eroman
   \item
$\boxplus$ is a commutative binary operation $\tTz \times \tTz \to
\mathcal P(\tTz),$ which also is associative in the sense that if we
define $$a \boxplus S = \cup _{s \in S} \ a \boxplus s,$$ then $(a_1
\boxplus a_2) \boxplus a_3 = a_1 \boxplus (a_2\boxplus a_3)$ for all
$a_i$ in $\tTz.$
  \item
$\zero$ is the neutral element.
\end{enumerate}

We   transfer addition to $\mathcal P(\tTz) $ by defining $ S_1 +
S_2 = \left\{ \bigcup (a_1 \boxplus  a_2) : a_i \in S_i\right\}.$
\end{defn}

  Note that repeated
addition in the hyper-semigroup need not be defined until one passes
to its power set, which hampers checking basic identities
such as associativity. Associativity could hold at the level of
elements but fail at the level of sets.

\begin{defn}\label{Hyp}
A \textbf{hyperzero} of a hyper-semigroup $\tTz$ is a subset of
$\mathcal P(\tTz) $ containing $\zero$.

 A \textbf{hypernegative} of an element $a$ in a hyper-semigroup $(\tTz,\boxplus,\zero)$ is an
element  ``$-a$'' for which $\zero \in a \boxplus (-a).$

(Following \cite[Definition 2.1]{Ju}) A \textbf{hypergroup} is a
hyper-semigroup $(\tTz,\boxplus,\zero)$ for which every element~$a$
has a unique hypernegative $-a$.   \textbf{Hypernegation} is the map
$a \mapsto -a.$

A \textbf{canonical hypergroup} is
a hypergroup satisfying the extra property:

\begin{itemize}
\item  (Reversibility) $ a \in b \boxplus c$ for $a,b,c \in \tTz$  iff  $ b \in a
\boxplus (-c)$.
\end{itemize}

\end{defn}

We need to translate this into triples.

\begin{rem}\label{Hen1}
 Henry ~\cite[\S 2]{Hen}
shows that the reversibility condition holds if associativity holds
and hypernegation distributes over addition, in the sense that
$-(a\boxplus b) = (-a)\boxplus (-b)$.
\end{rem}

M.~Akian indicated the reverse direction to me,
cf.~\cite[Lemma~3.7]{AGR}.

\begin{prop}\label{uniq6} A hypergroup is canonical if and only if hypernegation distributes over addition.
\end{prop}
\begin{proof} $(\Leftarrow )$ By Henry ~\cite[\S 2]{Hen}.

$(\Rightarrow)$  If $c \in -(a\boxplus b) $ then $-c \in a\boxplus
b,$ so $a \in (-c) \boxplus (-b),$ and $$\zero \in a\boxplus (-a)
\subseteq ((-c) \boxplus (-b))\boxplus (-a) = (-c) \boxplus
((-b)\boxplus (-a)),$$ i.e., $ c \in (-b)\boxplus (-a).$ This proves
that $-(a\boxplus b) \subseteq ((-b)\boxplus (-a)).$ On the other
hand, if $ c \in (-b)\boxplus (-a),$ then  $-b \in c\boxplus a,$
implying $\zero \in (-b) \boxplus b \subseteq c\boxplus (a \boxplus
b),$ so $-c \in a\boxplus b.$
\end{proof}

%
%
%
%
%
\begin{defn}\label{hyring} 
%
%
 $(\tTz,\boxplus,\cdot,\zero, \one)$ is a  \textbf{hyperring} if $\tTz$ is a monoid and
 $\mathcal P(\tTz)$  is both a bimagma (with multiplication $S_1 S_2 = \{ a_1a_2: a_i \in S_i\}$)
 and a $\tTz$-module.

 We put $\tT = \tTz \setminus \{ \zero \}.$ A hyperring $(\tTz,\boxplus,\cdot,\zero,\one)$ is a  \textbf{hyperfield} if
 $(\tT,\cdot,\one)$ is a group.

  \end{defn}
%
%

%
%
%
%


 \begin{lem}\label{mon6}  Hypernegation on a canonical hypergroup (or~hyperring) $\tTz$ is a negation
map, and induces a negation map on  $\mathcal P(\tTz)$, viewed as a
pseudo-triple via $(-)S = \{ -a: a \in S \}.$
\end{lem}
\begin{proof} We see that $-(a_1 \boxplus a_2) =  (-a_1) \boxplus (- a_2),$
by Proposition~\ref{uniq6}.

Likewise, $ \zero \in a_1 \boxplus (-a_1)$ implies $ \zero \in a(a_1
\boxplus (-a_1)) = aa_1 \boxplus (-aa_1).$

In case  $\tTz$ is a hyperring we note from the previous paragraph
that $-(a_1a_2)= (-a_1)a_2 = a_1(-a_2),$ and thus $(-a_1)(-a_2) = -
(-a_1) a_2 = a_1 a_2.$
 \end{proof}


Our first candidate for a system might be $(\mathcal P(\tTz),\tT,
(-), \subseteq)$. But
  $\tTz$ might not generate
$(\mathcal P(\tTz),\boxplus)$, as seen below for the phase
hyperfield (Example~\ref{Basicexamples}). Accordingly, we restrict
$(\mathcal P(\tTz),\boxplus)$.

  \begin{defn}\label{hyper}
  Given a hypergroup
  $\tTz$, we define $\overline{\tTz}$  to be the
sub-semigroup of $(\mathcal P(\tTz)\setminus \emptyset,\boxplus)$
generated by the singletons, which we identify with $\tTz$.
%
\end{defn}

\begin{rem}\label{hypersys0}   $(\overline{\tTz}, \tT, (-))$, with $(-)$   as in
  Lemma~\ref{mon6}, is a   triple having unique  quasi-negatives.
\end{rem}

   In \cite[Examples 2.8, 2.9,
2.12]{Bak}   the negation map is the identity map, whereas in
\cite[Examples 2.10, 2.11]{Bak} it is the usual hypernegative.

\begin{thm}\label{hypersys} $(\overline{\tTz}, \tT, (-),\subseteq) $ is
  a   system of hypergroup type, for any  hypergroup     $\tTz.$
 \end{thm}
\begin{proof}
We need to verify the conditions   of Definition~\ref{precedeq07}.
To see Condition (i), suppose $b = a \boxplus c^\circ$. Since $\zero
\in c^\circ,$ we have $a \subseteq a \boxplus \zero \subseteq a
\boxplus c^\circ = b.$

 The other conditions   are clear
 (since any nonempty subset of a singleton is that singleton).

 The  relation $\subseteq$
clearly satisfies $S \subseteq \{a\}$ iff $S = \{a \}.$  \end{proof}

We call  $(\overline{\tTz}, \tT, (-),\subseteq)$
 a \textbf{hypersystem}.

 \begin{MNote} Thus the theory of canonical hypergroups and
hyperrings embeds into the theory of hypersystems.
 We make this more formal  in Theorem~\ref{hypsys1}. The
 recent surge in research in hypergroups  provides further motivation and intuition for
 the study of systems.
\end{MNote}

%

%
%
%
%

%
%

\begin{defn}\label{bihyp1} A canonical hypergroup $\tTz$ is $(-)$-\textbf{closed}
if $a \boxplus b \in \tT  $ whenever $a \ne -b;$  $\tT$ is
$(-)$-\textbf{bipotent} if $a\boxplus b \in \{a, b \}$ whenever $a
\ne -b.$
\end{defn}

\begin{lem}\label{closed1}  $\tTz$ is $(-)$-closed,
resp.~$(-)$-bipotent, iff the   hypersystem $(\widetilde{\mathcal
P({\tTz})}, \tT, (-), \subseteq)$ is metatangible,
resp.~$(-)$-bipotent.
\end{lem}
\begin{proof} The definitions match.
\end{proof}




\begin{example}\label{manyexamples77}$ $

\begin{enumerate}\eroman  \item $(-)$-bipotent hypergroups include  Viro's ``tropical hyperfield,''
which is isomorphic to the tangible part of the supertropical
algebra,
 the Krasner hypergroup (of the first kind),
and the sign hypergroup (of the second kind), all of height 2.

 \item   A natural system that is not metatangible, the phase hyperfield, will be presented in Example~\ref{Basicexamples}, taken from  \cite{AGG2}.
It is idempotent of height 3.

\item   Viro's  ``triangle'' hyperfield of Example~\ref{Basicexamples},
 is of the first kind and is neither idempotent nor metatangible. Distributivity
 also fails in its system.

\end{enumerate}
\end{example}

In order to accommodate hypergroups such as the ``triangle''
 hyperfield,
 we relax
 the semiring assumption (i.e., distributivity) for $\mathcal A$.
Recently ties have been
 found in \cite{GJL} between hyperfields and fuzzy rings,   which
 also are described in terms of systems in Appendix~B.

%

\section{Exploring triples and systems}\label{surpasrel}$ $


\subsection{Layered \semirings0}\label{ELTalg0} $ $

``Layered semirings'' were introduced in \cite{IKR0}, and called
``extensions'' in ~\cite[Proposition-Definition~2.12]{AGG2}.  They
are of the form $  L \times \tG$, where $L$ is a ``layering
semiring'' and $(\tG,\cdot)$ is an ordered monoid.  
In fact, associativity of multiplication  in $\tG$ is irrelevant, so
we could also define ``layered
bimagmas.'' 

\begin{example}\label{AGGGexmod1}     We assume that the ``layering
semiring'' $L$ is an arbitrary semiring with $1$, but with a
negation map that we designate as $-$.
 We can define the \textbf{layered
semiring} as follows:

$\mathcal A = L \times \tG$. The \textbf{layers} are the subsets $\{
\ell \} \times \tG$ for $\ell\in L$. Multiplication is defined
componentwise. Addition is given by:
 $$(\ell_1,b_1) + (\ell_2,b_2) = \begin{cases}  (\ell_1,b_1) \text{ if } b_1 > b_2;
 \\ (\ell_2,b_2) \text{ if } b_1 < b_2;  \\  (\ell_1 + \ell_2,\, b_1)
 \text{ if }   b_1 =  b_2  .
 \end{cases}.$$%

Usually $\tT = \{\pm 1\} \times \tG$.
 The negation map
will be given by $(-)(k,a) =(-k,a).$ Thus the quasi-zeros will be of
layer $1-1$.
\end{example}

(This construction is modified and extended in \cite{AGR0}.) Some
explicit examples of layered semirings:

\begin{example}\label{nontang}$ $
\begin{enumerate}\eroman
 \item   $L=\Net$, formally with $-\ell = \ell$,   $\tT = \{ (\ell,a) \in L \times \tG : \ell =
1\},$ and $(-)$ is the identity (thus  of the first kind). $
\tT^\circ$ is the layer~2. (The higher layers, if they exist, are
neither tangible nor in $\tT^\circ$. In fact $e' = \one+\one+\one$
has layer 3.) This is useful for supertropical differentiation.

\item
 $L = \Net_0$ in (i), with $\{ \zero\}$ formally adjoined  at layer
$0$.

 \item  $L= \Z$ with the usual negation,
$ \tT  = \{ (\ell,a) \in L \times \tG: \ell = \pm 1\},$ and
$(-)(\ell,a) =(-\ell,a),$  of the second kind. This is useful for
supertropical integration, cf.~\cite{IzhakianRowen2007SuperTropical}.

\item A somewhat more esoteric example from the tropical standpoint. Fixing $n>0$, taking  $L = \Z _n $, identify   each
layer modulo $n$. (This   has height $n$ and    characteristic~$n$.)

\item (The truncated semiring) A weird example, which leads to counterexamples in linear algebra
in \cite{AGR} and must   be confronted.
  Fixing $n>1$, we say that $L = \{ 1, \dots, n\}$ is \textbf{truncated} at
  $n$ if
addition and multiplication are given by identifying every number
greater than $n$ with  $n$. In other words,

$$k_1 + k_2 = n  \text{ in $L\quad$ if }\quad k_1 + k_2 \ge n  \text{ in  }
\Net;\qquad k_1  k_2 = n  \text{ in $L\quad$ if }\quad k_1  k_2 \ge
n   \text{ in  } \Net.$$
 The negation map is the identity.

 This triple has
characteristic $0$, since $\mathbf m \ne \one$ for all  $\mathbf m
> \one$, but it has height $n$.
\item $L$ itself is a classical algebraic structure, such as a ring, or an exterior algebra, or a Lie algebra.
  \item   (exploded-ELT -- special case of (vi)) $L$ is  the residue ring of a valuation with value group $\tG$,
 where now $ \tT  = \{ (\ell,a) \in L \times \tG : \ell \ne 0\}.$
\end{enumerate}

\end{example}

Example~\ref{AGGGexmod1} can be modified,  viewing an ordered monoid
as a bipotent semiring as in Remark~\ref{makesemi}(i).

\begin{example}\label{symgen} Suppose
 $\tG$ is   a  semiring with a negation map $(-)$, whose addition
yields a PO on $\tG^\circ$, via Remark~\ref{makesemi}(ii), with
$(b^\circ)^\circ =b^\circ$ for all $b$. We write $b_1 <_\circ b_2 $
for $b_i\in \tG$ if $b_1^\circ < b_2^\circ,$
and define addition by:
 $$(\ell_1,b_1) + (\ell_2,b_2) = \begin{cases}  (\ell_1,b_1) \text{ if } b_1 >_\circ b_2;
 \\ (\ell_2,b_2) \text{ if } b_1 <_\circ b_2; \\  (\ell_1 + \ell_2,\,b_1^\circ)
 \text{ if }   b_1 ^\circ =  b_2 ^\circ.
 \end{cases}
 $$
 $$(\ell_1,b_1^\circ) + (\ell_2,b_2) = \begin{cases}
 (\ell_1,b_1^\circ) \text{ if } b_1 >_\circ b_2;  \\ (\ell_2,b_2) \text{ if } b_1 <_\circ b_2;
 \\  (\ell_1 + \ell_2,\,b_1^\circ   )
  \text{ if } b_1 ^\circ =  b_2 ^\circ.\end{cases}$$
 $$(\ell_1,b_1^\circ) + (\ell_2,b_2^\circ) =  \begin{cases}  (\ell_1,b_1^\circ) \text{ if } b_1 >_\circ b_2;
 \\ (\ell_2,b_2^\circ) \text{ if } b_1 <_\circ b_2;  \\  (\ell_1 + \ell_2,\,
 b_1^\circ)
 \text{ if }   b_1 ^\circ =  b_2 ^\circ.\end{cases}$$

 This has a negation map given by $(-)(\ell,b) = (-\ell, (-)b).$

The requirement $(b^\circ)^\circ =b^\circ$ gives this
 example a supertropical flavor, but if we try to delete this
 assumption we have some difficulty defining the connection between $(b^\circ)^\circ $ and
 $b^\circ$.
\end{example}

\subsection{Major examples of systems, by height}\label{exsys}$ $

%
Before delving further into the theory, we describe some of the main
examples in terms of height.

\begin{example}\label{manyexamples}  Semiring  and  semifield systems, where $\preceq = \preceq_\circ$.

\begin{enumerate}
 \eroman

  \item  Height 1. This makes $\tTz = \mathcal A.$
\begin{itemize}
 \eroman
 \item Classical algebra, for example an
 integral domain.  (There are many other examples in classical algebra, including graded algebras, cf.~\S\ref{grT}.)
 Here the quasi-negative  is
 the usual negative, which  is unique, and $ \mathcal A^ \circ= \{\zero\}$. $b \preceq_\circ a$ iff $a = b + \zero =b,$
so we have
   the metatangible system $(\mathcal A, \tT, -,
=)$. When $\tT$ is a group, a classical  system is just a
\textbf{partial field} in the terminology of
\cite[Definition~4.2]{Bak}, also cf.~\cite{SW}.

In some ways we want the general theory of metatangible systems to
mimic classical algebra.  The negation map is of the second kind
unless $\mathcal A$ has characteristic~2, in which case $(-)$ is of
the first kind. This helps to ``explain'' why the theory of
metatangible systems of the first kind often has the flavor of
characteristic 2.

 \item    The max-plus algebra $\tT $ yields a pseudo-triple, taking  $\mathcal A =  \tT $ and $(-)$ the identity
 map, so $a^\circ = a$ and $ \mathcal A^\circ = \mathcal A$, but quasi-negatives are far from unique,
 since whenever $b<a$ we have $a +b = a = a^\circ$.
 (This is one reason why we shy away from the max-plus algebra in our algebraic theory.)
 Here $a^\circ = b^\circ$ implies $a =b,$
 cf.~Definition~\ref{Anticlas}.
\end{itemize}

\item Height 2. These  systems provide  tropical structures designed to  refine the max-plus algebra. All of them are
$(-)$-bipotent.  The familiar examples have characteristic 0,
although some constructions can also be replicated in positive
characteristic.

\begin{itemize} \item Supertropical \semirings0,
 cf.~Definition~\ref{super1}, can be described as the   $(-)$-bipotent systems
 $(\mathcal A, \tT, (-), \preceq)$ of the first kind and height 2, where
$\tT$ is the set of tangible elements, and $a^\circ +a = a+a+a = a+a
= a^\circ  = a^\nu.$

 \item The symmetrized system
of Example~\ref{surpsym}  has height 2.
%
%

%


\item The ``exploded'' system, following \cite{Par}, cf.~Example~\ref{AGGGexmod1}, studied
axiomatically under the name of ELT-algebra by Sheiner~\cite{Sh},
where $\mathcal A = L \times \tG$ with $L$ the set of lowest
coefficients of Puiseux series, $\tT = (L\setminus 0) \times \tG$,
and $(-)(\ell,a) = (-\ell,a),$ is $(-)$-bipotent of the second kind,
provided $L$ is not of characteristic 2.
\end{itemize}

 \item Height $\ge 3.$
 \begin{itemize} \item The ``layered'' system of
Example~\ref{nontang}(i,ii,iii),  designed to handle derivatives and
integration, is $(-)$-bipotent of the first kind. Its   height is
equal to the cardinality of the submonoid of $L$ generated by $1$.
It often provides counterexamples to assertions that
  hold  in height~2.
\end{itemize}

\end{enumerate}
\end{example}

 \begin{rem}\label{fk}  Triples of the first kind behave quite differently from those of the
second kind.

Triples of the first kind  that contain $\one$ satisfy $e' = \one +
\one + \one = \mathbf 3.$

\begin{itemize}

\item If $e' = \one,$ then we are in characteristic 2.

\item If $e' = e,$ i.e.~$\mathbf 3= e = \mathbf 2, $ the system often has height 2, such as
in the first two examples of~Examples~\ref{manyexamples}(iii), and
when $(-)$-bipotent it is isomorphic to the supertropical domain:
$\tT $ is the set of tangible elements, and $\tT^\circ $ is the set
of ``ghost'' elements.

\item Otherwise we are usually in the more
esoteric region of height $\ge 3$, occurring for layered semirings,
cf.~Examples~\ref{manyexamples}(iii), as well as certain
hyperfields~(Examples~\ref{manyexamples77}(ii,iii)).
\end{itemize}

 $\tT$-Triples of the second kind often have either the flavor of classical
algebra or of the symmetrized algebra. $(-)$-bipotent triples of the
second kind  are all  idempotent since $a+a \in \max\{a, a\} = a$. (The
converse also holds for metatangible
triples, as to be seen in~Corollary~\ref{twoexamples2}.) 
\end{rem}

\subsection{Tropically related examples viewed in terms of algebraic varieties and model theory}\label{trop0}$ $

Let us see how well  notions related to $\tT$-modules  mesh with
universal algebra.

\begin{defn}[\cite{Jac1980}]\label{varclo} A
\textbf{ variety} is a class $\mathcal V$  of algebraic structures
(in universal algebra) which are  systems closed in the following
sense:
\begin{enumerate}\eroman
\item Any  substructure  of a structure of  $\mathcal V$   is  itself in $\mathcal
V$;
\item If $\varphi: \mathcal A \to \mathcal B$ is a
 homomorphism with $\mathcal A \in \mathcal V$, then $\varphi(\mathcal A) \in \mathcal
 V$;
\item  The Cartesian product of systems in $\mathcal V$ is in $\mathcal
V$.
\end{enumerate}

 \end{defn}

$\tT$-modules fit in the general framework of universal algebra once
we
 have the algebraic structures $\mathcal A $ and $\tT$.  Triples fit in, viewing  the negation map as
 a unary operator on each of $\mathcal A $ and
 $\tT$. However, the surpassing relation is not an identity, since it is not
 symmetric.

 Varieties arising naturally in tropical mathematics  include
 \semirings0, semirings, idempotent \semirings0, modules over
 \semirings0,
 bimagmas and semialgebras, and super-semialgebras. The defining
 identities are the familiar ones. Here are some subtler instances.

\begin{example}\label{varex000}$ $%
\begin{enumerate}\eroman
    \item \textbf{Matrix semirings.} The $n \times n$ matrix structure
can be obtained   in terms of matrix units $\{ e_{ij}: 1 \le i,j \le
n \}$, viewed as constants (0-are operators) satisfying the
identities
$$\sum _{i = 1}^n e_{ii} =\one; \qquad e_{ij}e_{k\ell} = \delta
_{j,k} e_{i\ell}, \quad 1 \le i,j,k,\ell \le n.$$ We get a variety
by viewing the matrix units  as part of the structure in this way.
Namely, given a matrix semiring $\mathcal A  = M_n(R)$ with matrix
units $\{ e_{ij}: 1 \le i,j \le n \}$ and a semiring homomorphism
$\varphi: \mathcal A \to \mathcal A',$ the set $\{ \varphi( e_{ij}):
1 \le i,j \le n \}$ is a set of $n \times n$ matrix units for
$\mathcal A'$. Note that the base ring $R$ for the matrices can be
recovered as $\sum e_{i,1} \mathcal A e_{1,i} ,$ so does not have to
be notated separately. One could make the base ring commutative by
means of a suitable identity.

One major point is that the standard proof given for example in
\cite[Proposition~13.9]{Row08} does not use negation.

 Matrices gives rise to the trace
operator $\tra (a) = \sum e_{ii}a e_{ii}.$ Note, for $n \ge 2,$ that
$\tra (I) = \one^\circ$ over the supertropical \semiring0. The
determinant is more problematic since the classical formula involves
negatives; we shall return to this issue in \S\ref{matr1}.

\item \textbf{Formal traces.}
Since much of linear algebra involves the trace bilinear form, let
us formalize the trace from the previous example. For a semialgebra
$ \mathcal A$ over a commutative associative ring $C$ define a trace
operator $\tra : \mathcal A\to C$ satisfying the identities
$$\tra (x_1 x_2) = \tra (x_2 x_1); \qquad \tra(\one_\mathcal A) = n {\one_C} .$$
The trace operator can be viewed as a unary  operator, so we can
define the variety of semialgebras with traces.

\item \textbf{Bilinear forms on a module $M$.} Linear algebra, which
has played a significant role in tropical mathematics, can be taken
over systems. Although one can define modules $M$ over \semirings0
(and systems) the same way as  modules over rings, \cite{IKR45,JMR},
much research focuses on the free module $M = \mathcal A^{(n)}.$

A \textbf{bilinear form} is an operator $b: M \times M \to \mathcal
A$ satisfying the classical identities defining bilinearity; an
example is the inner product.

\item \textbf{Quadratic forms.} The general definition of quadratic form
over a semiring is given in \cite{IKR45}, also cf.~\cite{ChR}.
Continuing $\operatorname{(iii)},$ we introduce a (quadratic)
operator $Q: M \to \mathcal A$ satisfying the identity $Q(x+y) =
Q(x) + Q(y) + b(x,y)$, where the binary operator  $b(x,y)$ is the
accompanying bilinear form of $Q$. We get a variety by incorporating
the  quadratic form into the structure.
%

\item \textbf{Blueprints.}
Lorscheid~\cite[Definition~1.1]{Lor} has put tropical geometry in a
rather general framework, which we review.

 \label{blue} A \textbf{blueprint} $B$ is a monoid
$A$ with zero, together with  an equivalence relation~$\Cong$ on the
monoid \semiring0 $\mathbb N [A] = \{ \sum a_i: a_i \in A\}$ (of
finite formal sums of elements of $A$) that satisfies the following
axioms (where we write $\sum  a_i \equiv  \sum  a'_j$ whenever $(
\sum a_i, \sum a'_j) \in \Cong$):

\begin{enumerate}\eroman
\item The equivalence $\Cong$ is additive and multiplicative. (Thus $\Cong$
 is a congruence.)

\item  The absorbing element $ \zero$ of $A$ is compatible with the zero
of $\mathbb N[A];$ i.e., $\zero \equiv \text{empty sum}.$

\item  If $a,b \in A$ and $a \equiv
b,$ then $a = b$ (as elements in $A$). \end{enumerate}
%

%
 When the monoid $A$ has a given negation map $(-)$,
we extend the congruence $\Cong$ to $\mathbb Z [A]$ via the identity
generated by the relation $(-m)x \equiv m((-)x),$ and then the
$A$-blueprint $B$ has the negation map given by $(-) a = ((-)\one)
a.$ Indeed, we verify   the extra relation: if $\sum a_i \equiv \sum
b_j$ and $\sum a_i' \equiv \sum b'_j$, then
$$\sum a_i - \sum a_i' =  \sum a_i (-) \sum a_i' \equiv  \sum b_i (-)
\sum b_i' = \sum b_i - \sum b_i'.$$


%


\end{enumerate}
 \end{example}

%
%
%
%
%
%

\subsubsection{Structures of   tropical mathematics which do not comprise varieties}$ $

Several important concepts   fail to correspond to varieties,
because at least one of the key ingredients of
Definition~\ref{varclo}, either homomorphic images
 or direct products,  is missing.

\begin{example}\label{badex}$ $

\begin{enumerate}\eroman
\item \textbf{Ordered \semirings0 versus bipotence.} One takes Green's order
$a+b = \max \{a,b\}$, using the sentence
$$x_1+x_2 = x_1 \quad \vee  \quad x_1+x_2 = x_2,\quad \forall x_1, x_2,$$
which passes to subalgebras and homomorphic images, but not to
direct products, since for example (componentwise) $(1,2) + (2,1) =
(2,2).$

\item \textbf{Supertropical \semirings0.} One can describe supertropical \semirings0
(Definition~\ref{super1}), by declaring the constant $e: = \one^\nu$
to be both an additive and multiplicative idempotent, i.e., $e+e =e$
and $e^2 = e.$ Then $re = r^\nu$, so the map $r \mapsto re$
corresponds to the ghost map, cf.~\cite[Remark~2.1]{IKR}.
Supertropicality passes to subalgebras  but not to direct products,
just as in (i), and is  ruined in images when $\varphi(a) =
\varphi(a')$ with $a < a',$ since then $\varphi(a+a') =
\varphi(a')\ne \varphi(a')^\nu = \varphi(a)+ \varphi(a')$.

\item \textbf{ub
semigroups.} Any ub semigroup satisfies the property
$$x_1+x_2 +x_3 = x_1 \quad \Rightarrow \quad x_1+x_2 =x_1,\quad \forall x_1, x_2, x_3.$$
This property passes to sub-semigroups and direct products, but not
to homomorphic images, for example  $\Net \rightarrow \Z _2$, taken
modulo 2. (It is an example of a ``quasi-identity''  in mathematical
logic.)

%
\item  * \textbf{Hypergroups.} The power set construction fails to satisfy the conditions of Definition~\ref{varclo}.
\end{enumerate}
 \end{example}

\subsection{Properties of $e$, and the characteristic of a system $\mathcal A$}\label{chartri}$ $

\begin{lem}\label{uniq77} In a triple, if $e' \in \tT$, then $e' = \one.$
\end{lem}
\begin{proof} If $e'\in \tT$ then $e' = e + \one$ implies $e' (-) \one = e+e \in \mathcal A^\circ ,$
implying $e' = \one $ by unique negation.
\end{proof}

When $e'=\one$ the characteristic behaves in the familiar manner.

\begin{prop}\label{char77} Suppose $e'=\one.$ If $\mathbf m' = \mathbf  m,$ for $m < m',$ then
the characteristic divides $m'-m.$ In characteristic $0$, the
natural map $\Z \to \mathbf Z$ is 1:1.
\end{prop} \begin{proof} $\mathbf {m-1} =  \mathbf {m-2}  + \mathbf 1 =  \mathbf {m-2} + e'  =  \mathbf  m (-)\one = \mathbf {m'} (-)\one =  \mathbf {m'-2} + e'  =  \mathbf {m'-2} + \one  =  \mathbf {m'-1}
.$  We apply this argument $m-1$ times until we get $\mathbf 1 =
\mathbf {m'-m+1}.$

 The second assertion follows at once from the first.
\end{proof}

\begin{rem}\label{useful1}$ $
 When $(-)$ is of the first kind, then $\mathbf 2 = e \in \mathcal
A^\circ,$ implying inductively $\mathbf {2m} \in \mathcal A^\circ$
for each~$m$. Consequently, if $\mathbf {k} \in \mathcal A^\circ$
for some  odd $k$, then $\mathbf {m} \in \mathcal A^\circ$ for all
$m\ge k$.

\end{rem}

\begin{example}\label{chart1}$ $
\begin{enumerate}\eroman  \item $\mathcal A$ has
characteristic~$1$ iff it is idempotent.
\item For $(-)$
of the first kind, $ \mathcal A$  has characteristic~$1$ or ~$2$ iff
$e' \in \tT, $ by Lemma~\ref{uniq77}, since in this case $\mathbf  1
= e' = \mathbf 3$.

\item We say that $\mathcal A$ is \textbf{quasi-periodic} if
 $\mathbf m = \mathbf {m '}$ for some $m
< m'.$ For  minimal such $m$,  the $\mathbf j$ are distinct for all
$j\le m$, and then for $j>m$ the $\mathbf j$ comprise a cycle with
period $m'-m.$ (When $m=1,$ this is precisely the definition of
characteristic $m'-m.$ But one could have characteristic~0 with
$m>1$, as illustrated in the truncated algebra of
Example~\ref{nontang}(vi).)
%
\end{enumerate}
\end{example}

\begin{prop}\label{Euc} If $\mathcal A$ has characteristic $k$ and $\mathbf {m+1} = \mathbf {1},$ then   $k$ divides $m $.
\end{prop}
\begin{proof} A standard Euclidean algorithm argument. Write $m  = qk +r,$ where $0 \le r<k.$ By definition $m \ge
k,$ and $\mathbf {r+1} =  \mathbf{ qk+r+1} = \mathbf{m} + \mathbf 1=
 \mathbf 1.$ But $r  <  k  , $ so we must have $r=0.$
\end{proof}

%

\subsection{$\tT$-Strongly negated systems}\label{stneg}$ $

The next definition becomes  relevant when
metatangibility is absent, as for the phase hyperfield.

\begin{defn}\label{precedeq5}

A system  $(\mathcal A, \tT, (-),\preceq)$ is
 \textbf{strictly negated}   if, for any $c, d\in \mathcal A$,

\begin{equation}\label{eeq2}
d(-)c \succeq \zero \text{ implies  either }  c, d \succeq \zero
\text{ or }   c \preceq d \text{ or }   d \preceq c .
\end{equation}

A system  $(\mathcal A, \tT, (-),\preceq)$ is $\tT$-\textbf{strictly
negated}, if  for $c \in \tT$,
 $d(-)c \succeq \zero \text{ implies }  c  \preceq d.$
\medskip

\end{defn}

\begin{lem}\label{Euc7}  A strictly negated system  $(\mathcal A, \tT, (-),\preceq)$  is
$\tT$-strictly negated whenever $\mathcal A= \tT \cup \mathcal
A^\circ.$
\end{lem}
\begin{proof} If $d \in \tT$ then $d =  c.$ If $d \in \mathcal A^\circ$
then we cannot have $d \preceq c,$ by  Definition~\ref{precedeq07}.
\end{proof}


Although our emphasis in this paper is on metatangible systems,
``$\tT$-strictly negated'' is a reasonably broad substitute since it
 has several nice consequences  and holds for hypersystems:

\begin{prop}\label{hypersys7} Any hypersystem $(\overline{\tTz}, \tT, (-), \subseteq)$ is    $\tT$-strictly negated.
 \end{prop}
\begin{proof}  $\zero  \preceq d  (-) c$  means $\zero \in d
\boxplus (-c)$   so $\zero  \in (-c) \boxplus a$ for some tangible
$a \in d;$ thus $a  = c,$ i.e., $c \in d.$
\end{proof}

Also we will see in Proposition~\ref{msums1} that every cancelative
metatangible system  of height 2 is strictly negated.

\subsubsection{$\tT$-reversibility}\label{mat200}$ $


\begin{defn}\label{precedeq}
A surpassing relation  $\preceq$ in a system is
$\tT$-\textbf{reversible} if
 $a \preceq  b +c   $ implies  $b \preceq a (-) c$  for $a,b \in \tT.$

%

 A   $\tT$-\textbf{reversible}  system
is a system $(\mathcal A, \tT, (-), \preceq)$   where  $\preceq$ is
$\tT$-reversible.
%

%
%

\bigskip

%

\end{defn}

%
%

\begin{prop}\label{hyper15}  Any $\tT$-strictly negated system $(\mathcal A, \tT, (-), \preceq)$  is   $\tT$-reversible.\end{prop}
\begin{proof}  Suppose that
 $a \preceq  b +c   $ for $a,b \in \tT$, i.e., $ b +c (-)a \succeq \zero.$ Then $(a(-)c) (-)b   =(-) (b +c (-)a )\succeq \zero,$ so $b  \preceq  a (-) c.$
\end{proof}
%

\section{Related notions}\label{unal}
In this section we apply extra algebraic structure to  triples and
systems.

\subsection{Matrices over triples and systems}\label{matr1}$ $

 Matrices were defined in Example~\ref{varex000}(i).
$(M_n(\tTz), \cdot)$ needs no longer be a monoid even when $(\tT,
\cdot)$ is a monoid, because the matrix product involves addition.
Thus, we define $\tT_{M_n(\mathcal A)}$ to be $\cup \tT e_{i,j},$
yielding the triple $(M_n(\mathcal A), \tT_{M_n(\mathcal A)}, (-)).$
Matrices over a  $\tT $-system are a   $ \tT_{M_n(\mathcal
A)}$-module, as with classical algebra,  and we get a system over  $
\tT_{M_n(\mathcal A)}$, defining $(-)$ and $\preceq$ componentwise.
Summarizing, we have:

\begin{rem} If $( \mathcal A , \tT , (-), \preceq)$ is a
 bimagma system, then $(M_n(\mathcal A), \tT_{M_n(\mathcal A)}, (-),
\preceq)$ is a  bimagma system, in view of Remark~\ref{pillar7}.
\end{rem}

We also have a pseudo-triple $(M_n(\mathcal A), \tT, (-))$, where
$\tT$ is identified
 with the scalar matrices, which works with linear algebra in \S\ref{linalg1}.

\subsubsection{$(-)$-Determinants and singularity}\label{matr}$ $

We formulate some standard concepts from matrix algebra over semirings, largely
translated from \cite{AGG1}.

 \begin{defn}\label{signeddet} Suppose a commutative semiring $\mathcal A$ has a negation map
 $(-)$.
 For a permutation $\pi$, write
 $$(-)^\pi b = \begin{cases} b: \pi \text{ even}; \\ (-)b:   \pi \text{ odd}.\end{cases}$$
 The \textbf{$(-)$-determinant} $\absl A $ of a matrix $ A $ is
 \begin{equation}\label{eq:tropicalDetsign}
  \sum_{\pi \in  S_n}  (-)^{\pi} \left( \prod_i a_{i,\pi(i)}\right).
\end{equation}
The \textbf{even part} is $   \sum_{\text{even }\pi \in  S_n   }
\left( \prod_i a_{i,\pi(i)}\right), $ and the \textbf{odd part} is
 $
  \sum_{\text{odd } \pi \in  S_n }   \left( \prod_i a_{i,\pi(i)}\right).
$

The matrix $A$ is \textbf{nonsingular} if $\absl A \in  \tT.$ $A$ is
\textbf{singular} if $\absl A \succeq \zero.$
  $A$ is
$\circ$-\textbf{singular} if  $\absl A \in \mathcal A^\circ.$
\end{defn}


\begin{lem}  \begin{enumerate}\eroman
 \item  The $(-)$-determinant is linear in any given row or column.

 \item If $(a_{i,j}) \preceq (b_{i,j})$ then $\absl{(a_{i,j})} \preceq
\absl{(b_{i,j})}.$ This yields the $\tT_{M_n(\mathcal A)}$-system
$(M_n(\mathcal A), \tT_{M_n(\mathcal A)}, (-), \preceq ).$
\end{enumerate}\end{lem}
\begin{proof}
(i) The same as for the classical situation.

(ii) Match the sums and products in the formula.
\end{proof}

\begin{lem}\label{pair1} If two rows or columns of a matrix $A$ are the same, then $A$ is  $\circ$-singular. \end{lem}
\begin{proof}
The formula for the $(-)$-determinant partitions into pairs of
opposite $(-)$.
\end{proof}

\begin{prop}\label{pair2} If the first row $v_1$ of $A$ $\circ$-surpasses a linear combination of the other rows $v_2, \dots, v_n$, then $A$ is  $\circ$-singular. \end{prop}
\begin{proof}
Breaking up the first row $v_1 = w^\circ +\sum_{i\ge 2} \a_i v_i$
for some vector $w$, we write $A_j$ for the matrix whose first row
is $ v_j $ and $A_w$ for the matrix whose first row is $w$, and see
that $|A| = |A|_w^\circ + \sum \a_j|A_j|$. Hence, by
Lemma~\ref{pair1}, $|A|$ is a sum of elements of~$\mathcal
A^{\circ}$.
\end{proof}

\subsubsection{The special linear monoid}\label{matmon}$ $

In order for  this system theory to be at our disposal for
$M_n(\mathcal A)$ for a general commutative  \semiring0 $\mathcal
A,$ we pass to $M_n(\widehat {\mathcal A}),$ with the switch
negation map, in which case, for a matrix $A$ , the
$(-)$-determinant matches~ \cite{AGG1}. Namely, we define $\absl A
_\circ = \absl{(A,( \zero ))}$. This is an ordered pair $(a_0,
a_1)$, where $a_0$ is the even part  of the determinant and $a_1$ is
the  odd part.
  These considerations lead us to  the \textbf{$\preceq$-special linear} monoid $\operatorname{SL}_n(\hat {\mathcal A}): = \{ A \in M_n({\mathcal A}): \absl
{A}_\circ \succeq_\circ (\one, \zero)\}$.

 This is essentially the definition used in \cite{INR}. It contains
 all the elementary matrices, but is not generated by them multiplicatively, cf.~\cite{Niv2}.
Mimicking the classical algebraic
groups,
we can define
 $\operatorname{PSL}_n(\mathcal A)$ by taking  $\operatorname{SL}_n(\mathcal A)$
 modulo the congruence $\{ (A, \a A) : \a \in  \mathcal A\}.$
Versions of the other algebraic groups will be obtained presently by
utilizing involutions.

\subsection{Polynomial triples}\label{poly}$ $

To set up affine geometry over systems, one defines the polynomial
cbimagma $\mathcal A[\la]$ in the usual way over a cbimagma
$\mathcal
A$, as a monoid cbimagma. We write $\la$ for $\one \la.$ 
 $\mathcal A[\la_1, \dots, \la_n]$ is defined inductively. $\tT[\la]$
is not closed under multiplication, since for example $(\la +
\one)(\la (-) \one) = \la^2 + e \la (-) \one.$ Instead, one takes
$\tT_{\mathcal A[\la_1, \dots, \la_n]}$ to be the set of monomials
$\{ a \la_1^{i_1}\cdots \la_n^{i_n}: a \in \tT, \, i_j\in \Net_0\}$.

When $(\mathcal A  , \tT  , (-))$ is a triple, $(\mathcal A[\la_1,
\dots, \la_n] , \tT_{\mathcal A[\la_1, \dots, \la_n]} , (-))$ is a
triple, under the negation map $$(-)( a \la_1^{i_1}\cdots
\la_n^{i_n}) = ((-) a) \la_1^{i_1}\cdots \la_n^{i_n},$$ and has
unique quasi-negatives  if $(\mathcal A  , \tT  , (-))$ has unique
quasi-negatives. A surpassing relation $\preceq$ on~$\mathcal A $ is
extended to $\mathcal A[\la_1, \dots, \la_n]$ by comparing
respective components. (One does the same for Laurent polynomials,
rational functions, etc.) The analogy for cbimagmas also works.

 \begin{rem}\label{commatch}
 $   \hat {\mathcal A}[\la_1, \dots, \la_n] \cong\widehat {\mathcal A[\la_1, \dots, \la_n]} ,$
 under the map $$ \sum _{\tilde   i = (i_1, \dots ,i_n)}(\a_{0,\tilde i}, \a_{1,\tilde i})\la_1^{i_1} \cdots \la_n^{i_n} \mapsto \left(\sum _ {\tilde i= (i_1, \dots ,i_n)} \a_{0,\tilde i} \la_1^{i_1} \cdots \la_n^{i_n} ,\sum _{\tilde i = (i_1, \dots ,i_n)}
  \a_{1,\tilde i}\la_1^{i_1} \cdots \la_n^{i_n}\right),$$ seen by matching components, and noting that the map is 1:1 and
 onto.
 \end{rem}

 Polynomial systems are important in affine
geometry. One  often identifies polynomials in terms of their values
as functions, by mapping ${\mathcal A}[\la_1, \dots, \la_n] $
to~$\Fun ( {\mathcal A}^{(\{1, \dots, n\})},{\mathcal A})$, as given
in ~Definition~\ref{varclo}(iii).
 \begin{rem}\label{geom1} Generalizing \cite{IzhakianRowen2007SuperTropical, Ju},
 $\mathbf a \in \tT^{(n)}$ is a  \textbf{systemic root} of
$f \in \mathcal A [\la_1, \dots, \la _n]$ if $f(\mathbf a) \succeq
\zero.$\end{rem}

 This yields an analog of the Zariski topology, in which the closed
 sets are the sets of systemic roots.

\subsection{Involutions}\label{inv}$ $

\begin{defn} An \textbf{involution} on a system $(\mathcal
A, \tT,(-),\preceq )$ is an anti-isomorphism of degree 2, i.e., an additive
homomorphism $(*)$ satisfying ($\forall  \ b,b_i \in  \mathcal
A$):

 \begin{enumerate}\eroman\item $(b^*)^* = b$,
\item $(b_1b_2)^* = b_2^*b_1^*,$
 \item
 $((-)b)^*= (-)b^*.$
 \item If $b_1 \preceq b_2$
in $\mathcal A,$ then $b_1^* \preceq b_2^*.$
\item $\tT^* = \tT.$

\end{enumerate}
\end{defn}

(iv) is automatic for~$\preceq_\circ,$ since   $ (b_1+ c^\circ)^* = b_1^* + (c^*)^\circ.$

\begin{lem}  $(c^\circ)^* =
(c^*)^\circ.$
\end{lem}
\begin{proof} $(c^\circ)^* = (c(-)c)^*  = c^* (-)c^* =
(c^*)^\circ.$
\end{proof}

\begin{example}\label{matrixinv}  $ $ Examples of involutions on the matrix semialgebra $M_n({\mathcal
A})$  over $\mathcal A$:
\begin{enumerate}\eroman
   \item The transpose map on $M_n({\mathcal
A})$ is an involution denoted by $A \mapsto A^t.$

\item  When $n = 2m$ and $\mathcal
A$ has a negation map, there is another involution, called the
\textbf{symplectic} involution $(s)$, given by $\(\begin{matrix}
                A_{11} & A_{12}\\
                A_{21} &  A_{22}
               \end{matrix}\)^s = \(\begin{matrix}
               A_{22}^t & (-)A_{12}^t\\
                (-)A_{21}^t &  A_{11}^t
               \end{matrix}\)$, where the $A_{ij}\in M_m( \mathcal
A).$
               \end{enumerate}
\end{example}
%

 \begin{rem} We can define $\preceq$-orthogonal
matrices via the condition $(I,(\zero))\preceq A A^t, A^tA , $ and
thereby define the $\preceq$-orthogonal monoids, special
$\preceq$-orthogonal monoids, and $\preceq$-symplectic
monoids.\end{rem}


%

Involutions are studied in terms of   \textbf{symmetric} and
\textbf{antisymmetric} elements, given respectively as
$$ \{ a \in \mathcal
A: a^* = a\}, \qquad \{ a \in \mathcal A: a^* = (-)a\}. $$

\begin{lem}\label{symantisym} Define $(\mathcal A,*)^+: =  \{  a^* + a : a \in \mathcal A\}$
and $ (\mathcal A,*)^-: =  \{  a^* (-) a : a \in \mathcal A\}$.
These sets respectively are symmetric and antisymmetric.\end{lem}
\begin{proof} $( a^* + a )^* = a + a^* = a^* + a,$
 and $(a^* (-) a )^* = a + ((-)a)^* = a(-)a^* = (-)(a^* (-) a ).$
\end{proof}


\subsubsection{Involutions under symmetrization}$ $

\begin{prop}\label{symmetinv}If $\mathcal A$ has an involution $(*)$, then the
symmetrization $\hat {\mathcal A}$   has an involution  given by
$$(a_0, a_1)^* = (a_0^*, a_1^*).$$ The symmetric elements are
$\{(a_0, a_1):$ each $a_i \in \mathcal A$ is symmetric$\}$. The
antisymmetric elements are $\{ (a_0,a_0^*) : a_0 \in \mathcal A\}.$
\end{prop}\begin{proof}  The first two assertions are seen by matching components. For the last
assertion, we have $(-)(a_0, a_1)= (a_1, a_0),$ whereas $(a_0,a_1)^*
= (a_0^*, a_1^*)$,
 so matching components  in the antisymmetric case shows that $a_1=
a_0^*$.
\end{proof}

\begin{cor} The sets of symmetric and antisymmetric elements of $\hat {\mathcal A}$ are precisely
$(\hat {\mathcal A},*)^+$ and $(\hat {\mathcal A},*)^-$
respectively. 
\end{cor}


\subsection{Localization of semiring triples}$ $

 We  use the  standard technique of commutative localization, cf.~\cite{Bo}, to pass
 from the case that $\tT$ is a commutative monoid to the case that  $\tT$ is an Abelian group. (We
defer noncommutative localization for future work.)

Assume that $S$ is a submonoid of $\tT$.
 One defines the equivalence $(s_1,b_1) \equiv
(s_2,b_2)$ when $s(s_1b_2) = s(s_2 b_1)$ for some $s\in S$. In the
cancelative case, we can dispose of $s$. We write  $s^{-1}b$ for the
equivalence class of $(s,b).$ We localize a triple $( \mathcal A,
\tT, (-) )$ with respect to $S$, via

$$s_1^{-1}b_1 s_2^{-1}b_2 = (s_1s_2)^{-1}b_1 b_2 ,\qquad s_1^{-1}b_1 + s_2^{-1}b_2 = (s_1s_2)^{-1}(s_2 b_1 + s_1 b_2).$$
%

The ensuing triple is denoted  $(S^{-1}\mathcal A, S^{-1}\tT, (-)),
$ where $$(-)(s^{-1}b ): =  s^{-1}( (-) b), \quad s\in S.$$

\begin{lem} $ $

\begin{enumerate}\eroman
   \item  $s^{-1}( (-) b) = ((-)s)^{-1}b.$
    \item $(s^{-1}b )^\circ = s^{-1}b ^\circ.$
     \end{enumerate}
\end{lem}
\begin{proof} (i) Cross multiply to get $(-) sb = ((-)s)b.$

(ii) $(s^{-1}b )^\circ = s^{-1}b (-) s^{-1}b = s^{-1}(b (-)b)=
s^{-1}(b^\circ).$
\end{proof}

\begin{prop}\label{local} If  $ (\mathcal A, \tT, (-))$ is a  triple  with unique quasi-negatives, and  $S$ is a  multiplicative
submonoid
 of~$\tT$, then the triple $(S^{-1}\mathcal A,\, S^{-1}\tT
, (-))$ also has unique quasi-negatives.\end{prop}
\begin{proof}  Suppose $s_1^{-1}a_1
$ is a quasi-negative of $s^{-1}a,$ for $a \in \tT$. Then $$ (s s_1)^{-1}(sa_1 +
s_1a) = s_1^{-1}a_1 +
 s^{-1}a \in \mathcal (S^{-1}A)^\circ,$$ implying $sa_1 + s_1a \in \mathcal
 A^\circ,$ and thus $sa_1 = (-) s_1a = s_1((-)a),$ and then $s_1^{-1}a_1
 = s^{-1}((-)a).$ 
\end{proof}

In particular, if $( \mathcal A, \tT, (-) )$ is   cancelative over a
monoid $\tT$, then taking $S = \tT$ we can localize~$\tT$ to the
group~$\tT^{-1}\tT.$ For instance, one might localize $(\mathcal
A[\la_1, \dots, \la_n] , \tT_{\mathcal A[\la_1, \dots, \la_n]} ,
(-))$ (where $\mathcal A$ is a $\tT$-semifield) at the monomials, to
get the \textbf{Laurent triple} $(\mathcal A[[\la_1, \dots, \la_n]]
, \tT_{\mathcal A[[\la_1, \dots, \la_n]]} , (-))$.

\subsection{Tensor products with a negation map, and their
semialgebras}\label{tenpro}$ $

The tensor product is a very well-known process in general category
theory, \cite{Gr,Ka1,Ka2,Tak}, and has been studied in the context
of \textbf{monoidal categories}. Here we consider the tensor product
of module triples and semialgebra triples over commutative
\semirings0. These are described in terms of congruences, as given
for example in~\cite[Definition~3]{Ka2}.

Let us work with a module $V$ over a commutative associative
\semiring0 $C$. If $V$ has a negation map~$(-)$, then we can
incorporate the negation map into the tensor product, defining a
negation map on $V \otimes _C W$ by $(-)(v \otimes w) = ((-) v)
\otimes w.$ When $W$ also has a negation map we define a
\textbf{negated tensor product} by imposing the extra axiom
$$((-)v) \otimes w = v \otimes ((-)w).$$ (This is done by modding
out   the usual congruence defining the tensor product  by the
congruence generated by all elements $((-)v \otimes w,\ v \otimes
(-)w)$ .) From now on, the notation $V \otimes W$ includes this
negated tensor product stipulation, and $C$ is understood.
%
%
%

\begin{rem}\label{tensorprop} One can easily prove the following facts,
modifying for example \cite[Chapter~18]{Row08}:

\begin{enumerate}\eroman\item If $f_i : V_i \to W_i$ are module maps then there is
a unique map $f_1 \otimes f_2 : V_1 \otimes V_2 \to W_1 \otimes W_2$
satisfying $$(f_1 \otimes f_2)(v_1 \otimes v_2) = f_1(v_1) \otimes
f_2 (v_2).$$

\item The tensor product $(\mathcal A, \tT, (-))\otimes (\mathcal A', \tT', (-)')$ of triples is a triple
 $$(\mathcal A \otimes \mathcal A', \{ a
\otimes a ' : a  \in \tT, a ' \in \tT'\}, (-)\otimes 1_{\mathcal
A'}).$$
\begin{itemize}
\item  This definition is suited towards ``multilinear'' algebra,
where $\tT_{\mathcal A \otimes \mathcal A'}$ is a set of rank 1
tensors.

\item We call $ a
\otimes a '$ a \textbf{simple tensor}.

\item  If    $F$ is a commutative associative semialgebra over
$C$, then $\underline{\phantom{w}}\otimes _C F$ yields a functor
from semialgebras over $C$  to  semialgebras over $F$, which extends
to triples. (In particular, this holds when $F$ is the
symmetrization of $C$.)
\item  The tensor product is a semiring when $\mathcal A$ and $\mathcal A'$ are semirings.
This enables us to view matrices and polynomials in terms of tensor
products.
\end{itemize}
\end{enumerate}
\end{rem}


 Next, as usual, given a module $V$ over  $C$, one defines
$V^{\otimes (1)} = V,$ and inductively $$V^{\otimes (k)} = V
\otimes V^{\otimes (k-1)} .$$ From what we just described, if  $V$
has a negation map $(-)$ then  $V^{\otimes (k)}$
also has a natural negation map.  
%

 Now define the \textbf{negated tensor semialgebra} $T(V) =
\bigoplus _n V^{\otimes (n)}$ (adjoining a copy of $C$ if we want to
have a unit element), with the usual multiplication. If $V$ has a
negation map then so does $T(V)$, induced from the  negation maps on
$V^{\otimes (k)}$; writing $\tilde a_k = a_{k,1} \otimes \dots
\otimes a_{k,k},$ we put $$(-) (\tilde a_k) = ((-) a_{k,1}) \otimes
\dots \otimes a_{k,k}.$$ 
%


%

One can view the polynomial semialgebra of \S\ref{poly} as a negated
tensor semialgebra, where $V = \mathcal A \la.$

\section{Metatangible triples and their systems}\label{surp22}$ $

This section explores metatangible triples $(\mathcal A, \tT, (-))$
(Definition~\ref{metatan}) and metatangible systems.  Decisive
results for metatangible triples are available, which cover the
tropical applications, but only hold for certain hyperfields.
Eventually we show in Theorem~\ref{circsurp} that $\preceq_\circ$
finishes out the system (although there are other possible
surpassing relations). Many of our arguments involve the height of
an element, from~\S\ref{height10}. Here are some of the main
results.

Every $(-)$-bipotent triple clearly is metatangible. Perhaps as a
surprise, conversely, by Theorem~\ref{twoexamples}, a metatangible
triple either is $(-)$-bipotent (with the ensuing tropical flavor)
or satisfies $e' = \one$ (in which case $e^\circ = e$), which
happens in classical algebra.

The elements of metatangible triples have a surprisingly nice form
to be given in Theorem~\ref{matrsys}, which enables us to prove
various nice properties 
although there also are some annoying
counterexamples.

The heights of elements tie in with the characteristic
 of a triple,  in describing
surpassing relations in Theorem~\ref{circsurp}. Theorem~\ref{sign16}
enables us to describe the symmetrized semialgebra in terms of
classical considerations about sums of squares. This pertains to
``real'' groups of tangible elements, in Proposition~\ref{sign1}.
%
%

Theorem~\ref{mon11} largely classifies metatangible systems, in the
sense that it describes the properties of all of the major examples.
But of course a full classification requires all of the examples,
some of which are rather esoteric and not present here. Some more
examples are given in \cite{Row16v}.

\subsection{Basic properties of  metatangible
 triples and systems}$ $

The key property is:

\begin{lem}\label{nummatch0} A triple $(\mathcal A, \tT, (-))$ with unique quasi-negatives is  metatangible iff
$\tT + \tT \subseteq \tT \cup \mathcal A^\circ.$
\end{lem}
\begin{proof} $(\Rightarrow)$ Immediate from the definition.

$(\Leftarrow)$ If $a' +a \notin \tT $, then $a' +a \in \mathcal
A^\circ$, so $a'  = (-)a$ by uniqueness of quasi-negatives.
\end{proof}

\begin{rem}\label{negmatch1} By Lemma~\ref{negmatch}(i), in a metatangible triple, if  $a_1 \ne
a_2$ in $\tT$ and   $a_1 + a_2 =  a_2$, then $a_1 (-) a_2 = (-) a_2$
and $a_2 (-) a_1 =  a_2$.\end{rem}

\subsection{Metatangible versus $(-)$-bipotent}$ $

 One may
ask whether metatangible systems are necessarily $(-)$-bipotent.
Here is a surprisingly strong observation.

\begin{lem}\label{bip1}
One of the following must hold, for $a,b \in \tT$ in a metatangible
triple:
\begin{enumerate}\eroman  \item $ a = (-)b$.
\item $a+b = a$ (and thus $a^\circ + b = a^\circ$).
\item $a^\circ + b
= b.$ \end{enumerate} We cannot have both (ii) and (iii)
simultaneously.
\end{lem}
\begin{proof} Assume that $a\ne (-)b$, implying $a+b \in \tT.$ If
$a+b \ne a$, then $a^\circ + b  = (a+b) (-)a \in \tT.$ Hence $b =
a^\circ +b$ by Definition~\ref{precedeq07}(i,v).

For the last assertion, if  $a+b = a$ then  $a^\circ + b =  a (-) a
= a^\circ ,$ and if also $a^\circ + b = b$ then $ a^\circ = a^\circ
+ b = b \in  \tT \cap \tT ^\circ = \{ \emptyset \}.$
\end{proof}

%

Examples of a non-bipotent, metatangible triple are the classical
$(\mathbb Z,+),$  or $(\mathbb Z/n,+),$ taking $\tT = \mathcal A
\setminus \{ \zero\},$ but they all have a classical flavor, because
of the next observation.

\begin{lem} \label{negmatch5} Suppose in a metatangible triple $(\mathcal A, \tT, (-))$ that  $a_i \in \tT$
with $a_2 \ne (-) a_1$ and $a_3: = a_1+a_2 \notin \{ a_1, a_2\}.$
Then $a_{1} = a_{3} (-)a_2$ and $a_{2} = a_{3} (-)a_1$. Let
$\mathcal A'$ be the
 sub-semigroup generated by $\{ (\pm)a_1,(\pm) a_2,(\pm)
 a_3\}$, and $\tT'= \mathcal A' \cap \tT$. Then
 $(\mathcal A',\tT', (-))$ is a triple  in which $(a_1+ a_2+a_3)^\circ$
acts as the $\zero$ element, and $a_i = a_i e'.$ The unique
quasi-negative of $a_i$ is $a_{i+1}+a_{i+2}$, subscripts modulo 3.

%
%
%

\end{lem}
\begin{proof} $a_{1} = a_{3} (-)a_2$ and
$a_{2} = a_{3} (-)a_1$ by Lemma~\ref{negmatch}(ii). Hence
$$ a_{3} =  a_1+a_2 =  a_{3} (-)a_2 +a_2 =  a_{3} + a_2^\circ,$$ and
likewise $ a_{3} =   a_{3} + a_1^\circ,$ implying $a_{3} ^\circ =
a_{3}^\circ + a_2^\circ = a_1^\circ$ and likewise $ a_{3} ^\circ =
a_2^\circ.$ Now $a_3 + a_{3} ^\circ = a_3 + a_{1} ^\circ = a_3$, so
all the $a_i^\circ$ act as the zero element for $a_3,$ and
symmetrically also for $a_{1}$ and $a_{2}$. Finally, $$(a_1+
a_2+a_3)^\circ = a_1^\circ  + a_2^\circ + a_3 (-)a_3 =  a_3 + (-)a_3
= a_3^\circ,$$ and $a_i = a_i (-) a_i + a_i = a_i e'.$

The last assertion holds since $(a_1+ a_2+a_3)^\circ$ is the only
quasi-zero.
\end{proof}%
%

\begin{rem}  We say that $\{a_1, a_2, a_3\}$ satisfy \textbf{the trio
property} when $a_2 \ne a_1$ in Lemma~\ref{negmatch5}.
 When $(-)$ has the first kind, the trio
property is automatic and $\tT' = \{ a_1, a_2, a_3\}$ and $\mathcal
A' = \tT' \cup \{(a_1+ a_2+a_3)^\circ \}$.

So suppose for the rest of this remark that $(-)$ has the second
kind in Lemma~\ref{negmatch5}.
 We could have the degenerate condition $a_2
= a_1$  in Lemma~\ref{negmatch5}. Now  $\tT'$ consists of multiples
of $a_1$, but could conceivably be infinite.

 Bypassing the previous paragraph, assuming that $a + a= a$ for
all $a\in \tT,$ in other words
 that $\mathcal A$ is idempotent. Then $\tT' = \{
(\pm)a_1,(\pm) a_2,(\pm) a_3\}$ since $a_i = a_i e',$ and again
$\mathcal A' = \tT' \cup \{(a_1+ a_2+a_3)^\circ \}$.
\end{rem}

We do have a non-classical example for pseudo-triples:

\begin{example}\label{meta3}
In Example~\ref{nontang}(v), take $L = \{ 0, 1, \ell, \ell +1\}$ to
be the finite field of 4 elements. Although only a pseudo-triple
since $1$ does not generate $L$, the layered algebra $\mathcal A = L
\times \tG$ does satisfy the metatangibility condition and is of
first kind, but not $(-)$-bipotent
 (since $(1,a) + (\ell,a) = (\ell +1,a))$,   satisfying $e' = \mathbf 3 = \one.$ 
\end{example}

Summarizing, we have:

 \begin{thm}\label{twoexamples} Any  cancelative metatangible unital triple  $(\mathcal A,\tT,(-))$  satisfies one of the following cases: \begin{enumerate}\eroman  \item   $\mathcal A$
 is $(-)$-bipotent. 
 \item $e' =  \one,$ with one of the following two possibilities.
 \begin{enumerate}  \item  $(-)$ is of the first kind, with $(\mathcal A,+)$ of characteristic 2. (In other words
 $e' = \mathbf 3 = \mathbf 1 = \one.)$  In this case, $A$ has height
 $\le 2.$
  \item   $(-)$ is of the second either with $(\mathcal A,+)$ of finite
  characteristic or with  $\{ \mathbf m : m \in \Z \}$ all distinct.
  \end{enumerate}\end{enumerate}
 \end{thm}
\begin{proof}  If $\mathcal A$
 is not $(-)$-bipotent, we  have $a, a' \in
\tT$ with $a' \neq (-) a$ and  $a+a' \notin \{ a, a'\}.$ By
Lemma~\ref{negmatch5}, $a = a+ a(-)a;$ canceling $a$ yields $e' =
\one .$

If $(-)$ is of the first kind, then $ \mathbf 1 = e' = \mathbf 3 ,$
so $\mathcal A$ has characteristic 2.

If $(-)$ is of the second kind, then
  we conclude with  Proposition~\ref{char77}.
%
\end{proof}

 \begin{cor}\label{twoexamples2}  Suppose $(\mathcal A,\tT,(-))$ is a cancelative unital metatangible $\tT$-semifield
 triple of the second kind. Then~$\mathcal A$
 is
  $(-)$-bipotent iff $\mathcal A$
 is
idempotent.
 \end{cor}
\begin{proof}
$(\Rightarrow)$ Since $(-)\one \ne \one$, we have   $\one + \one
\in \{\one, \one\}= \{\one\},$ so $a+a =  a(\one + \one) = a\one =
a.$

 $(\Leftarrow)$ $e' = \one (-)\one  = e,$ so  $e' \ne \one,$ implying $\mathcal A$
 is
  $(-)$-bipotent by Theorem~\ref{twoexamples}.
\end{proof}

\begin{MNote}
This dichotomy of $(-)$-bipotence and  $e' =  \one$ flavors our
entire discussion.

There are examples for each of these conclusions.
\begin{itemize}\item
The classical triple satisfies $e' = \one, $ but is not
$(-)$-bipotent.
\item  The triple of the
standard supertropical algebra is $(-)$-bipotent of first kind,
  satisfying $e' = e,$ and the same holds for the symmetrized algebra.
\item The triple of the
modified symmetrized algebra  of Example~\ref{surpsym} is
$(-)$-bipotent of second kind.

\item The ELT triple satisfies both conditions and is of the second kind.

\item Example~\ref{nontang}(i) is $(-)$-bipotent of the first kind, even
failing $e' \in \tT \cup \tT^\circ.$ \end{itemize}
\end{MNote}

\medskip


\subsection{Computations in a metatangible triple $(\mathcal A, \tT,
(-))$ via $\tT$}$    $

\begin{lem}\label{nummatch2} Suppose $a_i \in  \tT$ with $\sum _{i=1}^{k-1} a_i \in  \tT$ but $\sum _{i=1}^{k} a_i \notin \tT$. Then $\sum _{i=1}^{k-1} a_i = (-)a_k.$
\end{lem}
\begin{proof} $\sum _{i=1}^{k} a_i \in
\tT$ unless $\sum _{i=1}^{k-1} a_i= (-)a_k$.
\end{proof}

 \begin{prop}\label{nummatch3} Suppose $\sum _{i=1}^t a_i \succeq \zero.$ Then for some $k<t,$
$a_k = (-)\sum _{i=1}^{k-1} a_i$. \end{prop}
\begin{proof} Take $k<t$ minimal satisfying Lemma~\ref{nummatch2}.
\end{proof}

\subsubsection{The characteristic of a  metatangible
 triple}$ $

We continue from \S\ref{chartri}. Define $\mathbf Z = \{ (\pm)
\mathbf a: a \in \Net\}$.


We can make this more explicit.

\begin{prop}\label{use23} Suppose that $(\mathcal A,\tT,(-))$ is a unital
metatangible
 triple. Either $\mathbf N  \subseteq \tT$  or $\mathcal A$
 has characteristic $k$ for some $k \ge 1$, in which case one of the following
possibilities holds:
\begin{enumerate}\eroman  \item  $k= 1,$ i.e,  $\mathcal
 A$ is idempotent.

\item   $k>1$, with $\mathbf k \in \tT.$

\item   $k>1$, with $\mathbf
{k-1} = (-)\one,$ and $\mathbf {k} = e.$

\item $(-)$ is of the first kind and $k$ is even.
\end{enumerate}


Moreover, if $(-)$ is of the second kind, then either $\mathcal A$
is idempotent or
 $\mathbf {k+1}= \one$, implying the characteristic of $\mathcal A$ divides $k$.
\end{prop}

\begin{proof}
First assume that $\mathbf n \ne \one$ for each $n>\one$. By
induction, we may assume that $\mathbf n \in \tT$, and then each
$(\mathbf {n+ 1}) = \mathbf n + \one \in \tT$ and we conclude that
$\mathbf Z \subseteq \tT$ by negating.

Thus $\mathbf {n+ 1} = \one$ for some $n$, implying  $\mathcal A$
 has some characteristic $k$. We are done unless $k>1.$ If $\mathbf k \in \tT$ then we have
 (ii), so assume that $\mathbf k \notin \tT$.
%
%

Assume presently that $(-)$ is of the second kind. If $\tT$ is also
$(-)$-bipotent then $\one +\one = \one,$ so we have~(i). Otherwise
$\mathbf {2}(-) \one =  e' = \one$ by Theorem~\ref{twoexamples},
implying by induction that $\mathbf {j}(-) \one = \mathbf {j-1}$ for
all $j>1$.  Take $k'  $ minimal such that $\mathbf {k'}\notin \tT$.
(Clearly $k'
>1$, and $k' \le k$.) Then
by Lemma~\ref{nummatch2}, $\mathbf {k' -1} = (-)\one$, $\mathbf
{k'}= e$,   and  $\mathbf {k'+1}= e' = \one,$ so $k' = k$
 and we have (iii). 

So we may assume that $(-)$ is of the first kind. If $k$ is odd then
$\one = \mathbf {k+1} = \frac{k+1}2 e \in \mathcal A^\circ,$ which
is a contradiction, so $k$ is even.
\end{proof}


\subsubsection{Comparing $\mathcal A$   to
 $\tT$}$ $

\begin{lem}\label{drive} If $(\mathcal A,\tT,(-))$ is a metatangible
triple, then $\tT + \mathcal A^\circ = \mathcal A$.
\end{lem}
\begin{proof}
Write $b = \sum _{i=1}^t a_i + c^{\circ}$ with $a_i \in \tT,$ and
$t$ minimal. (Clearly $t$ is at most the height of $b,$ since one
could take $c = \zero.$) Then $a_t = (-) a_{t-1},$ since otherwise
we could take $a_t + a_{t-1} \in \tT$ and lower $t.$ But then $b =
a_{t-1}^\circ + c^{\circ} + \sum _{i=1}^{t-2} a_i$ and we apply
induction on $t$, reducing to the case $t=1$.
%
%
\end{proof}
%
Actually we are aiming for the condition of whether $\tT \cup
\mathcal A^\circ = \mathcal A.$

\begin{lem}\label{drive2} Suppose the unital triple $(\mathcal A,\tT,(-))$ is   metatangible.
\begin{enumerate}\eroman \item If $(-)$  is of the
second kind, then $\tT \cup \mathcal A^\circ = \mathcal A.$
\item If $(-)$ is
of the first kind, then $ \mathcal A = \cup \{ \mathbf m \tT:  m \in
\Net\}.$
\end{enumerate}
\end{lem}
\begin{proof} Write $b = \sum _{i=1}^t a_i$ with each $a_i \in \tT,$ $t$ minimal. This
implies $a_i = (-)a_j$ for each $i,j$ since otherwise we could
combine them to an element of $\tT$ and reduce $t$. For $t
>2$, when $e' = \one$, we can replace $a_1(-)a_1+a_1$ by $a_1$, so
again reduce $t$.  Hence we are done by Theorem~\ref{twoexamples}
unless $\mathcal A$ is $(-)$-bipotent. But then for second kind we
get $b = (\pm a_1)$ or $b = a_1^\circ,$ yielding $(i).$

 Thus we may assume that  $a_1 =(-)a_1$.   we
obtain (ii) using induction on height.
\end{proof}

 Let  us consider the
possibilities for $e'$.

 \begin{prop}\label{ematch3}  One of the following must hold in a
 metatangible unital
triple $\mathcal A$:
\begin{enumerate}\eroman  \item $e' \in \tT,$ and then $e' = \one.$
  \item $e' \in \tT^\circ,$ and then $e' = e.$
\item $e' $ has  height $\ge 3$, i.e., $e'\notin \tT \cup \tT^\circ,$  and then $(-)\one = \one,$ with $e' = \mathbf 3.$
\end{enumerate}
\end{prop}
\begin{proof} By Lemma~\ref{uniq77} we may assume that $e' =
\mathbf 2 (-) \one $ is not tangible, or else we have (i) by  unique
quasi-negation. But then if $\mathbf 2$ is tangible then $\mathbf 2
= \one,$ implying $e'=e.$

Thus we may assume that $\mathbf 2$ is not tangible, so $(-)\one =
\one.$ If $e' = a^\circ = 2a$ for $a \in \tT$ with $a \ne \one,$
then writing $b = \one + a \in \tT,$ we have $ b +a  = \one + 2a =
\one + e' = e +e  \in \mathcal A^\circ ,$ and thus $b =a $. But then
$\one + b = \one + a  =  b,$ so $3 a = e'+a =\one + \one  + b= \one
+b  = b = a ,$ implying $\one = \mathbf 3 = e'.$

We are left with (iii).
\end{proof}

\subsubsection{Triples of height 2}$ $

Our earlier considerations give  decisive results for height 2,
which include both the supertropical and symmetrized constructions.

 \begin{prop}\label{ematch5} The following assertions are
 equivalent for a unital triple $(\mathcal A, \tT,
(-))$ (not presumed a priori to be metatangible):
\begin{enumerate}\eroman\item $\tT \cup  \tT^\circ = \mathcal A,$
\item $\mathcal A$ is  metatangible  of height 2,
\item  $\mathcal A$ is  metatangible with $e' \in \{ \one, e\}.$
 \end{enumerate}
  \end{prop}
\begin{proof} $( (i) \Rightarrow (ii)) $ A fortiori.

$((ii) \Rightarrow  (iii))$ We exclude (iii) in
Proposition~\ref{ematch3}.

$((iii) \Rightarrow (i) )$ The assertion is that $ \mathcal A$ has
height 2.  We need to show that any sum of $\tT \cup  \tT^\circ$ and
elements of $\tT$ remains in  $\tT \cup  \tT^\circ$, and it is
enough to show that if $a \in \tT$ and $b \in \tT \cup \tT^\circ$,
then $a +b \in \tT \cup \tT^\circ$.  The assertion is clear if $b
\in \tT,$ and follows from Lemma~\ref{bip1} unless $b = a ^\circ$,
in which case
$$a+b = ae'  \in \{   a,  ae\} .$$
\end{proof}

 \begin{cor}\label{ematch55} Any metatangible triple of second kind
 has height 2.
  \end{cor}
  \begin{proof} By Proposition~\ref{ematch5} since $e' = e.$
\end{proof}

 An example of a non-metatangible hypersystem of height 2
is given in
Example~\ref{Basicexamples}. 

\begin{prop}\label{msums1} Every unital metatangible system of height 2 is strictly negated, when $\preceq =
\preceq_\circ.$
\end{prop}\begin{proof} Suppose $d(-)c \succeq \zero.$ The result is immediate when $c,d \in
\tT^\circ$,  and clear by unique quasi-negation for $c,d \in \tT$
(since then $c=d$),
 so by Proposition~\ref{ematch5} we may assume that $c \in \tT$ and $d,\ d(-)c \in \tT^\circ$. Write $d
= a^\circ$  for $a\in \tT$.

We cannot have $d(-)c=c$ since $\tT  \cap \mathcal A ^\circ =
\emptyset$. We are done if $d =d(-)c =c + a^\circ.$    Thus,
applying Lemma~\ref{bip1} for $b=c,$ we may assume that $c =
(\pm)a$. Now $e' \ne \one$ since otherwise $c = c + c^\circ = c +
a^\circ = c+d$. Hence $e'=e$ by Proposition~\ref{ematch5}(iii),
implying $c \preceq ce' = ce = a^\circ =d.$
\end{proof}

 \begin{prop}\label{twoexamples4}  Any metatangible
 triple $(\mathcal A,\tT,(-),\preceq)$ over a group $\tT$ is
 either
  $(-)$-bipotent or of height~2.
 \end{prop}
\begin{proof}
We are done by  Theorem~\ref{twoexamples} unless $e'=\one$, so
Proposition~\ref{ematch3} implies $(\mathcal A,\tT,(-),\preceq)$ has
height~2.
\end{proof}

\subsubsection{The natural pre-order on $\mathcal A^\circ$}$ $

\begin{lem}\label{bip4}
Any  triple   has the partial pre-order $<^\circ$, given by $a_1 \le
^\circ a_2$  iff  $a_1 ^\circ = a_2^\circ$ or $a_1 ^\circ +
a_2^\circ = a_2^\circ.$  $a_1 \le ^\circ a_2$ iff  $a_1 \le ^\circ
(-) a_2$, iff $(-)a_1 \le ^\circ a_2$.
\end{lem}
\begin{proof} Suppose  $a_1 ^\circ +
a_2^\circ = a_2^\circ$ and  $a_2 ^\circ + a_3^\circ = a_3^\circ.$
Then $$a_1 ^\circ + a_3^\circ = a_1 ^\circ  + (a_2 ^\circ +
a_3^\circ)  =  (a_1 ^\circ  + a_2 ^\circ) + a_3^\circ =  a_2^\circ +
a_3^\circ = a_3^\circ.$$ The other verifications are patent.
\end{proof}

\begin{prop}\label{bip3}$ $ For any metatangible triple, with $a_i
\in \tT,$
\begin{enumerate}\eroman \item $a_1 ^\circ + a_2^\circ \in
\{ a_1^\circ, a_2^\circ, \mathbf 2 a_1 ^\circ\},$ the last
possibility occurring when $ a_1 = (\pm) a_2$.

\item In particular,  $\le ^\circ$ is a surpassing $\circ$-PO.

\end{enumerate}
\end{prop}
\begin{proof} (i) by Lemma~\ref{bip1}, applied to both $a_2$ and $(-)a_2$. (The first two
possibilities arise when $ a_1 \ne (\pm) a_2,$ and the last,
$\mathbf 2(a_1
 (-)a_1 ),$ when $ a_1 = (\pm) a_2$).

(ii) Follows at once from Definition~\ref{precedeq07}.

\end{proof}
%


\begin{prop}\label{C1}   Any metatangible system $(\mathcal A,
\tT, (-), \preceq)$ has the congruence $$\Cong = \{ (b_1,b_2):
b_1^\circ = b_2^\circ\},$$ and  the corresponding system of
$\mathcal A/\Cong$ is $(-)$-bipotent of first kind.
\end{prop}
\begin{proof} If $(b_1,b_2), (b_1',b_2')\in \Cong$ and $a \in \tT$ then
$$(a b_1)^\circ = a {b_1} ^\circ =  a{b_2} ^\circ = (a b_2)^\circ,
\qquad   (b_1 +{b_1'})^\circ =  b_1 ^\circ + {b_1'} ^\circ = b_2
^\circ + {b_2'} ^\circ,$$ implying $\Cong$ is a congruence, modulo
which $(-)b$ becomes $b$ since $b^\circ = ((-)b)^\circ .$
Furthermore, Lemma~\ref{bip1} yields $(-)$-bipotence since $a^\circ
+ b = b$ implies $(a + b )^\circ = a^\circ + b^\circ = b^\circ.$
\end{proof}

\subsection{Uniform elements and height}$ $

A cancelative metatangible triple $\mathcal A$ is called
\textbf{exceptional} if it is of first kind, of height $>2$.
Case~(ii) of~Theorem~\ref{twoexamples} is ruled out, so $\mathcal A$
is $(-)$-bipotent. The main example is the layered triple
(Example~\ref{AGGGexmod1}). Exceptional triples also are a source of
weird counterexamples, as in~\cite[Examples~7.40 and
  ~7.42]{Row16v}.

%
%

\begin{defn}\label{unif1} A nonzero element $c \in \mathcal A$ of height $m_c \in \Net$    is \textbf{uniform}
if for some element $c_\tT\in \tT$,   one of the following three
possibilities occurs:
\begin{enumerate}\eroman
 \item Type 1: $m_c=1,$ i.e.,
$c = c_\tT\in \tT$,
  \item Type 2: $c = c_\tT^\circ$.
 \item Type 3: $m_c\ge 3,$ and
the triple $(\mathcal A, \tT, (-))$ is exceptional, with $\mathbf 3
\ne \mathbf 1.$
\end{enumerate}
We call this the \textbf{uniform presentation} of $c$.
\end{defn}

\begin{MNote}
The uniform presentation  enables us to reduce many proofs to the
tangible case. Type 2 is the only case for which we do not have $c =
m_c c,$ and
 requires separate treatment. If we can show that some element has
Type 3, the negation map is of first kind, and the notation becomes
simpler.

 The uniform presentation need not be unique. For
example, in a ring,   all quasi-zeros are equal (to $\zero$). But we
will see in Theorem~\ref{uniq71} that quasi-zeros are the only
elements with non-unique presentation in $(-)$-bipotent triples.
 \end{MNote}

\begin{thm}\label{matrsys} Every element of a cancelative metatangible triple $(\mathcal A,\tT,(-))$
is uniform.

%

\end{thm}\begin{proof} 
%
%
%
%
Let $m$ be the height of $c$. We may assume that $m \ge 2.$
 If $m =2 $ and $a_1  \ne (-) a_2$, then $c
 \in \tT$, so $a_1  = (-) a_2$ and $c = {a_1}^\circ.$


For $m\ge 3$, if some $a_i \ne a_j$ then $a_i+a_j \in \tT$,
contradicting definition of height. Hence all $a_i = a_1.$ If
$\mathbf 3 = \mathbf 1$ then we replace $a_1+a_1+a_1$ by $a_1$,
again contradicting the definition of height. If $(-)$ is of second
kind, then $a_1+a_1$ is tangible, again  contradicting the
definition of height. Hence $(-)$ is of first kind, so $\mathcal A $
is exceptional.
\end{proof}


 \begin{cor}\label{sumunif} There are the following possibilities for  $c+ d$
 in a cancelative metatangible triple $(\mathcal A,\tT,(-))$, where $m_c \le m_d$:
  \begin{enumerate}\eroman
    \item $c$, with  $c_\tT +d_\tT = c_\tT$.
   \item $c$, where  $m_c = 1$ and $m_d =2$.
    \item $d$, with  $c_\tT +d_\tT = d_\tT$.
  \item $d$, where  $m_d \ge 2$ and  $c +d_\tT = d_\tT$.
\item $c_\tT +d_\tT \in \tT$, where $c= c_\tT $ and $d = d_\tT$ but $c \ne (-)
d.$
\item $(c_\tT +d_\tT)^\circ \in \tT^\circ$, where $c= c_\tT ^\circ $ and $d = d_\tT^\circ$ but $c_\tT
 \ne (-) d_\tT.$
     \item $2 c_\tT ^\circ $, where $c  = d$ and $m_c =
  2.$
\item $ c_\tT ^\circ$, where $c = c_\tT = (-) d_\tT = (-)d.$
     \item $( m_c +  m_d) c_\tT $, where $c_\tT = d_\tT $.
   \end{enumerate}
     \end{cor}
\begin{proof}  If $c_\tT +d_\tT \in \{ c_\tT, d_\tT\}$ we get
(i), (iii) by iterating, noting   by Remark~\ref{negmatch1}  to
handle Type 2, that for tangible elements  $a+b =b$ also implies
$(-)a + b = b.$ Thus, we assume that $c_\tT +d_\tT \notin \{ c_\tT,
d_\tT\}$.

For the next three paragraphs we suppose that $m_c = 1.$ If $m_d =2$
then $d = d_\tT^\circ$ and then by~Lemma~\ref{bip1}, we have three
cases. Either  $c + d_\tT= d_\tT,$ in which case $c+d =d$, which is
(iv), or $c + d_\tT= c,$ which is (i), or excluding these yields $c
= (-) d_\tT,$  so  $c+d = c(-)c+c$. Hence $c+d (-)c \in \mathcal
A^\circ.$ If $c+d \in \tT$ then $c+d = c$ by unique quasi-negation,
yielding (ii). That leaves us with $c = (-)d \in \tT\cap \mathcal
A^\circ = \emptyset,$ a contradiction.

 If $d \in \tT$ then $c+d
\in A^\circ$ implies $d = (-)c,$ yielding (viii); if $c_\tT = d_\tT
$ we have (v).

If $m_d \ge 3$ then we write $d = 2d + (m_d -2)$ and apply the
previous two paragraphs to get $c$ which is (i), or $d$ which is
(iv), or we have $c_\tT =d_\tT ,$ yielding (ix).

Hence we may assume $m_c \ge 2.$
 If $c = c_\tT^\circ$ and $d =
d_\tT^\circ$ then $c+d = (c_\tT+d_\tT)^\circ,$ yielding (v) and
(viii).

If $c = c_\tT^\circ$ and $ m_d \ge 3$, then we again apply
Lemma~\ref{bip1} to $c_\tT^\circ$ and $d_\tT,$ yielding (i), (iv),
or (ix).

For $m_c \ge 3,$ if we do not have (i) or (iii), then $c_\tT =
(-)c_\tT = d_\tT,$ yielding (ix).
\end{proof}
%
%

 \begin{cor}\label{sumcirc} $c+d \in \mathcal A^\circ$ in a cancelative unital metatangible triple $(\mathcal A,\tT,(-))$
  precisely in the following situations, again according to the Type of $c$:
 \begin{enumerate}\eroman
 \item   Type 1.
  \begin{enumerate}\eroman
 \item
 $d = (-)c$ and $c+d = c^\circ.$
  \item $d =  d_\tT^\circ$ has
 Type 2 and:

\indent (1) $d_\tT =(\pm) c$ so $d =   c^\circ ,$ and $c+d =   ce'$
(when $e'=e$).

\indent (2) $d_\tT +c=d_\tT$, so  $c+d = d.$

 \item $d =  m_d d_\tT$ has
 Type 3 and:

\indent (1) $d_\tT =c$ so  $c+d = (m_d+1) c$ (when this is in $
A^\circ$).

\indent (2) (for $d \in  A^\circ$) $d_\tT +c=d_\tT$, so  $c+d = d$.

   \end{enumerate}
  \item Type 2.
  \begin{enumerate}\eroman
 \item $d $ has Type 1 with  $c+d = c.$
 \item  $d = d_\tT^\circ$ has Type
  2 with $c+d = (c_\tT+d_\tT)^\circ.$

 \item $d  $ has Type 3.

\indent (1) $d_\tT = (-)c_\tT.$ If $(-)$ is of second kind, then
$c+d = c$.

If $(-)$ is of first kind, then $c = 2d_\tT$ and $c+d = (u+2) b,$
which is in $A^\circ$ if $\mathbf {u+2} \in A^\circ$.

\indent (2) (for $d \in \mathcal A^\circ$) $c_\tT^\circ +d_\tT =
d_\tT.$ Then $c+d = d.$

\indent (3)  $c_\tT+d_\tT = c_\tT.$ Then $c+d = c.$
   \end{enumerate}

    \item Type 3. (Then $(-)$ is of first
    kind.)

     \begin{enumerate}\eroman
      \item  $d $ has Type
  1, i.e., $m_d = 1.$

    \indent (1) $d_\tT = c_\tT$, with $(m_c +1)d \in \mathcal A^\circ$.

\indent (2) $d_\tT +  c = c,$ with $ m_c c \in \mathcal A^\circ$.

      \item  $d $ has Type
  2.

  \indent (1) $d_\tT = c_\tT$. Then $d = 2c_\tT$ and $c+d = (m+2) c_\tT.$

\indent (2)  $c_\tT+d_\tT = d_\tT$. Then $c+d = d.$

\indent (3) (for $c \in  A^\circ$) $c_\tT+d = c_\tT$. Then $c+d =
c.$

 \item $d $ has Type 3.

 \indent (1) $d_\tT = c_\tT$. Then $c+d = (m_c+m_d)  c_\tT,$ which could
  be in $A^\circ.$

\indent (2)  (for $d \in  A^\circ$) $c_\tT+d_\tT = d_\tT$. Then $c+d
= d.$

\indent (3)  (for $c \in  A^\circ$) $c_\tT+d = c_\tT$. Then $c+d =
c.$

    \end{enumerate}
      \end{enumerate}
  \end{cor}
\begin{proof}
We go through the cases of Corollary~\ref{sumunif}. (v) is
impossible.

 If $c$ has Type 1 then  (i), (ii), (vi),  and (vii) are impossible;
 (iii), (iv) yield (b)(2) or (c)(2); (viii) yields (a);  (ix) yields (b)(1)
 or
 (c)(1).

  If $c$ has Type 2 then (ii) and (viii) are  impossible; (i) yields (a)(2) or (c)(3);  (iii) yields (b) or
  (c)(2);  (iv) yields (b) or (c)(2);
     (vi), (vii) yield  (b) ;   (ix)
   yields (c)(1).

  If $c$ has Type 3 then $(-)$   must be of first kind
  and (ii), (vi), (vii), and (viii) are  impossible;  (i) yields
  (a)(2); (iii), (iv) yield  (b)(2) or   (c)(2); (ix) yields
  (c)(1).
\end{proof}

The cases (i)(c)(1), (ii)(c)(1), (iii)(b)(1), and (iii)(c)(1) are
particularly intriguing. A sufficient condition clearly is for
$m_c+m_d$ to be even, and we will consider  necessity in
Corollary~\ref{symunif} below.

 Distributivity  for cancelative $\mathcal A$ can be obtained from metatangibility, cf.~
 \cite[Theorem~7.34]{Row16v}.
%
%
We turn to uniqueness of the uniform presentation.

\begin{thm}\label{uniq71} Suppose  $c = m_c c_\tT$ or $c  = c_\tT ^\circ$ has another
presentation $m_{c'} {c'_\tT}$ or  $ {c'_\tT} ^\circ$  in a
cancelative unital metatangible triple $(\mathcal A,\tT,(-))$.  Then
one of the following occurs taking $m_{c'}$ minimal with respect to
$c'$:
\begin{enumerate}\eroman
    \item  $c_\tT=c_\tT',$ and thus $m_c = m_{c'}.$
  \item   $ m_c = m_{c'} =  2, $ and $c_\tT^\circ=c_\tT'^\circ.$
\item    $m_c = 1$,  $(-)$ is of the second kind, and the triple is not $(-)$-bipotent.
\end{enumerate}
Consequently, the uniform presentation of any element of height $>
2$ is unique.
\end{thm}
\begin{proof} Write $m = m_c$ and $m' =m_{c'}.$ If $m=m' =1$ then $c_\tT = c = c' = c'_\tT,$ yielding
(i). If $m=m' =2$ then we have (ii), so we may assume $m\ne 2.$  If
$m'=1$ then $m \ne 2$  since otherwise $c \in \tT \cap \mathcal
A^\circ = \emptyset.$ If $m=1 $ then we have (iii) unless $(-)$ is
of the first kind or , in which case either $\mathbf {3} = \one$ or
$(\mathcal A,\tT,(-))$ is $(-)$-bipotent, yielding (i) or (ii). Thus
we may assume that $m
 \ge 3,$ so $(-)$ is of the first kind, and height~2 ($m=2 $) is no longer relevent.

If $\mathbf {3} = \one$, we can reduce $m,m'$ mod 2 and get (i),
with $m=m'=1$.
 Hence,
 by Theorem~\ref{twoexamples}, we may assume that  $(\mathcal A,\tT,(-))$ is $(-)$-bipotent.

 Assume that $c_\tT\ne c_\tT'$. If
$c_\tT+c_\tT' = c_\tT$ then $c_\tT+ m'c_\tT' = c_\tT$ by iteration,
so for $k \ge m'+1,$
\begin{equation}\label{trythis1} (m+1)c_\tT = c_\tT+ mc_\tT = c_\tT+ m'c_\tT' = c_\tT;\end{equation}
\begin{equation}\label{trythis2}(k+1){c'_\tT} = (k+1-m'){c'_\tT}+ m'{c'_\tT}  = (k+1-m'){c'_\tT}+ mc_\tT =
mc_\tT = m'c_\tT',\end{equation} which likewise is  $(k
+2){c'_\tT}$. By cancelation we get $\mathbf {m+1} = \one$ and
$\mathbf {k+2} = \mathbf {k+1}$ for all $k \ge m'+1$. Putting these
two equalities together shows that
 $$\one = \mathbf {m+1} = \mathbf {2m+1} = \dots = \mathbf {m'm+1}= \mathbf {m'm+2}=  \mathbf {m+2} = \one + \one ,$$ so $c_\tT' = m'c_\tT' =
mc_\tT = c_\tT$. Hence minimality yields $ m = m'$, yielding (i).

The analogous argument holds if $c_\tT+c_\tT' = c_\tT'$. Hence we
may assume  that $c_\tT=  c_\tT'$, and again we have (i).

%
%
%
%

The second assertion follows from the first.
\end{proof}

The case (ii) is truly exceptional, since the classical triple has
$a(-)a = a-a = 0$ for every $a$. Likewise the ambiguity in (iii) can
arise there since $\mathbf 3 = 3\cdot \one.$ However, when the
triple is $(-)$-bipotent, one has $c_\tT = (\pm) c_\tT,$
cf.~\cite[Proposition 2.14]{AGR0}, and (iii) does  not arise.

 \begin{cor}\label{symunif} Suppose $\mathcal A$ is as in the theorem, and $c \in \mathcal A^\circ$ has height $m$.
  Then  $\mathbf m = 2\mathbf k$ for some $k$.
   \end{cor}
\begin{proof} We write $c = b^\circ$. There is nothing to prove unless
$m_c \ge 3,$ in which case $c = 2 m_b b_\tT = 2 m_b c_\tT $,
implying $\mathbf m = 2\mathbf m_b$.
\end{proof}

\subsection{Surpassing relations  on a metatangible triple}$ $

Let us see  how surpassing relations can arise on metatangible
triples to yield systems.

\begin{thm}\label{circsurp} Suppose  $(\mathcal A,\tT,(-))$ is a metatangible triple.
 \begin{enumerate}\eroman
\item  $ \preceq_\circ$
is a surpassing relation, so $(\mathcal A,\tT,(-), \preceq_\circ)$
is a metatangible system.
\item  We say that a pair $(b,b')$ with $b \preceq b'$ is \textbf{usual}
if  $b \preceq_\circ b'$. A pair which is not usual  is
\textbf{unusual}. One of the following holds for any cancelative
unital metatangible triple, with surpassing $\circ$-PO $ \preceq$:

\begin{enumerate}\eroman \item   $ \preceq\  =  \ \preceq_\circ$, i.e., all pairs are usual;
 \item  The triple $(\mathcal A,\tT,(-))$ is exceptional.
\item The triple $(\mathcal A,\tT,(-))$ has height 2, and there is an unusual pair satisfying $b + b' = b.$ In this case,
 if $ \preceq$ is a surpassing $\circ$-PO, then $(-)$ is of the first
 kind, with $\mathcal A$ of
 of characteristic 2.
\end{enumerate}
\end{enumerate}\end{thm}

\begin{proof} (i) By Proposition~\ref{uniq671} and  Proposition~\ref{nozero7}.

(ii) If $b$ or $b' $ has height $\ge 3$, Theorem~\ref{matrsys}  says
that the triple is exceptional. Thus we may assume  that both have
height $\le 2.$ We shall show that an unusual pair $(b,b')$ yields
(c), relying heavily on Lemma~\ref{bip1} and Remark~\ref{negmatch1}.
 If
$b, b' \in \tT$ then $b=b'$ and the assertion is trivial. If $b' \in
\tT$ we cannot have $b$ of height~2, by Proposition~\ref{nozero7}.

Thus we may assume that $b' \notin \tT,$ so
 $b' = d^\circ$
for $d\in \tT$. We cannot have $b+d =d$, for then the pair would be
usual (since $b + d^\circ = d^\circ = b'$), so we may assume by
Lemma~\ref{bip1} and Definition~\ref{precedeq07}(i),(v)
  that
\begin{equation}\label{eq:tropical007} b = b + d^\circ =  b + b',\end{equation}  and we have $(c).$


Furthermore, in (c) assume that $ \preceq$ is a surpassing
$\circ$-PO. We have $b^\circ = b^\circ + b',$ so $ b'^\circ  \preceq
b^\circ $ and  $ b^\circ  \preceq b'^\circ $, implying $ b^\circ =
b'^\circ $.  If
 $(-)$ is of the second kind then $\mathcal A$ is idempotent so
 $b \notin \tT$ (since otherwise $b = b+b'^\circ  = b
 +b^\circ = b^\circ,$ a contradiction), and then $b \in \mathcal
 A^\circ$ implying
$b = b^2 =  b ^\circ
 = {b'} ^\circ = ( b')^2 =b',$ contradicting unusuality.

 Hence   $(-)$ is of the first
 kind. If $e' = \one$ then  $\mathbf 3 = \one,$ and we are done.

 We are left with $e' \ne \one,$
implying that the system is $(-)$-bipotent. Write $b = m b_\tT$. If
$b_\tT = d$ then $mb_\tT=(m+2)b_\tT$, implying $\mathbf 3 = \one$;
if $b_1+d =d$ then $b+ b' = b'$ implying $(b,b')$ is usual. Thus we
may assume that $b_\tT+d = b_\tT.$
 But now $b + d^\circ = b,$ implying $b^\circ + b'^\circ  = b^\circ,$
 i.e.~ $b'^\circ \preceq b ^\circ$, as well as
$ b ^\circ \preceq  {b'} ^\circ,$  and thus $ b ^\circ =  {b'}
^\circ$.  Hence
 $b = b + {b'} ^\circ = b + b ^\circ
 = 3b,$  so again $\mathbf 3 = \one$ after all.\end{proof}

\begin{example}\label{weird7}$ $

Here are some unexpected surpassing relations. Take the
$\Net$-layered system of Example~\ref{nontang}(i) (of characteristic
0), with $L = \Net$.
\begin{enumerate}\eroman
\item Write $(\ell_1,a_1) \preceq (\ell_2,a_2) $ if $  a_1
< a_2$, or if
  $a_1 = a_2$ with $\ell_1 \le \ell_2.$ Here $$\one = (1,1) \preceq
  (2,1) = (1,1)+(1,1) =
  \one + \one = \mathbf
  2,$$ but $\one \not \preceq _\circ
 \mathbf
  2$.
\item Write $(\ell_1,a_1) \preceq (\ell_2,a_2) $ if $  a_1
< a_2$ with $\ell_1>1,$ or if
  $a_1 = a_2$ with $\ell_1 \le \ell_2.$ Here   $\one \not \preceq
   \mathbf
  2 = (2,1) $, whereas $\mathbf
  2 = (2,1) \preceq
(3,1) =  \mathbf
  3$ but $\mathbf
  2 \not \preceq _\circ
   \mathbf
  3$.
\end{enumerate}
  One also has unusual pairs arising from Example \ref{chart1}.

\end{example}

%
%
%

\subsection{$\tT$-classical metatangible triples}$ $
%

%

 The usual classical triple  satisfies $a^\circ = b^\circ = 0 $ for
all $a,b$, leading us to the following weaker definition.
\begin{defn}\label{Anticlas} A triple $(\mathcal A,\tT,(-))$ is \textbf{$\tT$-classical} if  $a^\circ =
b^\circ$ for some $a\ne (\pm) b$ in  $\tT$.

$(\mathcal A,\tT,(-))$   is \textbf{$\tT$-nonclassical} if $a^\circ \ne
b^\circ$ for all $a\ne (\pm) b$ in  $\tT$.
\end{defn}

\begin{cor}\label{uniq167} If   a cancelative unital metatangible triple
$(\mathcal A,\tT,(-),\preceq)$  is $\tT$-classical, then $e' =
\one$.\end{cor}
\begin{proof}  Suppose  $a ^\circ =b^\circ $ with $a
\ne (\pm) b$. If $e' \ne \one$ then by  $(-)$-bipotence
(Theorem~\ref{twoexamples}) we may assume that $a+b = b$. But then
by Remark~\ref{negmatch1}, $$b =  a+b = a+(b(-)a) =
  a ^\circ + b = b ^\circ + b = e'b,$$ implying $e' = \one$ after
  all.
\end{proof}

  There is a nice partial converse.

\begin{lem}\label{uniq16}   Any non-$(-)$-bipotent  metatangible triple
$(\mathcal A,\tT,(-),\preceq)$ is $\tT$-classical.\end{lem}
\begin{proof}  Suppose $a \ne (\pm) b$ and $a +b \notin \{ a,b\}.$
Then $a ^\circ +b = (a+b)(-)a \in \tT,$ implying  $a ^\circ +b = b$
by unique quasi-negation. 
Hence $(a+b)^\circ =  a ^\circ +b (-) b = b(-)b = b ^\circ,$ with
$a+b \ne b.$
\end{proof}

\cite[Example 3.12]{Row16v} provides a way to merge classical and
nonclassical in a metatangible triple.
\subsubsection{Anti-negated triples}$ $

Next we view the opposite situation, coming from tropical considerations, in which sums are rarely~$\zero.$

\begin{defn}\label{antineg}
A triple is \textbf{anti-negated} if $a^\circ \ne \zero$ for all $a \in
\tT$.
\end{defn}

It follows that if $a+b = \zero$ for $a,b \in \tTz$ then $a =
\zero$. This  property  has different names in the literature:
``antiring'' in \cite{DO,Tan}, ``zero-sum free'' in \cite{golan92},
and ``lacking zero sums'' in \cite{IKR5}.


\begin{lem}\label{char1} In a cancelative metatangible
triple $(\mathcal A, \tT, (-))$, if  some sum of tangible elements
$\sum_{i=1}^t a_i $ is $ \zero,$ with   $t\ge 2$ minimal, then one
of the following holds:

\begin{enumerate}\eroman \item
 $t=2$ with $a_2 = -a_1$ (the classical negative).  \item $t\ge 3$, $(-)$ is of the first kind,
 with all $a_i$ equal, and $\mathbf t = \zero$.
\end{enumerate}
  \end{lem}\begin{proof} Otherwise,  if $a_i \ne (-)a_j,$
  then we could replace $a_i + a_j$ by their   sum in $\tT$ and reduce
~$t$.

  Thus, we may assume that all of the $a_i$ are quasi-negatives of
  each other. If $t\ge 3,$ then  all of the $a_i$ are equal with
  $(-)a_i = a_i.$ (If say $a_1 \ne (-) a_1$ then $a_2 = (-)a_1$ and
  $(-)a_1 = a_3 = (-)a_2 = a_1,$ contradiction.)

   Hence
  $a_t\mathbf t = \sum _{i=1}^t a_t \one = a_t\zero.$ Canceling  $a_t$ yields
   $\mathbf t =\zero$.

   We are left with the case $t=2$, in which case $a_1 + a_2 = \zero,$ so  $a_2 = -a_1$.
  \end{proof}

\begin{lem}\label{twoexamples3}  Any ub metatangible unital
 triple  is anti-negated, and satisfies $e' \ne
 \one$.
 \end{lem}
\begin{proof} The first assertion is immediate. If $e' = \one,
$   the ub property implies $\one (-)  \one  = \one\in \tTz \cap
\tTz^\circ = \zero$, a contradiction.
\end{proof}
%



 Dolzan and Oblak \cite{DO} develop the tie between  anti-negated
semirings with matrix theory, by showing that the only invertible
matrices over multiplicatively cancelative anti-negated semirings
are generalized permutation matrices. Let us formulate this  key
feature from tropical algebra, in terms of triples.

\begin{prop}  Over an anti-negated metatangible
triple with   $\mathbf n \ne \zero$ for each $ n \in \Net,$ the only
invertible matrices  are the generalized permutation
matrices.\end{prop}
\begin{proof} In view of Lemma~\ref{char1}, the proof in \cite{Tan}
goes through.
\end{proof}

%

\subsection{Squares and sums of squares}$ $

One example is  analogous to the classical theory of real closed fields,
in which   squares are always positive.

\begin{lem}\label{sign15}   Suppose $(\tT,\cdot,(-))$ is an Abelian group with a negation map, and
$N$ is a subgroup of $\tT,$ containing the set of elements of $\tT$
which are squares, with $N$  maximal with respect to the property
that $(-)1 \notin N .$ (Such $N$ exists by Zorn's lemma.) Then  for
any $a \in \tT\setminus N$ we have $t\ge 0$ for which $a^{2^t} \in
(-)N$.
\end{lem}
\begin{proof} $N \cap (-)N = \emptyset,$ since if $a, (-)a  \in N$ then
$(-) 1 = ((-)) a a^{-1} \in N. $ For $a \in \tT\setminus N,$ we
could adjoin $a$ to $N$ (i.e. expand $N$ with all elements $a^ib$
for $b\in N$) contrary to maximality of $N$, unless $a^i b = (-) a^j
b'$ for some $i,j \in \Z$ and $b,b' \in N$. Now
$$a^{i-j} = (-)b'b^{-1} \in (-)N.$$ Take $m >0$ minimal such that
$a^{m}   \in (-)N,$ and write $m = 2^tq$ for $q$ odd.

Let $c = a^{2^t}\in \tT;$  then $c^q \in (-)N.$ But $c^2 \in N$
since $N$ includes all squares of $\tT$, so reducing the power $q$
modulo 2 must yield 1,
i.e.,  $a^{2^t} = c \in (-)N.$ 
\end{proof}

\begin{defn}\label{positive} An element    $a $ of a  monoid $\tT$ with negation
map $(-)$  is \textbf{positive} if $(-)a$ is not  a square in~$\tT.$
A submonoid $\tT'$ of $\tT$   is \textbf{positive} if each element
of $\tT'$ is positive. (In particular $(-)\one$ cannot be a square.)
\end{defn}

Note that even when $(\tT,\cdot)$ is a  group, a subgroup of $\tT$
could be positive, for example the subgroup~$\Q^+$ of $
(\Q^\times,\cdot)$, where $(-)$ is classical negation.


\begin{thm}\label{sign16}  Suppose $(N,\cdot)$ is a positive  subgroup of an
Abelian  group $\tT$  with a negation map, containing all squares in
$\tT$, which is maximal with respect to the property that $N \cap
(-)N = \emptyset,$ and suppose that $(-)\one$ is not  a square in
$\tT$. Then $\tT = N \cup (-) N.$ Furthermore, suppose that
$(\mathcal A, \tT, (-), \preceq)$ is a $(-)$-bipotent triple, and
let $\widehat{N}$ be the symmetrization of $N$ by
Definition~\ref{sym00}. Then there is an additive homomorphism
$\varphi: \widehat{N} \to \mathcal A$ given by $(a,a') \mapsto a
(-)a'.$
\end{thm}
\begin{proof}  Suppose $a \in  \tT\setminus N$. Then by
Lemma~\ref{sign15}, $(-)a^{2^t}$ is in $N$, but
$a^{2^t}=(a^{2^{t-1}})^2$ for $t\ge 1$ so the element $(-)a^{2^t}\in
N$ would not be positive, a contradiction unless $t=0,$ i.e., $(-)a
\in N$. Hence $\tT = N \cup (-) N.$ We write $a>b$ in $\tT$ if $a+b
= a.$ The map $\varphi$ is a homomorphism, since for $a > a'\in N$
we have $ \varphi (a,a') = a$ and $ \varphi (a,a) = a(-)a,$
so the greater term dominates in each verification.
\end{proof}

\subsubsection{Sign maps}\label{signf}$ $

The following system ties in with  triples over  a multiplicative
group.

\begin{example}\label{signsys} The \textbf{sign system}  is
$ (\mathcal A_{\operatorname{sgn}},
 \tT_{\operatorname{sgn}}, (-), \preceq_\circ)$, where
 $ \mathcal A_{\operatorname{sgn}} = \{ -\one, \zero, \one, \infty
 \}$ endowed with the obvious multiplication, and with idempotent addition
 also satisfying
$a + \zero = \zero  + a = a,$ $a + \infty = \infty  + a = \infty,$ $
-\one + \one = \infty. $ Thus $\mathcal A_{\operatorname{sgn}}^\circ
= \{ \zero, \infty\}.$

This can be identified with the  hyperfield of signs in \cite{Bak},
described in~Example\ref{Basicexamples}, where $\infty$ is
identified with $\{ 0, 1 , -1\}$.
\end{example}

\begin{defn} A \textbf{sign map} on a monoid $\tT$ is
a multiplicative homomorphism $$\sgn :\tT \to (\{-1, 0, 1, \infty\},
\cdot)$$

When $\tT$ has a negation map $(-),$ we require furthermore that
$\sgn ((-)a) = - \sgn(a).$ 



%
%
%
%
%

\end{defn}
(This is very close to the minus sign used in \cite[\S 3.1]{GaP}.)
$\tT ^+ : = \sgn\invr(1)$ is  a submonoid of $\tT.$  

%
%
%
%

\begin{example}\label{surpsym1} (i) $\Real$ has the classical sign map.

(ii) The semiring $\mathcal A$ of Example~\ref{surpsym} has a sign
map, given by $$\sgn(\zero,\zero) = 0, \quad \sgn(a,\zero) = 1,
\quad \sgn( \zero, a) = -1, \quad  \sgn(a,a) = \infty, \quad \forall
a \in \tG.$$  The monoid $\mathcal A^+ = \tG \times \zero$  is
positive.  Indeed, suppose that
$$(-)\one = (\zero, \one)=(a_0,a_1)^2 = (a_0^2 + a_1^2, a_0 a_1 +
a_1 a_0).$$ If $a_0 = \zero$ or $a_1 = \zero$ the second component
is $\zero$, and if $a_0 = a_1$ then both components are equal, both
of which are contradictions.
\end{example}

Conversely, we have:

\begin{lem}\label{sign10}   Suppose $N$ is a subgroup of a  group $\tT,$ containing all squares, which is maximal with
respect to the property that $N \cap (-)N = \emptyset,$ and  suppose
that $(-)\one$ is not  a square in $\tT$.  Then there is a sign map
$\sgn$ on $\tT$ such that $N = \tT^+.$
\end{lem}
\begin{proof} Using Theorem~\ref{sign16},  we define $\sgn(a) = 1$ iff $a \in N.$
\end{proof}

\begin{prop}\label{sign1} Suppose $(\mathcal A,
\tT, (-), \preceq)$ is a metatangible system, with $\tT$ a group,
and suppose that $(-)\one$ is not  a square in $\tT$. Then there is
a sign map $\sgn$ on $\mathcal A$ given by Lemma~\ref{sign10} on
$\tT,$ and $\sgn(\zero) = 0,$ and $\sgn(a^\circ) = \infty$ for each
$a \ne \zero.$
\end{prop}
\begin{proof} Take the sign map of  Lemma~\ref{sign10}, formally defining $\sgn(a^\circ) = \infty$.
\end{proof}

%
%

\subsubsection{Other basic properties of metatangible
triples}$ $

Reversibility holds for
tangible elements.

\begin{lem}\label{rev1} (Compare with Lemma~\ref{negmatch}(ii).)  In a metatangible system, if $a   \preceq_\circ b +c$ for $a,b,c \in \tT,$
then $b  \preceq_\circ a (-) c.$
 \end{lem}
\begin{proof}  Write $a +d^\circ =  b +c.   $ Then $(a+d)^\circ = b+c(-)a.$
If $c \ne a$ then $c(-)a$ is tangible, and thus equals $(-)b$ (since
otherwise $ b+c(-)a \in \tT \cap \tT^\circ = \emptyset).$

Thus we may assume $c = a.$ If $b+c$ is tangible then $a=b+c,$
implying $b \preceq _\circ b+ c^\circ = (b+c)(-)c = a(-)c.$ So we
are done unless $c=(-)b,$ so $b = (-)a \preceq _\circ (-)b(-)c =
a(-)c.$
\end{proof}

In general, we have  a weird situation.
\begin{example}\label{weird1}
In the truncated layered system of Example~\ref{nontang}(v), for $n
= 5,$ take $a = (1,1),$ $b =  (1,2),$ and $c   = (4,2).$ Then $a +
(5,1) ^\circ = (5, 2) = b +c,$ but we cannot write $b + d^\circ = a
(-)c = a+ c = c$ because the parities in $L$ do not
match.\end{example}

This is the only sort of counterexample:

\begin{thm}\label{rev2} In  a cancelative   metatangible unital
system, if $a \preceq b +c$ for $a,b \in \tT$,
 then $b  \preceq a (-) c,$ except in the following situation
 (as in  Example~\ref{weird1}):
there are $1 < m' \le m$ such that $c = \mathbf m b,$ with $\mathbf
{m'} = \mathbf {m}$ but $\mathbf {m'-2} \ne \mathbf {m-2}$, and $a
+c = c$.
 \end{thm}

 The proof, omitted here but given in \cite[Theorem~7.44]{Row16v}, is by uniform presentation to reduce the assertion
 to tangible elements; this
 reduction fails to pass through the above exception.

\subsection{Classifying metatangible systems}\label{ch0}$ $

In a cancelative unital semiring $\mathcal A$,
 $\mathbf 2 = \one$ precisely when $\mathcal A$ is idempotent. 
In other words, characteristic~1 of the first kind gives the
max-plus algebra, but  quasi-negatives are far from unique since
$\tT = \tT^\circ$. The following result shows how metatangible
systems naturally lead us to the other main tropical structures. The
theorem is quite comprehensive, except for glossing over the
exceptional systems.

\begin{thm}\label{mon11} Any  cancelative metatangible unital system $(\mathcal A,
\tT, (-), \preceq)$  must satisfy one of the following:
\begin{enumerate} \item $(-)$ is of the first kind.
$\mathcal A = \cup _{m\in \Net} \, m \tT,$ and $e' = \mathbf 3.$

\begin{enumerate}\eroman
  \item $\mathbf 3 \ne \one.$ Then $\tT$ is $(-)$-bipotent, and $(\mathcal A, \tT, (-), \preceq)$ is isomorphic to a layered system (either layered by
$\mathbf N$ or quasi-periodic in characteristic 0
(Example~\ref{chart1}), and layered by $\Z /k$ in characteristic
$k>0$).

In particular, when $\mathbf 3 =  \mathbf 2, $ we have $\mathbf m =
\mathbf 2$ for all $ m \ge 2,$ and $\mathcal A$ has height 2.

 \item $\mathbf 3 = \one.$ Then  $(\mathcal A,\tT,(-),\preceq)$
 has characteristic $2$ and height 2. The \semiring0 $\mathcal A^\circ$ is bipotent, and   the conditions of Proposition~\ref{ematch5} hold.

\end{enumerate}

 \item $(-)$ is of the second kind.  There are two possibilities:
 \begin{enumerate}\eroman
  \item
 $\tT$ is $(-)$-bipotent,
and $\tT$ (and thus $\mathcal A$) is idempotent, 
 and
$\mathcal A$ has height~2.

 \item  $\tT$ is not $(-)$-bipotent.  Then the system is $\tT$-classical, and the \semiring0 $\mathcal A^\circ$ is bipotent. Furthermore $e' = \one$.
 Hence $\mathcal A $ has height 2. Either $\mathbf N  \subseteq \tT$,  or $(\mathcal A,\tT,(-),\preceq)$
 has characteristic~$k$ for some $k \ge 1$. In the latter case, $(\mathcal A,\tT,(-),\preceq)$
 is layered by $\Z /k.$

\end{enumerate}\end{enumerate}
\end{thm}
\begin{proof} We start with
Theorem~\ref{twoexamples}, which says that  $\tT$
 is $(-)$-bipotent or $e' = \one.$ This enables us to subdivide parts (1) and (2) (although in the reverse order).
  Also, by Theorem~\ref{matrsys}, every element of $\mathcal
A$ is uniform.

(1) If  $(-)$ is of the first kind, this means that $(-)a = a$ and
all elements have the form $  m a$ for $a \in \tT.$ If $\tT$ is
$(-)$-bipotent, and $a+ b = b,$ we get $$  m a +  m' b =  m' a ,$$
$$  m a +  m' a =  ({m+m'}) b ,$$   $$ (
m a )( m' b) =  (m m') ab ,$$ which are precisely the rules for
layered addition and multiplication, so $\mathcal A$ is layered by $
\mathbf N$. Eventually the numbers~$\mathbf m$ may cycle modulo $k$,
in which case one can identify subsequent layers modulo $k$.

When $\mathbf 3 =  \mathbf 2, $ we clearly have $\mathbf m =
\one^\circ = \mathbf 2$ for all $\mathbf m \ge 2.$  The last
assertion is by Theorem~\ref{matrsys}.

When $\mathbf 3 =  \mathbf 1, $ every element has height $\le 2$ by
Theorem~\ref{matrsys}, and we conclude with
Proposition~\ref{ematch5}, noting that $a^\circ + a^\circ = e' a + a
=  a + a = a^\circ.$

(2) First assume that  $\tT$ is $(-)$-bipotent, so $\tT$ (and thus
~$\mathcal A$) is idempotent. In particular, $\mathcal A  $ has
height~2, cf.~Corollary~\ref{ematch55}.%


Now assume that  $\tT$ is  not $(-)$-bipotent, so $e'=\one$  by
Theorem~\ref{twoexamples}. The system is $\tT$-classical by
Lemma~\ref{uniq16}. Again Proposition~\ref{bip3} shows that
$\mathcal A^\circ$ is bipotent, noting that  $a^\circ + a^\circ = e'
a (-)  a =  a (-) a = a^\circ.$  By Proposition~\ref{use23}, either
$\mathbf N \subseteq \tT$  or $(\mathcal A,\tT,(-),\preceq)$
 has characteristic $k$ for some $k \ge 1$. In the latter case,
 $\mathbf 1, \dots, \mathbf {k-1},$ are distinct, since if $\mathbf {m} = \mathbf {m'}$ for $1 \le m < m' \le k,$
 adding $(-)\mathbf {m'-1}$ to both sides lowers the characteristic, which is a contradiction.
\end{proof}

\subsubsection{Examples in terms of Theorem~\ref{mon11}}\label{Exs}$ $
%

\begin{rem}\label{AGGGexmod18}$ $

\begin{itemize}\item  When
$\mathbf 3 = \mathbf 2,$ Case (1a) boils down to the supertropical
domain $\mathcal A = \tT ^+$, of height~$2,$ where $\mathcal A
^\circ = \tT^\circ.$ We get the  system $(\mathcal A,\tT, (-),
\preceq)$ of the first kind, where
  $\tT  $ is the monoid of ``tangible elements,''  $(-)$ is the
identity map, and $\succeq$ is ``ghost
  surpasses.''
Proof: For $a,b$ tangible, $a + b$ is a ghost only when $b = a =
(-)a.$
\item In general, case (1a) becomes the layered structure, as stated
in the theorem. Note that although the $\Net$-layered system is
$(-)$-bipotent of first kind, $\mathcal A^\circ$ is not bipotent
since $e+ e = \mathbf 4 \ne \mathbf 2 = e.$

\item
The classical algebra of characteristic 2 fits into (1b), with each
$a^\circ = \zero$.

There also is
 the layered algebra of Example~\ref{nontang}(v), whose system is
 metatangible of first kind but not $(-)$-bipotent.

%
%

\item Case (2a) includes another example \cite[Definition~3.6]{Row16v}.

\item
``Layered semirings'' (which come up in Cases (1a) and (2a)) were
reviewed in Example~\ref{AGGGexmod1}, also see  \cite[Example~ 7.58]{Row16v}.
\item Case (2a) leads to the following approach. If $\tT$ is positive, then $(\mathcal A, \tT, (-), \preceq)$
comes from the symmetrized system $(\mathcal
A')_{\operatorname{sym}}$ of Theorem~\ref{sign16}. (This is layered
over $\mathcal A$.)
%
%
%

\item For Case (2b), in characteristic $\ne 2,$ the classical system $(\mathcal
A,\mathcal A, -, =)$ is of second kind, satisfying $e' = \one$.

\end{itemize}

\end{rem}

%

\subsection{Specific applications of metatangible systems}$ $

We turn now to specific applications, i.e., hypergroups and fuzzy
rings, to be elaborated in the appendices.

\subsubsection{* Metatangible systems versus metatangible
hypergroups}$ $

We have embedded the theory of hypergroups into that of systems. We
can go the other direction for metatangible systems.

\begin{prop}\label{closed}
 Any metatangible (resp.~$(-)$-bipotent) system  $(\overline{\tTz}, \tT, (-), \subseteq)$
gives rise to a $(-)$-closed (resp.~$(-)$-bipotent) hypergroup
structure  on the set $\tT$, as follows:

Define $[a] = \{ a' \in \tT : a' + a = a \}.$

 Then
define addition on $\tT$
 by $$a \boxplus b =  \begin{cases} a + b : \quad a \ne (-) b,\\
 [a]: \quad a = (-) b .\end{cases}$$
  \end{prop}
\begin{proof} We verify the conditions of Definitions~\ref{Hyp00} and~\ref{Hyp}.

Recall from Remark~\ref{negmatch1}  that $a \in [a']$ iff $a' (-)a =
(-)a ,$ so $[a] =  [(-)a].$ Hence $a \boxplus (-a)  = [a] = (-a)
\boxplus a ,$ implying addition is commutative.

Next write $b<a$ if $b+a =a$ (so that $[a] = \{ b: b<a \}$, and
$(b,a]$ for $\{ c:\ b<c \le a\}$, and note that
$$[a] \boxplus b = \begin{cases} b \quad \text{if} \quad  b >a;\\ [a] \quad \text{if} \quad  b =a ;\\  [b] \cup (b,a] = [a] \quad \text{if} \quad  b
<a.
\end{cases} $$

We need to check associativity. $(a_1 \boxplus a_2) \boxplus a_3 =
a_1 \boxplus (a_2 \boxplus a_3) = \max \{ a_1, a_2, a_3\}$ unless we
get equality at some intermediate stage, i.e., one of the following
holds:
\begin{enumerate}\label{assoc15}\eroman
\item
$a_1 = (-) a_2$, \item $a_2 = (-)  a_3$, \item $a_1 + a_2 = (-)
a_3$,\item $(-) a_1= a_2 + a_3$,
\end{enumerate}
 which we check
in course.

\begin{enumerate}\eroman
\item If
$a_1 + a_3 = a_3 $ then also $a_2 + a_3 = a_3 $ and $(a_1 \boxplus
a_2) \boxplus a_3 = [ a_1 ] +a_3 = a_3 = a_1 \boxplus (a_2 \boxplus
a_3) $.

If $a_1 + a_3 = a_1 $ then also $a_2 + a_3 = a_2 $ and $(a_1
\boxplus a_2) \boxplus a_3 = [a_1]+a_3 = [a_1] =  a_1 \boxplus a_2 =
a_1 \boxplus (a_2 \boxplus a_3) $.

If $a_1 = (-) a_3 $ then $(a_1 \boxplus a_2) \boxplus a_3 = [a_1]
\boxplus a_3 = [a_3] =   a_1 \boxplus (a_2 \boxplus a_3) $.

\item  Symmetric argument to (i).

\item Suppose $a_1 \ne (-) a_2$ and $a_1 + a_2 = (-) a_3$.  Then, by
$ (-)$-bipotence, $a_1 =  (-)a_3$ or $a_2 = (-)a_3 $, so either way
$(a_1 \boxplus a_2) \boxplus a_3 = [a_3]$; $a_1 \boxplus (a_2
\boxplus a_3)$ is respectively $(-)a_3 + (a_2+a_3) = (-)a_3 + a_3 =
[a_3]$ or $a _1 + [a_3] = [a_3].$

\item  Symmetric argument to (iii).
\end{enumerate}

 In each verification we applied
  brackets to each side whenever there is equality.

Define $-a = (-)a$. Then the quasi-zeroes are exactly the sets
$[a]$, which are the hyperzeros, and $a_1+a_2$ is a hyperzero
precisely when $a_2 = -a_1.$
\end{proof}

A more encompassing result, for hypersystems that are not
necessarily metatangible, is given in \cite[Theorem~3.27]{AGR0}.


%
%
%
%
%

%

\subsubsection{**Metatangible triples and the fuzzy property}$ $

We can also recover the key property of fuzzy ring,
{\cite[Definitions~2.1,2.8]{Dr}.

\begin{defn}\label{precedeq59}
The \textbf{fuzzy property} for a triple $(\mathcal A, \tT, (-))$
is:
\medskip

 $b_1 (-) b_1' \in \mathcal
A^\circ$ and $b_2 (-) b_2' \in \mathcal A^\circ$ imply $\ b_1b_2 (-)
b_1'b_2'\in \mathcal A^\circ, \quad \forall b_i, b_i' \in  \mathcal
A $.

\end{defn}

\begin{thm}\label{fuzz1} Metatangible triples satisfy the  fuzzy
property.
 \end{thm}
\begin{proof}
Suppose that $b_1 (-) b_1' \in \mathcal A^\circ$ and $b_2 (-) b_2'
\in \mathcal A^\circ$. We appeal to the uniform presentation
(Definition~\ref{unif1}), Theorem~\ref{matrsys}, and
Corollary~\ref{sumcirc}.

\textbf{Case I} $b_1 + b_1' = b_1$. By Remark~\ref{negmatch1}, $ b_1
= b_1 (-) b_1' \in \mathcal A^\circ,$ so
$$  b_1b_2 (-) b_1'b_2' =  (b_1+b_1')b_2 (-) b_1'b_2' = b_1b_2  +(b_1'(-) b_1')b_2' \in \mathcal A^\circ.$$
Likewise if $b_2 + b_2' = b_2'$.

\textbf{Case II} $b_1 \in \tT$ and   $b_1' \in \tT$. By unique
quasi-negatives $b_1' = b_1$, in which case $\ b_1b_2 (-) b_1'b_2' =
\ b_1(b_2 (-) b_2')\in \mathcal A^\circ$ and we are done.

\textbf{Case III} $b_1 \in \tT$ and   $b_1'$ has Type 2. Writing
$b_1'  = a^\circ$ we have by Case (i)(b) of Corollary~\ref{sumcirc}
that either $a +b = a$ so we conclude by Case I, or $a^\circ =
(\pm)b$ with $e' = e$, in which case $$\ b_1b_2 (-) b_1'b_2' = \
b_1b_2 (-) b_1b_2' (-) b_1'b_2' =  b_1(b_2 (-) b_2') (-)b_1'b_2' \in
\mathcal A^\circ$$ and we are done.

Thus we may assume that if $b_1$ has Type 1 then $b_1'$ has Type 3,
and the analogous assertions hold for $b_1', b_2, b_2'.$

 We are done if $b_1$ and $b_1'$ are both of Type 2, so we may
assume $b_1$ has Type $\ne 2.$ But then $b_1$ or $b_1'$ has Type 3,
so $(-)$ is of the first kind.

\textbf{Case IV} $(-)$ is of the first kind, $b_i =  m_i a_i$ and
$b_i' =  m_i ' a_i$ for $i=1,2$ and $a,a' \in \tT$. If $m_1 -m_1'$
is even, then $m_1 m_2 - m_1' m_2' \equiv m_1 (m_2 - m_2') \pmod 2,$
implying $b_1 b_2 (-) b_1' b_2' = (m_1m_2-m_1'm_2')a_1 a_2 \in
\mathcal A^\circ$ since $(m_2 - m_2')a_2 = b_2 (- )b_2' \in \mathcal
A^\circ,$ and we are done. A similar proof applies if $m_2 -m_2'$ is
even.

 We write
$b_i =  m_i {b_i}_\tT$ and $b_i' =  m_i' {b_i'}_\tT$ for ${b_i}_\tT,
{b_i'}_\tT \in\tT$. In view of Case I, we may assume that ${b_i}_\tT
= {b'_i}_\tT.$

  Hence, by Case IV we may assume that $m_1 -m_1'$ and $m_2 -m_2'$ are
odd, and may assume that $m_1 \le m_2.$
 But then $\mathbf m a_i\in
\mathcal A^\circ$ for all $m \ge |m_1 - m_1'|$, by
Remark~\ref{useful1},
 and in particular $$b_1
b_2 (-) b_1' b_2' = m_1 m_2 a_1 a_2 (-) m_1' m_2' a_1 a_2   \in
\mathcal A^\circ.$$

%
%
\end{proof}%

We elaborate in Appendix B.

\section{Categorical  properties of systems}\label{catnot}

Having established the ubiquity of systems, let us  view systems in
categorical terms in order to relate different systems.
%
%
%
%
%
%
We  make  systems of a given algebraic structure into a category,
whose objects are systems.
 The question is how to define morphisms. The
customary way to do so in universal algebra would be via a
homomorphism, as defined in \S\ref{univalg}.
 However, in the context of systems, it
  is often preferable  to bring $\preceq$ into the picture.

\subsection{$\preceq$-Morphisms of systems}\label{surp3}$ $

%


\begin{defn}\label{morph}  A
  \textbf{$\preceq$-morphism}  of   systems $$\varphi:
(\mathcal A, \tT, (-), \preceq)\to (\mathcal A', \tT', (-)',
\preceq')$$ is a map $\varphi: \mathcal A \to \mathcal A'$
satisfying the following
properties  for $a_i \in \tT$ and   $b\preceq b'$  in $\mathcal A$:
\begin{enumerate}\eroman    \item $\varphi(a_1) \in \tT',$ and $ \varphi((-)a_1)=
(-)'
\varphi(a_1);$
\item $ \varphi(a_1 + a_2) \preceq ' \varphi(a_1) + \varphi( a_2) ;$
\item  $ \varphi(a_1 b)=  \varphi(a_1)  \varphi( b) $.
\item $ \varphi(b) \preceq ' \varphi(b').$
\item
$\varphi (\mathcal A_\Null) \subseteq \varphi (\mathcal A)_\Null.$
\item When $\zero_\mathcal A \in \mathcal A,$ we also require
$\varphi (\zero_\mathcal A ) = \zero_\mathcal {A'} .$
\end{enumerate}

%
\end{defn}

\begin{MNote}
The key difference from homomorphism is (ii), which plays a major
role in the structure (as in \cite{JMR,JMR1}) and often holds when
equality may fail, for example with tropicalization. It is in line
with hypergroup morphisms, which stipulate that $ \varphi(a_1
\boxplus a_2) \subseteq \varphi(a_1)\boxplus \varphi( a_2).$
 \end{MNote}

 These conditions arise naturally in the cases of hypergroups and also
 for Lie semialgebras. For example, although  for any $R$-module
 $M,$
  the
left multiplication map  $\ell_r : M \to M$ is a homomorphism   iff
left multiplication by $r$ distributes over $M$, the   map $\ell_r
$~is a $\preceq$-morphism iff $r(a_1+a_2) \preceq ra_1 + ra_2$ for
each $a_i \in M$; this is  described in \cite[\S 4.1]{Vi}, and
treated in \cite[\S6.9]{Row16v}.

\begin{lem} Any homomorphism $\varphi$ of systems  is a $\preceq_\circ$-morphism.
\end{lem}
\begin{proof} If $a_2 = a_1 + c^\circ$, then $\varphi(a_2) = \varphi(a_1) +
\varphi(c)^\circ$.
\end{proof}

\subsubsection{*Embedding hypergroups into the
 category of systems}$ $

In Definition~\ref{hyper} we presented the    system  of a
hypergroup. This can be made more explicit using the formalism of
categories.

 \begin{thm}\label{hypsys1} There is a faithful functor $\Psi$ from the
 category of canonical hypergroups  into the category
 of  $\tT$-reversible systems, whose morphisms are the $\preceq$-morphisms, sending a
 hypergroup
$\tT$ to its  hypersystem $(\overline{\tTz}, \tT, (-), \subseteq)$.
Furthermore, the   hypergroup $\tT$ is~metatangible (resp.~closed),
iff its hypersystem $(\overline{\tTz}, \tT, (-), \subseteq)$
is~metatangible (resp.~$(-)$-bipotent).\end{thm}
\begin{proof} The first assertion is by Proposition~\ref{hyper15} and Theorem~\ref{hypersys}, and
the second by Lemma~\ref{closed1}.
\end{proof}

In particular, closed hypergroups can be studied in terms of
\S\ref{surp22}.

\subsection{Valuations on systems}$ $

 Usually  valuations have been studied in terms of multiplication in
monoids, cf.~\cite{HV1,IKR1},  but addition plays the main role
here. Suppose $\tG$ is an ordered semigroup. We view it as a
max-plus algebra where $(-)$ is the identity map $1_\tG$, and build
its supertropical semiring system of Definition~\ref{sutropsys},
written here as $(\mathcal A _\tG, \tT_{\mathcal A _\tG},
1_{\mathcal A _\tG}, \preceq_\circ)$, where $\tT_{\mathcal A _\tG}$,
a disjoint copy of $\tG$, is the target of the valuation, and
$\mathcal A _\tG: = \tT_{\mathcal A _\tG} \cup \tGz $.

\begin{defn}\label{val1} A \textbf{valuation} of a cbimagma system  $(\mathcal A, \tT, (-), \preceq)$
is a $\preceq$-morphism  $$v:(\mathcal A, \tT, (-), \preceq)\to
(\mathcal A _\tG, \tT_{\mathcal A _\tG}, 1_{\mathcal A _\tG },
\preceq_\circ)$$ satisfying
\begin{equation}\label{V2} v(a_1a_2) = v(a_1) v(a_2) \ \forall a_i
\in \tT.\end{equation} (Recall that multiplication in $\tT_{\mathcal
A _\tG}$ is really addition.)
\end{defn}

 This resembles the ``modulus'' in
\cite{AGR}, except that $v$ normally sends elements to
$\tT_{\mathcal A _\tG}$ instead of $\tG$. Note that $v((-)a) =
v((-)\one)v(a) =  1_{\mathcal A _\tG} v(a) = v(a)$.

 \begin{MNote} We use the supertropical algebra rather than the max-plus
algebra, in order to take into account the uncertainty of
$v(a_1+a_2) $ in regard to $v(a_1)+  v(a_2)$.
 \end{MNote}

\begin{rem}\label{val5} $ $\begin{enumerate}\eroman
\item By definition, $v(a_1+a_2) \preceq_\circ v(a_1)+  v(a_2)$ for $a_i \in \tT.$ This is the opposite direction
from the custom for valuations, but we could just reverse our
definition of inequality.
    \item  Zelinsky~\cite{Ze}  also defines valuations over (classical) nonassociative algebras.

 \item Condition \eqref{V2} is not needed in our next result, modeled after
a well-known proof!
  \end{enumerate}
\end{rem}

\begin{prop}\label{val3} If $v:(\mathcal A, \tT, (-), \preceq)\to
(\mathcal A _\tG, \tT_{\mathcal A _\tG}, 1_{\mathcal A _\tG },
\preceq_\circ)$ is a $\preceq$-morphism and $v(a_1)\succeq_\circ
v(a_2)$ for $a_i \in \tT,$ then $v(a_1+a_2) = v(a_1).$
\end{prop}\begin{proof} We are given $v(a_1+a_2)\preceq_\circ v(a_1)+v(a_2) = v(a_1).$
 But $a_1  \preceq  a_1+a_2 (-)a_2$ and $v((-)a_2)= v(a_2),$ so
$v(a_1) \preceq_\circ v(a_1+a_2 (-)a_2) \preceq_\circ v(a_1+a_2 )  +
v((-)a_2)= v(a_1+a_2 )  + v(a_2),$ a contradiction unless $
v(a_1+a_2 )   = v(a_1).$
\end{proof}

\subsection{The transfer principle}\label{transf}$ $

\begin{defn}\label{precide} A \textbf{$\preceq_\circ$-identity} of a system is a
universal elementary sentence involving the surpassing relation
$\preceq$.
\end{defn}

Although not strictly identities in the sense of universal algebra
since they do not pass to homomorphic images,
$\preceq_\circ$-identities take the place of identities in the
systemic theory.

The transfer principle, whose roots are in \cite{RS}, is a method of
obtaining $\preceq_\circ$-identities,  introduced formally in
\cite{Ga} and made explicit in~\cite{AGG1}. This treatment
essentially is a reformulation of \cite[Corollary~4.18]{AGG1},
expressed through morphisms  in order to increase its applicability.
 It is based on a way of
passing from \semirings0 to rings, by means of the symmetrization
$\widehat{\Net  \{ x \}}$  of the  free \semiring0 $\Net  \{ x \}$,
with the switch negation map.   We start with an easy but
enlightening special case.


Take $\mathbf N=\mathbf N(\mathcal A)$ as in Definition~\ref{barn}.
Given $P = \sum _{\mathbf i} (a_{\mathbf i} (-) b_{\mathbf i})
x_{\mathbf i}\in  \mathbf N\{ x, (-)x  \}$, we define the
corresponding classical polynomial $\bar P = \sum _{\mathbf i}
(a_{\mathbf i} - b_{\mathbf i}) x_{\mathbf i}\in  \Net\{ x, (-)x
\}$.

 \begin{prop} \label{trans} Suppose $P = \sum
_{\mathbf i} (a_{\mathbf i} (-) b_{\mathbf i}) x_{\mathbf i}\in
\mathbf N\{ x, (-)x   \}$. If the free $\tT$-semiring $ \Net \{ x ;
 \Omega\}$ (under the usual operations of $\Net$) satisfies the identity $\bar P =
 0$,  then  $ P \succeq_\circ
\zero$ in $\mathbf N \{ x, (-)x  \}$.\end{prop}
 \begin{proof} $\bar P = \sum _{\mathbf i}(a_{\mathbf i} - b_{\mathbf i}) x_{\mathbf
 i}.$ For this to be  0, we must have each $a_{\mathbf i} = b_{\mathbf
 i},$ so $P = \sum
_{\mathbf i} (a_{\mathbf i} (-) a_{\mathbf i}) x_{\mathbf i}\in
\mathbf N
\{ x, (-)x  \}^\circ$. 
 \end{proof}

  The same ideas
give the full transfer principle (strong form) of \cite{AGG1}. Let
$\mathbf Z$ be as in Example~\ref{twist2}. Extending
Remark~\ref{hom1} we have:

\begin{lem}\label{conf1}There is a $\preceq$-morphism $\varphi: \mathbb Z \to
\mathbf Z$ given by
 \begin{enumerate}\eroman
 \item $ n \mapsto (\mathbf n,\zero);$
 \item $-n \mapsto (\zero,\mathbf n);$
  \item $ 0 \mapsto (\zero,\zero).$
 \end{enumerate}
\end{lem}
 \begin{proof}  $\varphi (m) + \varphi (- n) = (\mathbf m,\mathbf
 n),$ which for $m\ge n$ is $(\mathbf n,\mathbf
 n)+ (\mathbf {m-n},\zero)\succeq \varphi (m - n),$ and the other
 verifications are analogous.
 \end{proof}

\begin{lem}\label{conf2}   The $\preceq$-morphism $\varphi: \mathbb Z \to \mathbf
Z$ of Lemma~\ref{conf1} extends to a $\preceq$-morphism  $$\varphi:
\mathbb Z \{ x   \} \to \mathbf Z \{ x, (-)x  \}$$ by $x_i \mapsto
(x_i,\zero)$ and $(-)x_i \mapsto
 (\zero,x_i)$.
\end{lem}
 \begin{proof} The same proof as in Proposition~\ref{trans}, since
 any presentation of $\zero$ must be sent to $\mathbf N \{ x, (-)x  \}^\circ.$
 \end{proof}

 \begin{thm} \label{trans2} Suppose $P= \sum
_{\mathbf i} a_{\mathbf i}  x_{\mathbf i}, \  Q=   \sum _{\mathbf i}
b_{\mathbf i} x_{\mathbf i}\in  \mathbf N \{ x, (-)x   \}$, where $
a_{\mathbf i}\ge b_{\mathbf i}$ for each $\mathbf i$. If the free
(classical) semiring $ \Net \{ x \}$  satisfies the identity $\bar P
= \bar Q $, then  $ P \succeq_\circ Q$ in $\mathbf N \{ x, (-)x \}$.
\end{thm}
 \begin{proof} $\bar P -\bar Q= \sum _{\mathbf i}(a_{\mathbf i} - b_{\mathbf i}) x_{\mathbf
 i}.$  Now applying Lemma~\ref{conf2} yields the assertion.
 \end{proof}

%
%
%

\begin{rem}\label{conf} Because of the ambiguity involved with $\mathbf  n$, it is misleading to deal
with identities over~$\Net$ whose coefficients  are not~$(\pm)
\one$.\end{rem}

The same ideas apply to arbitrary varieties, even nonassociative,
and can provide powerful intuition.

\section{Linear algebra over a  triple}\label{linalg1}$ $

Here we tackle the various notions of linear algebra over a system.
Only the foundation is presented here; deeper theorems and their
subtleties involved are given in \cite{AGR}.

\subsection{$\preceq_\circ$-identities for matrices}\label{matr2}$
$

Identities of $n \times n$ matrices can be translated (matching the
matrix entries) into $n^2$ identities in commuting indeterminates.
Using the transfer principle, we see that many identities of
matrices over rings translate to $\preceq_\circ$-identities  of
$\mathbb N_{\operatorname{max}}[\Lambda]$, and thus of semirings.
(These results   hold more generally over $\tT$-\semirings0.)

 \begin{lem}\label{prodform0} Suppose $A$ is
 a square matrix whose entries are all $\zero$ and $\pm \one$. If the
 determinant of $A$ (taken in~$\Z$) is $\zero$, then $A$ is $\circ$-singular in
 the sense of Definition~\ref{signeddet}. \end{lem}
\begin{proof} Immediate from Proposition~\ref{trans}. \end{proof}

  Although other results also are
consequences of the transfer principle, we indicate their easy
direct proofs.

\begin{defn}\label{adjo} Write $M_{i,j}$ for the $i,j$ minor, which
is the $(-)$-determinant of $(a_{k,\ell})_{k\ne i, \ell\ne j}$ of a
given matrix $A$. The \textbf{$(-)$-adjoint} matrix $\adj A$ is
$(M_{j,i})$. When $A$ is ambiguous we will write $a_{i,j}'$ for
$M_{i,j}$.
\end{defn}

We inductively write $(-)^k$ for $(-)((-)^{k-1})$, where $(-)^1$ is
$(-)$.

\begin{rem} $|A| = \sum _{j=1}^n (-)^{i+j} M_{i,j}a_{i,j},$
for any given $i$.
\end{rem}

\begin{prop}  $ \adj{B}\adj{A} \preceq_\circ \adj{AB}$ for $n \times n$ matrices $A$ and $B$.
 \end{prop}
\begin{proof}
 Write $AB = (c_{i,j}),$ we see that $\adj{AB} =
(c'_{j,i})$ whereas the $(i,j)$-entry of $\adj{B}\adj{A}$ is $\sum
_{k=1}^n b'_{k,i}a'_{j,k}$. Since $a'_{j,k}b'_{k,i}$ appears in
$c'_{j,i}$, we need only check that the other terms in $c'_{j,i}$
occur in matching pairs with opposite signs. This kind of
computation goes back to \cite{St}. These are sums of products of
the form
$$d_{k_1,\pi(k_1)}d_{ k_2,\pi(k_2)}\cdots
d_{k_{n-1},\pi(k_{n-1})},$$ where $k_t \ne j,$ $\pi(k_{t}) \ne i$
for all $1 \le t\le n-1,$ and
$$d_{k_t,\pi(k_t)} =  a_{k_t,\ell}b_{\ell,\pi(k_t)}, \qquad\text{for suitable }\ell.$$
If the $\ell$ do not repeat, we have a term from $\adj{B}\adj{A}$.
But if some $\ell$ repeats, i.e., if we have $$d_{k_t,\pi(k_t)} =
a_{k_t,\ell}b_{\ell,\pi(k_t)}, \qquad d_{k_u,\pi(k_u)} =
a_{k_u,\ell}b_{\ell, \pi(k_u)},$$ then in computing $c'_{j,i}$ we
also have a contribution from another permutation $\sigma$ where
$\sigma(k_t) = \pi(k_u)$ and $\sigma(k_u) = \pi(k_t)$ (and   $\sig =
\pi$ on all other indices), whereby we get
$$a_{k_t,\ell}b_{\ell,\sigma(k_t)}a_{k_u,\ell}b_{\ell, \sigma(k_u)} =
a_{k_t,\ell}b_{\ell,\pi(k_u)}a_{k_u,\ell}b_{\ell, \pi(k_t)} =
a_{k_t,\ell}b_{\ell,\pi(k_t)}a_{k_u,\ell}b_{\ell, \pi(k_u)},$$ as
desired. \end{proof}

 \begin{lem}\label{prodform2} $\absl A I \preceq_\circ A\adj A
 $ over any  $\tT$-semiring triple $\mathcal A$.
 \end{lem}\begin{proof} The diagonal terms are equal, by definition,
 and the extra terms off the diagonal are  known to match, by rewording
 \cite[Lemma~2]{RS}.
\end{proof}
 \begin{thm}\label{prodform3} $ \absl {A } \absl {B} \preceq _\circ \absl
 {AB},
  $ for any matrices $A,B\in M_n(\mathcal A)$. 
 \end{thm}\begin{proof} By the semiring
argument in \cite{RS}, matching terms in the products, since any
term in $\det (AB)$ not in $\det (A )\det( B)$ occurs twice, with
opposing signs. This follows from
 \cite[p.352, end of proof of (a)]{RS}.
\end{proof}

Note however that the determinant is not a morphism since it
reverses $\preceq$.  In this sense the determinant could be
considered an ``$\preceq$-antimorphism.'' We do get equality when $
\absl
 {AB}\in \tT.$

\subsection{Ranks of matrices}$ $


Our next task is to compare different notions of rank of matrices,
in terms of its row vectors and its column vectors. A vector $v\in
 \mathcal A^{(n)}$ is called \textbf{tangible} if each of its entries
is in $\tTz.$ A matrix is \textbf{tangible}  if each of its rows is
tangible. We only consider tangible matrices $A$, mostly for
metatangible systems. This is a small step back from
\cite{IzhakianRowen2009TropicalRank}, but the tangible case is a
compelling one, since one can recover the full supertropical result
from it.

  Modules over triples are studied in subsequent papers, such as
\cite{JMR,JMR1,GaR,AGR}. Here, to stay on track, we just consider
the ``free module'' $M := (\mathcal A^{(n)}, \tT_M, (-))$ where
$\tT_M$ is the set of vectors with a single nonzero component, which
is in $\tT$, and $(-)$ is defined componentwise. In this case, $M$
is also a $\tT$-module, with multiplication defined componentwise.

\begin{defn}\label{dep} Suppose that $(M:=\mathcal A^{(n)}, \tT_\mathcal M,(-), \preceq)$ is a  systemic module over $\mathcal A$.

 A  set  $S
\subseteq \mathcal M$ is $\circ$-\textbf{dependent} if there are
$v_1, \dots, v_m \in S$ and   $ \a _j \in \tT$ such that
$$ \sum _{j=1}^{m} \a_j v_j \in  \mathcal M^\circ.$$ Otherwise $S$
  is $\circ$-\textbf{independent}.

 An element $v\in \mathcal M$ is $\circ$-\textbf{dependent} on a
$\circ$-independent set $S \subseteq \mathcal M$, written $v
\in_{\operatorname{dep}} S,$ if  $S \cup \{v \}$ is
$\circ$-dependent.

 \end{defn}

\begin{definition}\label{rdef0} The \textbf{(surpassing) row rank} of a matrix $A$ is the maximal
number of  $\circ$-independent rows of $A$. The \textbf{column rank}
of the matrix $A$ is the maximal number of  $\circ$-independent
columns of $A$.

The \textbf{submatrix rank} of the matrix $A$ is the maximal $k$
such that $A$ has a nonsingular $k\times k$ submatrix.
\end{definition}

%
%
%

Let us consider   the following assertions:

\begin{enumerate}\eroman  \item \textbf{Condition A1}:
The submatrix rank is less than or equal to the row rank and the
column rank.

  \item \textbf{Condition A2}:  The three definitions of rank are equal for any
tangible matrix, when $\tT $ is a multiplicative group.
\end{enumerate}

%
%
%

\begin{MNote} We are about to see that Condition A1  holds for
cancelative $(-)$-bipotent triples. Condition A2 is considerably
more delicate, usually
 also  requiring
   height 2 even in the $n \times n$ case; a
thorough treatment is given in \cite{AGR}, linked to
\cite{AGG2}.
 \end{MNote}

An easy induction argument enables one to reduce Condition A1 to
proving that a square matrix $A$ is $\circ$-singular  if its rows
are
 $\circ$-dependent, which is our next result.

\begin{lem}\label{neucla17} Suppose $(\mathcal A, \tT, (-))$ is a
semiring triple, with $a_i\in \tT$ and $ b_i\in \mathcal A$.
 If $a_1 + a_2 = a_1$ and $b_1 + b_2 =
b_1$, then $a_1 b_1+ a_2 b_2 = a_1 b_1.$
 \end{lem}

\begin{proof}
 $a_1 b_1+ a_2 b_2 = (a_1+ a_2) b_1 +  a_2 b_2  = a_1 b_1+ a_2
(b_1 +b_2) = a_1 b_1+ a_2 b_1= (a_1+ a_2) b_1 = a_1 b_1.$
\end{proof}

\begin{thm}\label{A1part} Suppose that  $( \mathcal A, \tT, (-) )$ is  a cancelative
$(-)$-bipotent triple over a monoid $\tT$.
 If the rows $v_1,\dots, v_n$ of a tangible $n \times n$ matrix $A$ over a
cancelative $(-)$-bipotent triple
 are $\circ$-dependent, then $|A| \in \mathcal A ^\circ.$
\end{thm}\begin{proof} Localizing, we may assume that $(\tT,\cdot)$ is a group.
We start  by mimicking the proof in \cite{IR1}. Suppose $ \sum
_{j=1}^{m} \a_j v_j \in  \mathcal M^\circ$ for $\a_j \in \tT$.
Replacing $v_j$ by $\a_j v_j$, we may assume that the sum of the
rows is a vector in $({\mathcal A ^{(n)}})^\circ. $ Write  $A
=(a_{i,j}).$

Recall the uniform presentation of Definition~\ref{unif1}. We say
that an element $c$ \textbf{dominates $c'$} if either $c_\tT =
(-)c'_\tT $ or $c_\tT +c'_\tT  = c_\tT $. (Then if $c = \sum b_i$
for $b_i \in \mathcal A$, some $b_i$ must dominate $c$.)

 From
Equation~\eqref{eq:tropicalDetsign}, there are  $k_1, \dots, k_n $
such that $|A|$ is dominated by $ a_{k_1,1} \cdots a_{k_n,n}\in
\tT.$ Interchanging rows we may assume that $k_i = i$ for each $i,$
i.e., $|A| $ is dominated by $a_{1,1} \cdots a_{n,n}.$

We say that  $a_{i,j}$ is (column) \textbf{critical} if $a_{i,j}$
dominates all other entries in the $j$ column of $A$. Any critical
diagonal entry of $A$  must be matched by another critical entry in
the same column; i.e., $a_{i,j} = (-)a_{i',j}$. Now starting with
some critical nondiagonal entry, say $a_{i_2,a_1}$ with $i_1 = 1,$
we take $a_{i_3,i_2}$ with $i_3 \ne i_2,$ and continue in this way
until we return to $i_1$. This gives us
$$a_{i_1,
i_k} a_{i_{k},i_k-1} \cdots a_{i_2,i_1}$$ where each entry is
critical. Defining the permutation $\pi$ by $\pi (i_1) = i_k, \dots,
\pi(i_2) = i_1$ and the identity elsewhere, it is clear that $a_{i_1,
i_1}\cdots a_{i_{k-1},i_{k-1}} a_{i_k,i_k}$ is dominated by $a_{i_1,
i_k} a_{i_{k},i_k-1} \cdots a_{i_2,i_1}$, and thus $\Det{A}$ is also
attained by $\pi.$ This means that for each of these $i$,
  $a_{i,j} = (-) a_{j,j},$ and dividing the $j$-th column through by
  $a_{j,j}$ enables us to assume that every dominant entry is $(\pm)
  \one$. But now the non-dominant entries do not play a role either
  in the hypothesis  (that the rows are $\circ$-dependent)  or the conclusion   (that $|A| \in \mathcal A ^\circ)$, so we replace them by
  $\zero,$
  and assume that every entry of $A$ is in $\{ \zero, -\one, \one
  \}.$ In order to cancel $\one$ on the diagonal, we may assume that
  each row has some non-diagonal entry $-\one,$ so picking these
  entries when building $\pi$ of the previous paragraph, we have all
  $a_{i_\ell, i_{\ell +1}} = (-)\one$. Renumbering the rows and columns
  (since interchanging both the $i$ and $j$ rows and columns does not
  affect the hypothesis or conclusion) we may assume that $i_\ell =
  \ell$ for all $\ell \le k.$  Since $\left|\begin{matrix} \one & (-)\one
  & \zero
  &  \dots & \zero \\ \zero
  &  \one & (-)\one
  &  \dots & \zero \\ \vdots & \vdots &\ddots &\vdots \\  (-)\one
  & \zero &  \zero
  &  \dots &\one
   \end{matrix}\right|= \one (-) \one = e,$
   we see that $e$ is a summand of $|A|$.

   When $(-)$ is of the second kind, idempotence implies that $|A| =
   e,$ and we are done. Thus we may assume that $(-)$ is of the first
   kind,
   and every entry of $A$ is in $\{ \zero,  \one
  \}.$

First assume that in each column the number of entries that are
$\one$ is even. Then reading $A$ as a classical matrix, the sum of
each column is 0 $\pmod 2,$ so its classical determinant is 0 $\pmod
2,$ i.e., $|A| \in \mathcal A^\circ.$ Thus one may assume that in
some column the number of entries that are $\one$ is odd, and take
the column with the smallest such number of entries $k_0$. Then
$\mathbf {k_0} \in \mathcal A^\circ;$ since $\mathbf 2 \in \mathcal
A^\circ,$ we have $\mathbf {k} \in \mathcal A^\circ$ for all  $k\ge
k_0.$ But the argument of the fourth paragraph gives us at least
$k_0$ summands (each equal to~$\one$) of~$|A|$, so $|A|$ must be
some $\mathbf {k} \in \mathcal A^\circ$.
\end{proof}

 \begin{cor}\label{Lieinv}  Condition A1 holds over a
cancelative $(-)$-bipotent triple.
 \end{cor}
\begin{proof} If $m$ is the row rank of a matrix $A$, then any $m+1$ rows are
$\circ$-dependent, implying that every $(m +1) \times (m+1)$ minor
is $\circ$-singular, so the submatrix rank is at most $m$. The same
argument holds for the column rank.\end{proof}

\cite[Theorem~4.18]{AGG2} provides  a stronger conclusion, called
``Cramer's rule,'' under  extra hypotheses.
  Cramer's rule   is obtained rather
generally in \cite[Theorem~4.8]{AGR} when assuming a Noetherian-type
property.

\section{More applications}\label{Exsmore}$ $

We focus on three major examples --- tropicalization, exterior
semialgebras and Lie
semialgebras. 

\subsection{Tropicalization of Puiseux series}\label{tropi}$ $

Tropicalization, perhaps the main tool in tropical mathematics, has
been studied in various contexts. Originally ``standard''
tropicalization was a map  to the max-plus algebra, obtained by
applying logarithms to real or complex varieties, as exposed in
\cite{IMS}.

Most of the recent research on tropicalization
   has focused on  the  Puiseux series valuation.
For any cbimagma $K$, one can define the set $ K\{\{t\}\}$ of
Puiseux series on the variable $t$, which is the set of formal
series of the form $f = \sum_{j = \ell}^{\infty} c_j t^{j/N}$ where
$N \in \mathbb{N}$, $\ell \in \mathbb{Z}$, and $c_j \in K$. (One
could use any subgroup of $(\R,+)$ for the exponents in the series,
but the definition becomes more complicated without enhancing the
theory, since $(\Q, +)$ is model complete in the elementary theory
of ordered groups.) Then we have the \textbf{Puiseux valuation}
$\operatorname{val} : K\{\{t\}\} \setminus \{0\} \rightarrow
\Q_{\operatorname{max}} \subset \R_{\operatorname{max}}$ defined by
\begin{equation}
\operatorname{val}(f) = -\min_{c_j \neq 0}\{j/N\},
\end{equation}
and formally $\operatorname{val}(0) = \zero\ (=-\infty).$ (We put in
the negative to pass from minimum to maximum.) We also call
$\operatorname{val}$ \textbf{tropicalization}, now viewed as a
$\preceq$-morphism, cf. Definition~\ref{val1}.

 Customarily one takes $K$ to
be the field of complex numbers, so that $K\{\{t\}\}$ is an
algebraically closed field, even though we find it convenient to consider
tropicalization over other semirings, especially $\Net_0$. 
We would want
$\operatorname{val}$ to be a $\preceq$-morphism. But this does not
quite work since $\Q_{\operatorname{max}}$ does not have negatives,
so we consider several related versions of tropicalization which are
more amenable to algebraic methods.

\begin{example}\label{tropex1} $ $
The Puiseux series valuation comes in
   various forms:
 \begin{enumerate}\eroman

   \item  The usual Puiseux series valuation $\operatorname{val}$ to the max-plus algebra $\Q_{\max}$ from
the Puiseux series algebra $K\{\{t\}\}$  on the variable $t$, again
as exposed in \cite{IMS}, and to be reviewed presently.

 \item  The Puiseux series valuation from
the Puiseux series algebra $K\{\{t\}\}$
  to the supertropical semiring,~ \cite{IzhakianRowen2007SuperTropical}.

    \item  The Puiseux series valuation  from
 $K\{\{t\}\}$ to the ``exploded algebra,''  \cite{Par} or, more generally, to
Example~\ref{nontang}(viii).

  \item  The Puiseux series valuation from
 $K\{\{t\}\}$ to the  semiring layered by $\Net$,
  \cite{IKR0}.
%

\end{enumerate}
\end{example}

Each version has its specific motivation. Supertropical algebra is
compatible with the value group of the Puiseux series valuation. If
one wants to take the residue field into account one would pass to
the exploded algebra. Even so, this only utilizes the lowest term of
the Puiseux series. When this is lost through cancelation, one would
need to dig deeper into the Puiseux series, taking an infinite
direct sum $\oplus _{i \in \Net} \mathcal G_i$ of target systems.
This would be the tropicalization of the associated valuation ring,
but has not yet  been utilized in the
literature. 
   These various approaches are unified  in terms of $\preceq$-morphisms of systems of the relevant categories.



\begin{prop} \label{circmorph3}   In each of the cases
taken from Example~\ref{tropex1} (in the same order),
$\operatorname{val}(f)$ provides a $\preceq$-morphism $v$ from the
Puiseux series  $ K\{\{t\}\}$ (viewed as a classical system) to
$\tT$ in one of the metatangible systems that we have described
earlier:

 \begin{enumerate}\eroman
   \item $v(f) = -\operatorname{val}(f)$, taking values in the max-plus algebra,
   cf.~Remark~\ref{maxpsys}.

  \item $v(f) = -\operatorname{val}(f)$, taking  values in the supertropical algebra.

\item $v(f) = (1,-\operatorname{val}(f))$, taking values in the layered algebra.

 \item For a Puiseux series
$f = \sum_{k = \ell}^{\infty} c_k t^{k/N}$ with $c_{\ell}\ne 0,$
take $v(f) = ( c_{\ell} ,-\ell/N)$ in  the ELT~algebra.

This can be viewed more generally, in analogy to viewing
tropicalization as passing to the target of a valuation $v: R \to
\mathbb Q$,  where $R$ is a ring. Suppose that the valuation $v$ has
a \textbf{uniformizer} $\pi$ such that $v(\pi) = 1.$ Thus, for any
element $r$, taking $v(r) = m/n$ we have $v(\pi ^{-m/n}r) = 0.$ Now
one can also take into account the residue ring $R/P$ where $P$ is
the valuation ideal and, letting $L = R/P,$ consider the map $R \to
L\times \tG$,  the ELT-algebra, given by $a
\mapsto (\pi^{-m/n}r, w(r))$.

%
%
%
%

\end{enumerate}
 \end{prop}
 \begin{proof} In each case, we   verify that $v(-f) = (-)v(f),$
 and $v$ preserves addition (with respect to  $\preceq$).
 \end{proof}

%

This process  indicates a way of  tropicalizing standard algebraic
definitions in this setting, where one expects that some case of
Proposition~\ref{circmorph3} to be used, according to the
context.
%

%
%

\subsection{Exterior (Grassmann) semialgebras }\label{grT}$ $

 Paralleling the classical case, for free modules, the
tensor semialgebra yields a construction of the Grassmann
semialgebra whose base is the union of even elements and odd
elements. The definition given in \cite{GG} (which goes on to treat
the Pl\"{u}cker equations) is a semialgebra $\mathfrak G $ generated
by a free module $V$ with a base $\{ e_i : i \in I \}$, together
with a product $\mathfrak G \times \mathfrak G \to \mathfrak G$
satisfying $e_i^2 = \zero$  for each $i \in I.$ These could be
constructed by means of the tensor semialgebra, modulo the relations
$x_i^2 = \zero$. As noted in \cite{GG}, such a definition relies
  on the presentation in terms of the base, since in general
$v^2 \ne \zero$ for $v \in V.$ This would mean that a
sub-semialgebra of a Grassmann algebra need not be Grassmann, such
as the semialgebra generated  by $e_1$ and $e_1+e_2$.

\begin{defn}\label{Grass1} A (faithful) \textbf{Grassmann}, or \textbf{exterior}, semialgebra, over
a $C$-module $V$ with a negation map,
 is a semialgebra $\mathfrak G$ generated by $V$,
together with  a negation map extending $(-)$ and a product
$\mathfrak G \times \mathfrak G \to \mathfrak G$ satisfying
\begin{enumerate}\eroman
  \item
    \begin{equation}\label{G2} v_1v_2 = (-)
 v_2v_1 \qquad  \text{for} \qquad v_i \in V,
 \end{equation}
    \item     \begin{equation} (-)(v_1\cdots v_t) =
    ((-)v_1)v_2\cdots v_t.
     \end{equation}
 \end{enumerate}

Thus $v_{\pi(1)}\cdots   v_{\pi(t)} = (-)^{\sgn(\pi)} v_1\cdots
v_t.$

\end{defn}  When $V$ is the free module, this definition, which is independent of the base,
covers the one in
\cite{GG}, in which $(-)$ is the identity map.   Definition~\ref{Grass1} maps onto \cite{GG},
where $e_i (-) e_i$ is
sent to 0.

The appropriate triple is $(\mathfrak G, \tT, (-))$, where $\tT = \{
v_1\cdots v_t: \, v_i \in \tT, \,  t \in \Net \},$ the submonoid
generated by $V$.

\begin{lem} $ v^2 \in \{ b \in \mathfrak G: b = (-)b \}.$
\end{lem}
\begin{proof} $(\sum \a_i e_i)^2 = \sum \a _i^2 e_i^2 + \sum _{i<j} \a_i \a_j (e_i e_j + e_j
e_i) $, and note that $e_i^2 = (-) e_i^2.$
\end{proof}

(This set $\mathcal\{ b \in \mathfrak G: b = (-)b\}$ is just
$\mathfrak G$ when $\frac 12 \in \tT.$)

\begin{lem} If $\mathcal G = \mathcal
A^{(\overline I )},$ then it is enough to check that $$  e_i e_j =
(-) e_j e_i, \qquad \forall i,j \in I,$$ extended via
distributivity.
\end{lem}
\begin{proof}
 $(\sum \a_i e_i)(\sum \beta_j e_j) = \sum \a _i \beta_j e_i e_j = (-)\sum \a _i \beta_j e_j e_i = (\sum \beta_j e_j)(\sum \a_i e_i),$
 yielding~(i).   (ii)  is also by
linearity.\end{proof}

\begin{lem}\label{prod1}  $v_1 v_2 = (-) v_2 v_1 $ is central in $ \mathfrak G$, for all $v_1,v_2 \in V.$
\end{lem}
\begin{proof}
  $ v_1 v_2v_3= (-) v_1 v_3 v_2 =
v_3v_1 v_2,$ implying that $v_1 v_2$ is central.
\end{proof}

\begin{defn}\label{biG} Given a Grassmann semialgebra $\mathfrak G$ over a module
$V$  with a negation map $(-)$, we define $\tT^+$ to be the set of
all even products of elements of $V$, $\mathfrak G_0$ to be the
submodule of $\mathfrak G$ generated by $\tT^+$, $\tT^-$ to be the
set of all odd products of elements of $V$, and $\mathfrak G_1$ to
be the submodule of $\mathfrak G$ generated by $\tT^-$.
\end{defn}

\begin{lem}\label{biG1} $\mathfrak G = \mathfrak G_0 + \mathfrak G_1.$
$\mathfrak G_0$ is in the center  of $\mathfrak G,$ and $\mathfrak
G_1 = \mathfrak G_0V .$ When $V$ is a free module with
 negation,
 then $\mathfrak G = \mathfrak G_0 \oplus \mathfrak G_1$ is a super-semialgebra.
\end{lem}
\begin{proof} The first assertion is an immediate induction based on
Lemma~\ref{prod1}.  For the free module with negation, we match
components.
\end{proof}

\begin{example}\label{varex1}   When $V$ is a free  $\mathcal A$-module with negation, with base $\{ e_i, (-)e_i : i
\in I\}$,  the tensor semialgebra $T(V)$ becomes a Grassmann
semialgebra $\mathfrak G$ when we impose the extra relations that
$e_je_i =((-)e_i)e_j = e_i((-)e_j)$ for all $i,j\in I$. $\tT$ is the
set of simple tensors in which one does not have  both $e_i$ and $
(-)e_i$. Every
 term of even degree in the $e_i$ is central, so $\mathfrak G$ satisfies the $\succeq_\circ$-surpassing identity
$[x_1,[x_2,x_3]] \succeq \zero$.\end{example}

\begin{lem}\label{prod2}   In Example~\ref{varex1}, any nonzero element of $\mathfrak G$ is  of the form
$\sum (\pm) a_{\mathbf i} e_{i_1}\cdots e_{i_k},$ summed over $i_1 <
\dots < i_k$ with $a_{\mathbf i} \in \mathcal A,$ plus a quasi-zero.
\end{lem}
\begin{proof} Take an element $(\pm) \a e_{i_1}\cdots e_{i_k}. $ Rearrange the $e_i$ appearing in each summand, since any
time an~$e_i$ repeats, the product is in $\mathfrak G ^\circ.$
\end{proof}

  \begin{MNote}    Ironically, even when $V= \mathcal A ^{(n)}$ does not have a negation map, we can still
define a negation map on the ideal $T(V)_{\ge 2}$ of   $T(V)$
comprised of tensors of length $\ge 2$, given by $(-)v_i v_j = v_j
v_i.$ Letting $\tT_{T(V)_{\ge 2}}$ be the simple tensors of length
$\ge 2,$ we have the triple $(T(V)_{\ge 2}, \tT_{T(V)_{\ge
2}},(-)),$ and then the theory of triples is applicable! Various
 Grassmann semialgebras are studied in detail in~\cite{GaR}, which
reformulates identities for classical
  Grassmann semialgebras.\end{MNote}

\section{Nonassociative semialgebras with a negation
map}\label{Liealg}$ $

 In this section we bring in nonassociative semialgebras, especially Lie semialgebras,
since Lie algebras are so important in classical representation
theory. Now we might want $\tT$ to have Lie multiplication instead of being a monoid.
Whereas the Jacobi identity on a Lie algebra
$L$ is equivalent to the adjoint representation $\ad: L \to \adL$
being a Lie homomorphism, the correspondence in tropical algebra is
more delicate.

\subsection{Super-semialgebras}\label{superal}$ $

As in the classical case, one can ``superize,''  to  make a theory super.

%

\begin{defn}\label{supergr1}
  The \textbf{Grassmann envelope} of a super-semialgebra $\mathcal A = \mathcal A_0 \oplus \mathcal
  A_1$ is the sub-semialgebra $ (\mathcal A_0 \otimes \mathfrak G_0) \oplus ( \mathcal A_1 \otimes
  \mathfrak G_1)$ of $ \mathcal A \otimes \mathfrak G,$ with $\mathfrak G$ as in
Lemma~\ref{biG1}. 

Suppose $\mathcal V$ is a variety. A \textbf{super-$\mathcal V$}
semialgebra  is a   super-semialgebra $\mathcal A $ whose Grassmann
envelope is in $\mathcal V$.

For example,
 $\mathcal A$   is \textbf{super-commutative} if
$a_i a_j = (-)^{ij} a_j a_i$ whenever $a_i \in \mathcal A_i,$ $a_j
\in \mathcal A_j $ for $i,j \in \{ 0, 1 \}.$ \end{defn}
 \medskip

 The Grassmann envelope of a Grassmann super-semialgebra $\mathfrak G$ itself is $ (\mathfrak G_0 \otimes
\mathfrak G_0) \oplus ( \mathfrak G_1 \otimes
  \mathfrak G_1)$ which is commutative,
so $\mathfrak G$ is
  super-commutative. Conceptually, Definition~\ref{supergr1} is just an
  elegant form of book-keeping, where in evaluating multilinear operations on a
  superalgebra we put in $(-)^k$, where $k$ is the number of odd occurrences of the entries.

\subsection{Lie $\preceq$-semialgebras and Lie $\preceq$-super-semialgebras, and their triples}\label{Lie1}$ $

\begin{defn}\label{anticom} A bimagma $\mathcal A$ with negation map  is \textbf{$(-)$-anticommutative} if
it satisfies the following conditions for all $b,b' \in \mathcal A$:
\begin{enumerate}\eroman
\item   $b^2\in  \mathcal A^\circ;$
 \item
 $ b' b  = (-)( b  b' ) =  b( (-) b')  =  ((-)b)  b' .$
\end{enumerate}
\end{defn}

(In classical mathematics, (ii) is derived from (i) by
multilinearization, but this argument requires a genuine negative,
and so is inapplicable here.)

%

%

\begin{defn} For a  bimagma $\mathcal A$ with negation map,
 given $b \in  \mathcal A$, we  define $\ad_b \in \Hom (\mathcal A,\mathcal A)$ by $\ad_b (b') =
 bb',$ and
$\ad \mathcal A = \{ \ad_b: b \in \mathcal A\}.$
\end{defn}

  $\ad \mathcal A$ is a semialgebra and $\tT$-submodule of
 $\Hom (\mathcal A,\mathcal A)$,     with
 a natural negation map $(-)\ad_b = \ad_{(-)b} $.

%

\begin{defn} \label{Lies} A \textbf{Lie $\preceq$-semialgebra  with a negation map (over a semifield $F$)} is an $F$-module $L$ with a negation map $(-)$, endowed with
$(-)$-anticommutative multiplication $ L \times L \to  L,$ written
$(b,b') \mapsto [b b']$,
 called a  \textbf{Lie bracket} (in view of the standard notation $[ab]$ for Lie multiplication),
satisfying \begin{equation} \label{precedeq1}  \ad_{[b b']} \preceq[
\ad _b, \ad _b'] \quad \forall b,b' \in L,
\end{equation}
where the right bracket is the Lie commutator. (Note that we do not
require a negation map on $F$.)
\end{defn}

\begin{lem}\label{ideal12}  $ [[ b b']v] \preceq [ b[bv]](-)[ b'[bv]]$ for all
$b,b',v\in L$.
\end{lem}
\begin{proof} $ [[ b b']v]= \ad _{[ b b']}(v)\preceq \ad _b(\ad _{b'}(v))(-) \ad _{b'}(\ad _{b}(v)) =
[ b[bv]](-)[ b'[bv]].$
\end{proof}

 Lemma~\ref{ideal12} can be viewed as the $\preceq$-surpassing version of
  Jacobi's identity.

  \begin{prop}\label{Liesy} If $L$ is a Lie semialgebra  with a negation map, then there is a Lie
 $\preceq$-morphism
  $\ad: L \to \adL$, given by $b \mapsto \ad_b.$ (In fact   $\ad$ preserves addition.)
\end{prop}
\begin{proof} By Lemma~\ref{ideal12}.
\end{proof}

%
Definition~\ref{Lies} is a bit stronger than the analog of Blachar's
  definition~\cite{Bl}.

  \begin{prop}\label{Pois1} Any associative
  semialgebra $R$ with negation map becomes a Lie $\preceq_\circ$-semialgebra
  under the Lie product $[bb'] = [b,b']$.
\end{prop} \begin{proof} Using Lemma~\ref{Pois0}, one
verifies for any $v\in L$ that
 \begin{equation}\begin{aligned}\ad_{[b b']}(v) = [[ b b']v] & =   [[ b
,b'],v] = [bb'(-)b'b,v] = [bb',v] + [v,b'b]\\ & \preceq_\circ
b[b',v] + [b,v]b' + [v,b']b + b'[v,b] =(\ad _b \ad _{b'}(-) \ad
_{b'} \ad _b)v = [ \ad _b, \ad _b']v, \end{aligned}\end{equation}
for all $b,b' \in L.$
\end{proof}

We call this Lie $\preceq_\circ$-semialgebra $R^{(-)}.$


 \begin{cor}\label{Lieinv7}  For any associative
  semialgebra $(R,*)$ with involution and negation map,  $(R,*)^-$  is a Lie
 $\preceq_\circ$-sub-semialgebra of $R^{(-)}.$
 \end{cor}
  \begin{proof}
 It is closed under the Lie product.
 \end{proof}
 \begin{rem} We are now in a position to
 define the symmetrized analogs of the classical Lie algebras,
 over a \semiring0 $\mathcal A$.
Namely, we   take $\widehat{\operatorname{sl}_n}(\mathcal A) = \{
((a_{i,j}),(b_{i,j}))
 \in M_n(\hat {\mathcal A}):
  \sum _i a_{i,i}=
\sum _i  b_{i,i} \},$  the symmetrized analog of the classical Lie
algebra $A_{n-1}$. To obtain the analogs of $B_n$, $C_n$, and $D_n$,
one  applies Corollary \ref{Lieinv7} to the transpose and symplectic
involutions, taking the subset $\{ (A, A^*): A \in \mathcal M_n(\hat
{\mathcal A})\}$:
\begin{itemize}\item We get the symmetrized version of the classical Lie algebra
$B_n$ when $(*)$ is the transpose and $n$ is odd.
\item We get the symmetrized version of the classical Lie algebra
$C_n$ when $(*)$ is the symplectic involution and $n$ is even.
\item We get the symmetrized version of the classical Lie algebra
$D_n$ when $(*)$ is the transpose and $n$ is even.
\end{itemize}
 \end{rem}

\subsubsection{Lie
$\preceq$-super-semialgebras}$ $

Let us superize the Lie theory by means of
Definition~\ref{supergr1}.

\begin{defn} \label{Liess} A \textbf{Lie $\preceq$-super-semialgebra  with a negation map} is a module $L$ with a negation map $(-)$,
 endowed with super-$(-)$-anticommutative multiplication $ L \times L \to  L,$ written
as $(b,b') \mapsto [bb']_{\operatorname{s}}$,
 called a  $\preceq$-\textbf{super Lie bracket},
satisfying  $ [[   b  b'] _{\operatorname{s}}\, v]_{\operatorname{s}}
\preceq [ b[b'v]_{\operatorname{s}}\, ]_{\operatorname{s}}\, (-)[
b'[bv]_{\operatorname{s}}\, ]_{\operatorname{s}}$ for all homogeneous
$b,b',v \in L$.
\end{defn}

Thus, for all $b,b',v\in L$ we have  $ [[   b b']
_{\operatorname{s}}\, v]_{\operatorname{s}} \preceq [
b[b'v]_{\operatorname{s}}]_{\operatorname{s}}(-)[
b'[bv]_{\operatorname{s}}]_{\operatorname{s}}$. (The negations all
appear in the same degree in the super-version, so cancel out.)

%

  \begin{prop}\label{Pois1s} Any associative
  semialgebra $R$ with negation map becomes a Lie $\preceq$-super-semialgebra
  under the $\preceq$-super-Lie bracket \begin{equation}\label{supercom} [b_i b_j]_{\operatorname{s}} =  b_i b_j
  (-)^{ij} b_j b_i,\quad b_i \in R_i, \ b_j \in R_j.\end{equation}
\end{prop} \begin{proof} Reread the Leibniz identities (Lemma~\ref{Pois0}) in terms of \eqref{supercom}.
\end{proof}

\subsection{Poisson semialgebras and their module congruences}\label{Pois}$ $

The Leibniz  $\preceq$-identities} of Lemma~\ref{Pois0} motivate the
next notion.

 \begin{defn} A \textbf{Poisson $\preceq$-semialgebra}
 is an associative semialgebra $\mathcal A$ with a negation map, together with a
bilinear operation $\{ \phantom{a}, \phantom{a}\}: \mathcal A \times \mathcal A \to
\mathcal A$, called a {\bf Poisson bracket}, satisfying 
$$\qquad\quad \{ab,c\} \preceq  a\{b,c\}+\{a,c\}b   , \quad  \{a, bc\} \preceq \{a,b\}c
+ b\{a,c\} ,\quad \forall a,b,c \in \mathcal
A.$$\end{defn}

(This takes into account Definition~\ref{Lies}, as well as
Proposition~\ref{Pois1}.) Then $\{ \phantom{a}, \phantom{a}\}$
yields a Lie $\preceq$-structure as in Proposition~\ref{Pois1}.

\begin{example}  The following  are commutative   Poisson $\preceq$-semialgebras.
In each case we get a triple, where $\tT$ is taken to be the set of
monomials.
\begin{enumerate}\eroman  \item  If $L$ is a
f.d.~Lie semialgebra with negation map, having base $a_1, \dots,
a_n$, then, viewing the
 $a_i$ as commuting indeterminates in
the commutative polynomial semialgebra $R= F[a_1, \dots, a_n],$
introduce a Poisson bracket on $R$ by defining $\{a _i, a_j \}$ to
be the Lie product in $L$ and extending the Poisson bracket via the
 Leibniz identities, i.e.,
$$\qquad\quad  \{ab,c\} = a\{b,c\}+\{a,c\}b , \quad \{a, bc\} =\{a,b\}c
+ b\{a,c\}   ,\quad \forall a,b,c \in \mathcal A.$$

\item Suppose $V$ is a f.d.~vector space with an alternating
bilinear form (in the sense that $\langle v, v \rangle \succeq
\zero)$. Take a base $\{x _1, \dots, x_n\}$ of~$V$. The polynomial
semialgebra $F[x_1, \dots, x_n]$ becomes a Poisson
$\preceq$-semialgebra, where one defines $\{ x_i, x_j\}$ to
be~$\langle x_i, x_j\rangle.$



%
\end{enumerate}
\end{example}

The super-version is obtained by taking instead
Definition~\ref{Liess} and Proposition~\ref{Pois1s}.

\section{Areas for further research}\label{ques}

We have concentrated on the system as an algebraic structure. This
leads to the following questions:

\begin{enumerate}\eroman
  \item What is the complete classification of strictly negated systems for which $e'
  = \one?$ (See Proposition~\ref{char77} and Theorem~\ref{mon11}   for
  motivation.)

\item What systems are hypersystems with respect to some hyperfield? This is answered in
\cite {AGR0}.

 \item What systems other than hypersystems are of hypergroup type? Which satisfy the properties
 given in \S\ref{stneg}, namely  $\tT$-strictly negated,
$\tT$-reversible systems?  One might want to throw in idempotence,
which implies that $(-)$ is of the second kind and $e' = e.$

\item What can be proved in   linear algebra? (See \cite {AGR} for some results.)

\item What is the theory of affine and projective geometry over
systems, starting with Remark~\ref{geom1})?

\item What is the theory of matroids and valuated  matroids over
systems? This has been started in \cite{AGR0} and \cite{GaR}.

 \item How does the representation theory of systems fit in with
  \cite{CC2}? This is begun in \cite{JMR1}.

\item How far can one develop Lie structure theory along systemic lines,
starting with Lie's theorem and Engel's theorem? (The negation map
$(-)$ could be either the identity or the switch map on the
symmetrized algebra.)

\item How does one develop Hopf systems? (Hopf semialgebras are in the literature.)

\item What is the geometry of systems?
\end{enumerate}

\section{*Appendix A: Major examples of hypergroups and hyperfields
 and their power sets}\label{hype}$ $

We bring   the major examples of hyperfields, cf.~\cite{Bak} into
the systemic setting.

%
%
%
%
%
%

\begin{example}\label{Basicexamples}$ $
The first four of these examples correspond to $(-)$-bipotent systems (each
of height 2), but
the last two do not. 
\begin{itemize}

\item \textbf{The supertropical hyperfield.}
Define  $\R _\infty = \R \cup \{ - \infty\}$ and define the  product
$a \bigodot b := a + b$ and
$$a \boxplus b=
\begin{cases} max(a, b)\text{ if } a \ne b,\\   \{ c : c \le a\}
\text{ if } a = b.
\end{cases}$$
 Thus 0 is the multiplicative identity, $-
\infty$ is the additive identity, and we have a hyperfield
 called the \textbf{tropical hyperfield}, a special case of Proposition~\ref{closed}. This is
easily seen to be isomorphic (as hyperfields) to the supertropical
  semiring of Definition~\ref{super1}, which goes back to
Izhakian's \textbf{extended tropical arithmetic},  identifying
$(-\infty,a] : =\{ c : c \le a\}$ with $a^\nu$. We have a natural
  semiring system isomorphism with the sub-semiring $\widetilde{\R
_\infty}   $ of $\mathcal P(\R _\infty)$, sending the tangible
elements to $\R _\infty$, because

$$(-\infty,a] + b = \begin{cases} b \quad \text{if} \quad  b >a;\\ (-\infty,a] \quad \text{if} \quad  b =a; \\  (-\infty,b] \cup (b,a] = (-\infty,a] \quad \text{if} \quad  b
<a.
\end{cases} $$

\item  \textbf{The Krasner hyperfield.} Let $K = \{ 0; 1 \}$  with the usual
operations of Boolean algebra, except that now $ 1 \boxplus  1 = \{
0; 1 \} .$  Again, this generates a sub-semiring of $\mathcal P (K)$,
having three elements, which is isomorphic to the supertropical semiring system of
the monoid~$K$ with tangible elements $0$ and~$1$, where we identify $\{ 0; 1 \} $ with $ 1^\nu.$

\item \textbf{The hyperfield of
signs.} Let $ \tT := \{ 0, 1 , -1\}$  with the usual multiplication
law and hyperaddition defined by $1 \boxplus  1 = \{ 1\} ,$ $-1
\boxplus  -1 = \{ -1\} ,$ $ x \boxplus  0 = 0 \boxplus  x = \{ x\}
,$ and $1 \boxplus  -1 = -1 \boxplus  1 = \{ 0, 1,-1\} = \tT.$ Then
$\tT$ is a hyperfield called the  hyperfield of signs.

\item Valuative hyperfields (\cite[Example 2.12]{Bak}) also are
systemically isomorphic to the extended semiring in the sense of
\cite{IzhakianRowen2007SuperTropical}, in the same way.

%
%
%
%
%
%
%
%
%
%
%
%
%
%
%
%
%

  As noted in \cite[Example~6.9]{GJL}, the
four elements $ \{\{0\}, \{-1\}, \{1\}, \tT\}$ constitute the
sub-\semiring0 $\widetilde \tT$ of $\mathcal P(\tT)$,  comprising a
metatangible system, as seen in
Example~\ref{signsys}.

\item The phase hyperfield. Let $S^1$ denote the complex
unit circle, together with the center $\{ 0 \}$, and take $\tT = S^1
.$ Points $a$ and $b$ are \textbf{antipodes} if $a = -b.$
Multiplication is defined as usual (so corresponds on $S^1$ to
addition of angles). We call an arc from $a$ to $b$ of less than 180
degrees \textbf{short}, and denote it as $\overline{(a,b)}$. The
hypersum is given by
$$a \boxplus b=
\begin{cases} \overline{(a,b)} \text{ if } a \ne b;\\   \{ -a,0,a \} \text{
if } a = -b \ne 0 ;\\   \{  a \} \text{ if }   b = 0 .
\end{cases}$$  Then $\tTz$ is a hyperfield  called the \textbf{phase
hyperfield}.

 At
the power set level, given $T_1,T_2 \subseteq S^1$, one of which
containing at least two points, we define $T_1 \boxplus T_2$ to be
the union of all (short) arcs from a point of $T_1$ to a
non-antipodal point in~$T_2$ (which together make a connected arc),
together with $ \{ \zero \}$ if $T_2$ contains an antipode of $T_1$.
Thus the system  $\mathcal A$ spanned by $\tT$ is not metatangible;
the sum of two distinct points of~$\tT$ is never in~$\tT$, so this
is as far from metatangible as one can get. Its negation map is of
the second kind.

Its elements can be described as follows:

\begin{enumerate}  \eroman
\item $\{ \zero\}$, which   has height 0,
    \item  $\tT$, the points on $S^1$, each of which  has height 1,
    \item  Short arcs (the sum of non-antipodal distinct points), which  have height
    2,
\item The sets $\{ a, \zero, -a\} = a -a,$ which we write as
$a^\circ$,  which  have height
    2,
\item Semicircles with $\zero$ adjoined, having the form $a^\circ +
b$ where $b \ne \pm a$,  which  have height
    3 (which go clockwise or counter-clockwise depending on the relation from $b$ to $a$),
\item $S^1 \cup \{ \zero \} = a^\circ +  b^\circ$
where $b \ne \pm a$. This also can be written as the sum of three
points of an equilateral triangle on $S^1,$ i.e., at  angles of
$\frac{2\pi}3,$ so has height 3.
\end{enumerate}

The additive structure is described as follows in terms of the
hyperfield $\tT$:
\begin{enumerate}\eroman  \item A short arc $T$ plus a point $c$ with $-c \notin T$ is
 a short arc.

\item A short arc $T= \overline{(a,b)}$ plus a point $c$ with $-c\in  \{a, b\}$ is
 a semicircle with $  \{ \zero \}$ adjoined.

\item A short arc $T = \overline{(a,b)}$ plus a point $c $ with $-c \in T \setminus \{a, b\}$ is
$S^1 \cup \{ \zero \}.$

\item  $a^\circ +  b + c$ is either $a^\circ $
(for $b,c \in \{ \pm a \}$), a semicircle with $  \{ \zero \}$
adjoined (if $b,c$ are on the same side of $\pm a$), or $S^1 \cup \{
\zero \}$ if $b,c$ are on different
  sides of $\pm a$).
\end{enumerate}

It follows that any finite sum of elements of $\tTz$ is one of the
sets given above, so these comprise a system~$(\mathcal A,\tT,(-),
\subseteq)$. On the other hand, any proper arc of $S^1 $ can be
obtained as the product of two short arcs and a point. Hence,
$\mathcal A$ is not closed under
  multiplication!

  $S^1 \cup \{ \zero \}$ itself is obtained as the sum of
three  points (say with each 120 degrees apart). Thus the system
$\mathcal A$ has height $ 3$; the elements of  height 3 are
precisely the semicircles with the origin, and $S^1 \cup \{ \zero
\}$.


Distributivity fails  since certain arcs cannot be
obtained as unions of arcs. For example, take $a_1$ and $a_2$ almost
to be antipodes, $b_1 = a_2,$ and the arc connecting $b_1$ and $b_2$
just passes the antipode of  $a_1$; then  $(a_1 \boxplus a_2)(b_1
\boxplus b_2)$ is the arc from $a_1$ to~$b_2,$ a little more than a
semicircle, whereas $a_1 b_1  \boxplus a_1 b_2  \boxplus a_2 b_2$ is
already all of $S^1.$ But this can be remedied by defining
multiplication instead to be the convex union of hulls of the points
on the arcs. (This can be viewed as a special case of
Theorem~\ref{distres}.) It is easy to check that
Lemma~\ref{precedeq008} and the proof of Proposition~\ref{uniq671} are
applicable, since set inclusion is antisymmetric, so $(\mathcal A,
\tT, (-),\subseteq)$ remains a hypersystem for either choice of
multiplication.
 Viro~\cite{Vi} also presents another version.

\item The ``triangle'' hyperfield  $\tT$  defined over $\R ^+$ by the
 formula $$a \boxplus b = \{c \in \R^+ :
|a - b| \le c \le a + b\}.$$ 
 Again distributivity fails for $\mathcal A$.  Here $$\tT +
\tT = \{
 [a_1,a_2]: a_1 \le a_2 \},$$ so $\mathcal A$ obviously is not metatangible.
The negation map is of the first kind, satisfying $$a^\circ= a
\boxplus a = \{c \in \R^+ : 0 \le c \le 2a  \}.$$ It $\mathcal A$
has height 2 over $\tT$ . Indeed, $[a_1,a_2] = \frac {a_1+a_2}2 +
\frac {a_2-a_1}2 \in \hat A $, whereas $[a_1,a_2] + [a_1',a_2']$ is
some interval going up to $a_2+a_2'$.


\end{itemize}
\end{example}

\section{**Appendix B: Fuzzy rings as systems}\label{fuzz}$ $

Another concept which turns out to provide systems was introduced in
1986 and refined in 2011 by Dress \cite{Dr}, and Dress and
Wenzel~\cite{DW}. This treatment also is inspired by \cite{GJL}. Let
$\mathcal A := (\mathcal A, +, \cdot, \zero, \one )$ be a cbimagma,
for which  $(\mathcal A,   \cdot,   \one )$ is a commutative monoid.
$\mathcal A^\times$ denote the set of invertible elements of
$\mathcal A.$

\begin{defn}[{\cite[Definitions~2.1, 2.8]{Dr},
\cite[Definition~2.14]{GJL}}]\label{fuzzy1} $\mathcal A$ is a
\textbf{fuzzy ring} if it is an $\mathcal A^\times$-module and has a
distinguished element $\vep$ and a proper  ideal $\mathcal A_0$
satisfying the following axioms for $a,a_i\in \mathcal A$:
\begin{enumerate}\label{distr157}\eroman
\item $\vep^2 = \one;$
\item  $a = \vep $, iff $a\in \mathcal A^\times$  with  $\one + a \in \mathcal A_0$;
\item If   $a_1 + a_2,\ a_3 + a_4 \in \mathcal A_0,$ then
$a_1 a_3 + \vep a_2 a_4 \in \mathcal A_0;$
\item If   $a_1 +
a_2( a_3 + a_4 ) \in \mathcal A_0,$ then $a_1 +  a_2  a_3  +  a_2
a_4 \in \mathcal A_0.$

\end{enumerate}

The fuzzy ring is \textbf{coherent} if $\mathcal A^\times$
spans $(\mathcal A,+)$.

\end{defn}

Note that (iii) is the fuzzy property of
Definition~\ref{precedeq59}.

In line with the systemic approach, it is natural to generalize the
definition slightly, replacing $\mathcal A^\times$ by a monoid
$\tT$. On the other hand, conditions (iii) and (iv) do not initially
enter into our proofs (and also did not enter into the proof of
\cite[Theorem~3.3]{GJL}). This motivates us
 to suppress them at the outset, to get a more straightforward structure theory
 using triples.

\begin{defn} \label{fuzzy11}  A \textbf{pre-fuzzy
triple} is a cancelative cbimagma triple $(\mathcal A, \tT, (-))$
where $\tT$ is a multiplicative submonoid of $\mathcal A$, together
with a distinguished element $\vep \in \tT$ and a proper
 ideal $\mathcal A_0$ satisfying the following
axioms:
\begin{enumerate}\label{distr158}\eroman
\item $\vep^2 = \one;$
\item For any  $a_i\in \tT$, $a_1 + a_2 \in \mathcal A_0$ iff $a_1 = \vep a_2;$
\item  $(-)b := \vep b$.\end{enumerate}

 The pre-fuzzy triple is
$\tT$-\textbf{coherent} if $ (\mathcal A,+) = \{
\sum_{\text{finite}} a_i : a_i \in \tT \}$.
\end{defn}

\begin{rem} Condition (ii) implies $\mathcal A^\circ \subseteq \mathcal A_0$.
 The fuzzy
 ring definition takes $\tT = A^\times$.\end{rem}

%

\begin{rem}\label{drive4} $ $ \begin{enumerate}\eroman  \item
In a pre-fuzzy triple $\mathcal A$, the sub-$\tT$-cbimagma generated
by $\tT$ and $\mathcal A_0$ is also pre-fuzzy, so we  assume from
now on that it equals ~$\mathcal A$.
 \item Condition (iii) of Definition~\ref{fuzzy11} matches Definition~\ref{fuzzy1}(ii) for $a_1
 \in \mathcal A^\times.$
 \item Adjusting multiplication according to in Theorem~\ref{distres} always enables
 us to dispose of Condition (iv) of Definition~\ref{fuzzy1}.
\end{enumerate}
\end{rem}

\begin{defn}\label{fuzzy121} A pre-fuzzy
triple  is \textbf{$\tT$-strictly negated} if

\begin{equation}\label{eeq22}
a+b \in \mathcal A_0, \ a \in \tT, \quad\text{ implies }\quad b =
\vep a +c \text{ for some }c\in  \mathcal A_0 .
\end{equation}
\end{defn}

The next result reconciles fuzzy rings with
Definition~\ref{fuzzy11}.

\begin{lem}\label{nummatch00} Assume that $(\mathcal A   ,  \tT , (-))$ is a pre-fuzzy
triple with $\mathcal A = \tT + \mathcal A_0 $, and put $\vep =
(-)\one$.
 \begin{enumerate}\eroman
 \item Condition (iii) of Definition~\ref{fuzzy1} holds whenever $a_1 , a_2 \in \tT $.
 \item Condition (iii) of Definition~\ref{fuzzy1} holds whenever $\mathcal A $ is $\tT$-strictly negated.
\end{enumerate}
 \end{lem}
\begin{proof}
(i)  $a_1  +a_2 \in \mathcal A_0$, so $a_1  = \vep a_2$ by
Definition~\ref{fuzzy11}(ii). Hence  $a_1 a_3 + \vep a_2 a_4  = a_1(
a_3 + a_4) \in \mathcal A_0.$

(ii)  The assertion is obvious if $a_1, a_2 \in \mathcal A_0,$ so we
may assume that $a_1 \notin \mathcal A_0,$ i.e., $a_1 \in \tT.$ But
we are given $a_1  +a_2 \in \mathcal A_0$, so,  by
Definition~\ref{fuzzy121}, $a_1  + c = \vep a_2$ for some $c \in
\mathcal A_0.$  Hence $a_1 a_3 + \vep a_2 a_4  = a_1( a_3 + a_4) + c
a_4 \in \mathcal A_0.$
\end{proof}

%
%
%
%

\begin{prop}\label{fuzpro} Any  strictly negated system
with respect to $\preceq_\circ$ satisfies the fuzzy property of
Definition~\ref{precedeq59}. \end{prop}
\begin{proof}
This is clear if $b_i,b_i' \in A^\circ$ for either $i$, so by
\eqref{eeq2} we have either $b_i \preceq b_i'$ or $b_i' \preceq
b_i$, so we may assume that $b_i \preceq b_i'$ for $i = 1,2$. (We
can replace $b_i,b_i'$ by $(-)b_i,(-)b_i'$ if necessary.) But now
writing $b_i' = b_i + c_i^\circ$ we have
$$b_1b_2 (-) b_1'b_2' = b_1b_2 (-) b_1b_2 (-)c_1^\circ b_2 (-)c_2^\circ b_1 \in \mathcal A^\circ.$$
\end{proof}

%
%

 In a metatangible pre-fuzzy
triple $\mathcal A$, we can also replace $\mathcal A$ by
$\left\{\sum_{\text{finite}} a_i : a_i \in \tT \right\},$ in which
case $\mathcal A$ becomes $\tT$-coherent.  Condition (iv) of
Definition~\ref{fuzzy1} then  becomes superfluous, in view of
\cite[Theorem~7.34]{Row16v}. In short, pre-fuzzy triples often are
fuzzy rings.

\subsection{Fuzzy rings versus pre-fuzzy triples  and  systems}$ $


\begin{thm}\label{drive6} A $\tT$-coherent pre-fuzzy triple $\mathcal A$ gives rise to a
system $(\mathcal A, \tT, (-),\preceq_{\mathcal A_0}),$  cf.~ Proposition~\ref{uniq671}(iii), where
$(-)a = \vep a.$
\end{thm}
\begin{proof}  The map $a
\mapsto \vep a$ obviously is a negation map. Furthermore, if $a (-)
b \in \mathcal A^\circ$ for $a,b \in \tT,$ then $\one + \vep b
a^{-1} \in \mathcal A^\circ,$ implying $ \vep b a^{-1} = \vep,$ and
thus $b = a.$

 $a (-)a = a + \vep a \in \mathcal A_0$. For $\preceq$  to be a surpassing relation, we need to verify the
conditions of Definition~\ref{precedeq07}. Conditions (i)--(iv) are
clear; for (v), suppose $ a_0 \preceq a_1$ for $a_i\in \tT$. Then
$a_1 = a_0 +c$ for $c\in \mathcal A_0$ implies $a_1(-) a_0 =
a_0^\circ +c\in \mathcal A_0$, and thus  $a_1= a_0.$

Since $\tT \cap A^\circ = \emptyset$, this is a system, by
Lemma~\ref{precedeq008} and Proposition~\ref{uniq671}(iii).
\end{proof}

%
%

 \begin{MNote} We are back to the
 definition of system, where   $\mathcal A_{\Null}
 = \mathcal A_0.$ Thus systems provide a straightforward way of
 viewing fuzzy rings, and their theory includes that of fuzzy rings.
 \end{MNote}

 Conversely to Theorem~\ref{drive6},
 the notion of pre-fuzzy triple also encompasses cancelative triples.

\begin{prop}\label{drive3} Suppose that $\mathcal A := (\mathcal A, \tT, (-))$ is a cancelative
cbimagma triple with  unique quasi-negatives. Then $\mathcal A$
gives rise to a pre-fuzzy triple $\mathcal A'$ with the same
operations, where $\mathcal A_0 = \mathcal A^\circ$ and $(\mathcal
A',+)$ is generated by $\tT$ and $\mathcal A_0 $, and $\vep =
(-)\one$.

\end{prop}
\begin{proof} Note that $(\mathcal A', \tT, (-), \preceq)$  is a triple, so we may assume that $\mathcal A' = \mathcal A.$
The conditions   of Definition~\ref{fuzzy11} are clear.
%
%
%
%
%

\end{proof}

\end{document}